\documentclass[a4paper]{amsart}

\usepackage{xcolor}
\usepackage[utf8]{inputenc}
\usepackage[T1]{fontenc}
\usepackage{lmodern}
\usepackage{amssymb,amsxtra}
\usepackage{nicefrac,mathtools}
\usepackage[lite]{amsrefs}
\newcommand*{\MRref}[2]{ \href{http://www.ams.org/mathscinet-getitem?mr=#1}{MR \textbf{#1}}}

\newcommand*{\arxiv}[1]{\href{http://www.arxiv.org/abs/#1}{arXiv: #1}}

\renewcommand{\PrintDOI}[1]{\href{http://dx.doi.org/\detokenize{#1}}{doi: \detokenize{#1}}}

\BibSpec{book}{%
  +{}  {\PrintPrimary}                {transition}
  +{,} { \textit}                     {title}
  +{.} { }                            {part}
  +{:} { \textit}                     {subtitle}
  +{,} { \PrintEdition}               {edition}
  +{}  { \PrintEditorsB}              {editor}
  +{,} { \PrintTranslatorsC}          {translator}
  +{,} { \PrintContributions}         {contribution}
  +{,} { }                            {series}
  +{,} { \voltext}                    {volume}
  +{,} { }                            {publisher}
  +{,} { }                            {organization}
  +{,} { }                            {address}
  +{,} { \PrintDateB}                 {date}
  +{,} { }                            {status}
  +{}  { \parenthesize}               {language}
  +{}  { \PrintTranslation}           {translation}
  +{;} { \PrintReprint}               {reprint}
  +{.} { }                            {note}
  +{.} {}                             {transition}
  +{} { \PrintDOI}                   {doi}
  +{} { available at \url}            {eprint}
  +{}  {\SentenceSpace \PrintReviews} {review}
}

\usepackage{enumitem} %<--advanced enumerate, which handles label and ref
\setlist[enumerate,1]{label=\textup{(\arabic*)}}% ensure enumerates in theorems are upright

\usepackage{tikz}
\usetikzlibrary{matrix,calc,arrows,decorations.pathmorphing}
\tikzset{node distance=2cm, auto}

\tikzset{cd/.style=matrix of math nodes,row sep=2em,column sep=2em, text height=1.5ex, text depth=0.5ex}
\tikzset{cdar/.style=->,auto}
\tikzset{mid/.style={anchor=mid}} % put labels on the arrow
\tikzset{narrowfill/.style={inner sep=1pt, fill=white}}% style for nodes with filled background

\usepackage{microtype}
\usepackage[pdftitle={Iterated crossed products for groupoid fibrations},
  pdfauthor={Alcides Buss, Ralf Meyer},
  pdfsubject={Mathematics}
]{hyperref}

\numberwithin{equation}{section}
\theoremstyle{plain}
\newtheorem{theorem}[equation]{Theorem}
\newtheorem{lemma}[equation]{Lemma}
\newtheorem{proposition}[equation]{Proposition}
\newtheorem{corollary}[equation]{Corollary}
\theoremstyle{definition}
\newtheorem{definition}[equation]{Definition}
\theoremstyle{remark}
\newtheorem{remark}[equation]{Remark}
\newtheorem{example}[equation]{Example}

\DeclareMathOperator{\Iso}{Iso}% isotropy
\DeclareMathOperator{\dom}{dom}% domains
\DeclareMathOperator{\Prim}{Prim}% primitive ideal space
\newcommand*{\nb}{\nobreakdash}
\newcommand*{\Star}{\(^*\)\nobreakdash-}

\newcommand*{\C}{\mathbb C}
\newcommand*{\Z}{\mathbb Z}
\newcommand*{\R}{\mathbb R}
\newcommand*{\N}{\mathbb N}

\newcommand*{\Torus}{\mathbb T}%torus
\newcommand*{\Bound}{\mathbb B}%adjointable operators on a Hilbert module
\newcommand*{\Banb}{\mathfrak B}% Fell bundle
\newcommand*{\A}{\mathfrak A}% other Fell bundle

\newcommand*{\Sect}{\mathfrak S}% space of quasi-continuous sections

\newcommand*{\Cst}{\textup C^*}%C*-algebra
\newcommand*{\Cont}{\textup C}%continuous functions
\newcommand*{\Mult}{\mathcal M}%multiplier algebra

\newcommand*{\LL}{\mathcal{L}}
\newcommand*{\Id}{\textup{Id}}%identity

\newcommand*{\Hils}{\mathcal H}%Hilbert space
\newcommand*{\Hilm}[1][H]{\mathcal #1}%Hilbert module

\newcommand*{\defeq}{\mathrel{\vcentcolon=}}% used for definitions
\newcommand*{\congto}{\xrightarrow\sim}

\newcommand*{\norm}[1]{\lVert#1\rVert}% norm
\newcommand*{\abs}[1]{\lvert#1\rvert}% absolute value
\newcommand*{\conj}[1]{\overline{#1}}% complex conjugation
\newcommand*{\braket}[2]{\langle#1{\mid}#2\rangle}% right inner products
\newcommand*{\bigbraket}[2]{\bigl\langle#1\bigm|#2\bigr\rangle}% big inner products

\newcommand*{\s}{s}% source map of groupoids
\newcommand*{\rg}{r}% range map of groupoids

\newcommand*{\FellH}{\mathfrak A}% Fell bundle over H derived from a Fell bundle over L
\newcommand*{\cstar}{\texorpdfstring{\(\mathrm{C}^*\)\nobreakdash-\hspace{0pt}}{*-}}
\newcommand*{\into}{\hookrightarrow}
\newcommand*{\onto}{\twoheadrightarrow}

\newcommand*{\pr}{\mathrm{pr}}% projection maps

\newcommand*{\dd}{\,\mathrm d}% integration

\begin{document}
\title[Iterated crossed products for groupoid fibrations]{Iterated crossed products\\for groupoid fibrations}

\author{Alcides Buss}
\email{alcides@mtm.ufsc.br}

\address{Departamento de Matem\'atica\\
 Universidade Federal de Santa Catarina\\
 88.040-900 Florian\'opolis-SC\\
 Brazil}

\author{Ralf Meyer}
\email{rmeyer2@uni-goettingen.de}

\address{Mathematisches Institut\\
 Georg-August-Universit\"at G\"ottingen\\
 Bunsenstra\ss e 3--5\\
 37073 G\"ottingen\\
 Germany}

\begin{abstract}
  We define and study fibrations of topological groupoids.  We
  interpret a groupoid fibration \(L\to H\) with fibre~\(G\) as an
  action of~\(H\) on~\(G\) by groupoid equivalences.  Our main result
  shows that a crossed product for an action of~\(L\) is isomorphic to
  an iterated crossed product first by~\(G\) and then by~\(H\).  Here
  ``groupoid action'' means a Fell bundle over the groupoid, and
  ``crossed product'' means the section \cstar{}algebra.
\end{abstract}
\thanks{Supported by CNPq/CsF (Brazil).}
\maketitle

\setcounter{tocdepth}{1}% show only sections, not subsections
\tableofcontents

\section{Introduction}
\label{sec:introduction}

What does it mean for a topological groupoid~\(H\) to act on another
topological groupoid~\(G\)?  The definition of an action by groupoid
isomorphisms is straightforward if complicated
(see~§\ref{sec:classical_action}).  Some examples, however, require
arrows in~\(H\) to act by equivalences, not isomorphisms.  If~\(H\) is
an étale groupoid, such actions by equivalences are defined
in~\cite{Buss-Meyer:Actions_groupoids} using the inverse semigroup of
bisections of~\(H\).  Here we extend this notion of action to
non-étale topological groupoids through the notion of a
\emph{fibration of topological groupoids}, briefly groupoid fibration.

The idea is the following.  An action of~\(H\) on~\(G\) should give a
transformation groupoid \(L\defeq G\rtimes H\) that contains~\(G\) and
comes with a continuous functor \(L\to H\).  Thus defining actions of
topological groupoids on topological groupoids amounts to
characterising which chains of continuous functors \(G\hookrightarrow
L\to H\) correspond to actions.  We require \(L\to H\) to be a
``groupoid fibration'' with ``fibre'' \(G\subseteq L\) (see
Definition~\ref{def:groupoid_fibration}).  This gives the same notion
of action as in~\cite{Buss-Meyer:Actions_groupoids} if~\(H\) is étale.

Groupoid fibrations are inspired by higher category theory.  The
thesis of Li Du~\cite{LiDu:Thesis} describes actions of
\(\infty\)\nb-groupoids on \(\infty\)\nb-groupoids through Kan
fibrations.  By definition, a groupoid fibration between two
topological groupoids is a Kan fibration between the associated
topological \(\infty\)\nb-groupoids.

Let \(F\colon L\to H\) be a groupoid fibration with fibre \(G\subseteq
L\).  Then we describe an induced action of~\(H\) on the
\(\Cst\)\nb-algebra of~\(G\), such that the crossed product
is~\(\Cst(L)\).  This generalises the well known decomposition
\(\Cst(X\rtimes H)\cong \Cont_0(X)\rtimes H\) for an action of a
groupoid~\(H\) on a space~\(X\).  In general, an ``action'' of a
locally compact groupoid on a \(\Cst\)\nb-algebra is a
(saturated) Fell bundle over the groupoid, and its ``crossed product''
is the section \(\Cst\)\nb-algebra of the Fell bundle.  Saturated Fell
bundles are interpreted as actions by Morita--Rieffel equivalences
in~\cite{Buss-Meyer-Zhu:Higher_twisted}.  We tacitly assume all Fell
bundles over groupoids to be saturated.

More generally, we decompose the ``crossed product'' for any
``action'' of~\(L\) in such a way.  That is, for a Fell bundle
over~\(L\), we construct a Fell bundle over~\(H\) with the restricted
section algebra over~\(G\) as unit fibre and show that the section
algebra of this Fell bundle over~\(H\) is the section algebra of the
original Fell bundle over~\(L\).  In brief, \((A\rtimes G)\rtimes H
\cong A\rtimes L\) for an action of~\(L\) on a \cstar{}algebra~\(A\).

If \(G\) and~\(H\) are groups, then a groupoid fibration \(L\to H\)
with fibre~\(G\) is nothing but an extension of topological groups
\(G\rightarrowtail L \onto H\).  Our result on iterated crossed
products is known in this case in the language of Green twisted
actions.

As in~\cite{Buss-Meyer:Actions_groupoids}, a motivating example for
our theory is to construct an action of a locally Hausdorff, but
non-Hausdorff groupoid~\(H\) on a \(\Cst\)\nb-algebra that represents
the action of~\(H\) on its arrow space~\(H^1\) by left multiplication.
This is the simplest example of a free and proper action.
Since~\(H^1\) is non-Hausdorff, the \(\Cst\)\nb-algebra that best
describes~\(H^1\) is the groupoid \(\Cst\)\nb-algebra of the \v{C}ech
groupoid of a Hausdorff, open covering of~\(H^1\).  There is usually
no classical action of~\(H\) by automorphisms on this \v{C}ech
groupoid.
There is, however, a groupoid fibration describing this
action, and an associated Fell bundle over~\(H\) describing the action
in \(\Cst\)\nb-algebraic terms.

Groupoid fibrations of plain groupoids without topology or other extra
structure are defined already by Ronald
Brown~\cite{Brown:Fibrations_groupoids}.  A definition for topological
groupoids is given
in~\cite{Deaconu-Kumjian-Ramazan:Fell_groupoid_morphism}; but the
definition in~\cite{Deaconu-Kumjian-Ramazan:Fell_groupoid_morphism} is
not used in~\cite{Deaconu-Kumjian-Ramazan:Fell_groupoid_morphism},
works only in the étale case, and contains an unnecessary extra
assumption.  Our construction is similar to the one
in~\cite{Deaconu-Kumjian-Ramazan:Fell_groupoid_morphism}, except that
we insist on getting \emph{saturated} Fell bundles and allow
non-étale, locally Hausdorff, locally compact groupoids (with a
Hausdorff object space and a Haar system).  We do not require
amenability assumptions as
in~\cite{Deaconu-Kumjian-Ramazan:Fell_groupoid_morphism} since we work
with full crossed products throughout.

Reduced crossed products for non-Hausdorff groupoids do not always
behave well for iterated crossed products.  Counterexamples
in~\cite{Buss-Exel-Meyer:Reduced} in the étale case show this.

Section~\ref{sec:groupoid_fibrations} defines groupoid fibrations and
illustrates the notion by some examples and basic properties.  In
particular, classical groupoid actions by automorphisms and groupoid
extensions give examples of groupoid fibrations.  Most of the general
theory works for arbitrary topological groupoids, even for groupoids
in a category with pretopology as in~\cite{Meyer-Zhu:Groupoids}.

Section~\ref{sec:fibrations_gradings} compares groupoid fibrations
with étale~\(H\) to the actions of~\(H\) defined
in~\cite{Buss-Meyer:Actions_groupoids} using inverse semigroups.
Section~\ref{sec:locally_Hausdorff_compact} shows that the
transformation groupoid~\(L\) inherits the properties of being
(locally) Hausdorff and locally compact from \(G\) and~\(H\).
Section~\ref{sec:Haar} describes how Haar systems on \(G\) and~\(H\)
induce a Haar system on~\(L\).  Section~\ref{sec:crossed} contains our
main result on iterated crossed products.
Section~\ref{sec:applications} constructs the Fell bundle describing
the translation action on the arrow space of a locally Hausdorff
groupoid and explains how certain results in
\cites{Brown-Goehle-Williams:Equivalence_iterated,
  Brown-Huef:Decomposing, Ionescu-Kumjian-Sims-Williams:Stabilization,
  Kaliszewski-Muhly-Quigg-Williams:Coactions_Fell,
  Rennie-Robertson-Sims:Groupoid_CP, Sims-Williams:Equivalence_reduced}
are contained in our main theorem.

\section{Groupoid fibrations}
\label{sec:groupoid_fibrations}

A topological groupoid consists of two topological spaces \(G^1\)
and~\(G^0\) with \emph{open}, continuous range and source maps
\(\rg,\s\colon G^1\rightrightarrows G^0\) and continuous
multiplication, inversion and unit maps satisfying the usual algebraic
conditions.  The range and source maps are automatically open if the
groupoid has a Haar system (see~\cite{Renault:Groupoid_Cstar}).  Until
we consider Haar systems and groupoid \(\Cst\)\nb-algebras, we allow
\emph{arbitrary} topological spaces \(G^0\) and~\(G^1\) as
in~\cite{Buss-Meyer:Actions_groupoids}.  We need neither Hausdorffness
nor local compactness, but we do need open range and source maps.

\begin{definition}
  \label{def:groupoid_fibration}
  Let \(L\) and~\(H\) be topological groupoids.  A \emph{groupoid
    fibration} is a continuous functor \(F\colon L\to H\)
  (continuous maps \(F^i\colon L^i\to H^i\) for \(i=0,1\) that
  intertwine \(\rg\), \(\s\) and the multiplication maps), such that
  the map
  \begin{equation}
    \label{eq:groupoid_fibration}
    (F^1,\s)\colon L^1\to H^1\times_{\s,H^0,F^0} L^0
    \defeq \{(h,x)\in H^1\times L^0 \mid \s(h)=F^0(x)\}
  \end{equation}
  is an open surjection.  Its \emph{fibre} is the subgroupoid~\(G\)
  of~\(L\) defined by \(G^0=L^0\) and
  \[
  G^1 = \{ g\in L^1 \mid F^1(g) = 1_{F^0(\s(g))}\},
  \]
  equipped with the subspace topology on \(G^1\subseteq L^1\).

  A \emph{groupoid covering} is a functor \(F\colon L\to H\) for
  which~\eqref{eq:groupoid_fibration} is a homeomorphism.
\end{definition}

\begin{lemma}
  \label{lem:fibre_top_groupoid}
  The fibre of a groupoid fibration is a topological groupoid.
\end{lemma}

\begin{proof}
  If \(g\in G^1\), then \(F^1(g)=1_{F^0(\s(g))}\), so \(F^0(\rg(g)) =
  \rg(F^1(g)) = \s(F^1(g)) = F^0(\s(g))\).  Thus \(g^{-1}\in G^1\) as
  well.  We also get \(g_1\cdot g_2\in G^1\) if \(g_1,g_2\in G^1\), so
  that~\(G\) is a subgroupoid of~\(L\).  It remains to prove that the
  source map (and hence the range map) in~\(G\) is open.  We use that
  the pull-back of an open surjection is again an open surjection
  (see~\cite{Meyer-Zhu:Groupoids}).  There is a fibre product diagram
  \[
  \begin{tikzpicture}[baseline=(current bounding box.west)]
    \matrix (m) [cd,column sep=4em] {
      G^1 & G^0=L^0 \\
      L^1 & H^1\times_{\s,H^0,F^0} L^0,\\
    };
    \draw[cdar,->>] (m-1-1) -- node {\(\s\)} (m-1-2);
    \draw[right hook->] (m-1-1) -- (m-2-1);
    \draw[cdar,->>] (m-2-1) -- node {\((F^1,\s)\)} (m-2-2);
    \draw[right hook->] (m-1-2) -- node {\((u\circ F^0,\Id)\)} (m-2-2);
  \end{tikzpicture}
  \]
  where \(u\colon H^0\into H^1\) denotes the unit map.  Since
  \((F^1,\s)\) is an open surjection by assumption, so is \(\s\colon
  G^1\onto G^0\).
\end{proof}

\begin{remark}
  The inversion map turns~\eqref{eq:groupoid_fibration} into the map
  \[
  (F^1,\rg)\colon L^1\to H^1\times_{\rg,H^0,F^0} L^0.
  \]
  Hence \(F\colon L\to H\) is a fibration if and only if~\((F^1,\rg)\)
  is an open surjection.
\end{remark}

\begin{proposition}
  \label{pro:fibration_basic_action_GL}
  Let \(F\colon L\to H\)
  be a groupoid fibration with fibre~\(G\).
  The left multiplication action of~\(G\)
  on~\(L^1\)
  with the map \((F^1,s)\colon L^1 \onto H^1 \times_{\s,H^0,F^0} L^0\)
  in~\eqref{eq:groupoid_fibration} as bundle projection is a principal
  \(G\)\nb-bundle,
  that is, the bundle projection is an open surjection and the
  following map is a homeomorphism:
  \begin{equation}
    \label{eq:fibration_principal_bundle}
    G^1 \times_{\s,L^0,\rg} L^1 \congto
    L^1\times_{(F^1,\s),H^1\times_{\s,H^0,F^0} L^0,(F^1,\s)} L^1,\qquad
    (g,l)\mapsto (g\cdot l,l).
  \end{equation}
\end{proposition}

Actions that are part of principal bundles are called \emph{basic} in
\cites{Buss-Meyer:Actions_groupoids, Meyer-Zhu:Groupoids}.

\begin{proof}
  That~\((F^1,s)\) is an open surjection is exactly the assumption
  of~\(F\) being a groupoid fibration.  The
  map~\eqref{eq:fibration_principal_bundle} is well defined because
  \(F^1(gl) = F^1(l)\) and \(\s(gl)=\s(l)\) for all \(g\in G^1\),
  \(l\in L^1\) with \(\s(g)=\rg(l)\).  We claim that
  \((l_1,l_2)\mapsto (l_1\cdot l_2^{-1},l_2)\) is a well defined map
  \(L^1\times_{(F^1,\s),H^1\times_{\s,H^0,F^0} L^0,(F^1,\s)} L^1 \to
  G^1 \times_{\s,L^0,\rg} L^1\).  Let \((l_1,l_2)\in L^1\times L^1\)
  satisfy \((F^1,\s)(l_1)= (F^1,\s)(l_2)\).  Since
  \(\s(l_1)=\s(l_2)\), \(g\defeq l_1l_2^{-1}\) is well defined.  Since
  \(F^1(l_1)=F^1(l_2)\) and~\(F^1\) is multiplicative,
  \(F^1(g)=1_{F^0(\s(g))}\), so \(g\in G^1\).  And \(\s(g)=\rg(l_2)\) by
  construction, so \((g,l_2)\in G^1 \times_{\s,L^0,\rg} L^1\).  The
  map \((l_1,l_2)\mapsto (l_1l_2^{-1},l_2)\) is continuous and a
  two-sided inverse for~\eqref{eq:fibration_principal_bundle}.
\end{proof}

\begin{remark}
  \label{rem:def_pretopology}
  The definitions of a groupoid fibration and covering carry over to
  the setting of groupoids in categories with pretopology studied
  in~\cite{Meyer-Zhu:Groupoids}.  A pretopology specifies a class of
  special morphisms in the category called ``covers,'' subject to some
  axioms.  In particular, pull-backs of covers are again covers.  In
  this article, we use the category of topological spaces with open
  surjections as covers.

  For groupoids in a category with pretopology, a groupoid fibration
  is defined as a functor where~\eqref{eq:groupoid_fibration} is a
  \emph{cover}, and a groupoid covering as a functor
  where~\eqref{eq:groupoid_fibration} is an isomorphism.  The proof of
  Lemma~\ref{lem:fibre_top_groupoid} still works because it only uses
  that pull-backs of covers remain covers.  Similarly, the proof of
  Proposition~\ref{pro:fibration_basic_action_GL} still works for
  groupoids in any category with pretopology.  Most definitions,
  results and examples in this section generalise to groupoids in
  categories with pretopology.  Namely, this is the case for
  Example~\ref{exa:act_space},
  Proposition~\ref{pro:groupoid_covering},
  Definition~\ref{def:classical_action},
  Lemma~\ref{lem:classical_action_trafo},
  Proposition~\ref{pro:classical_action_fibration},
  Example~\ref{exa:bitransformation_groupoid},
  Example~\ref{exa:hypercovers}, Lemma~\ref{lem:fibration_to_space},
  Proposition~\ref{pro:compose_fibrations},
  Example~\ref{exa:act_arrows1},
  Lemma~\ref{lem:action_space_trafo-groupoid},
  Definition~\ref{def:groupoid_extension}, and
  Lemma~\ref{lem:fibration_iso_objects}.
  Lemma~\ref{lem:F1_F0_open_surjective} carries over if
  \cite{Meyer-Zhu:Groupoids}*{Assumption~2.9} holds for our
  pretopology.
\end{remark}

\begin{remark}
  \label{rem:Lie_groupoid_fibration}
  Lie groupoids are groupoids in the category of smooth manifolds with
  surjective submersions as covers.  The definition of a Lie groupoid
  fibration in Remark~\ref{rem:def_pretopology} is weaker than the
  usual one in \cites{Higgins-Machenzie:Fibrations,
    Mackenzie:General_Lie_groupoid_algebroid}, which also asks for the
  map \(F^0\colon L^0\to H^0\) on objects to be a cover.  This is
  reasonable if one wants~\(H\) to be determined by \(L\) and~\(G\) as
  a generalised quotient (see
  \cite{Mackenzie:General_Lie_groupoid_algebroid}*{Theorems 2.4.6 and
    2.4.8}).  But it rules out important examples.  For a groupoid
  covering, requiring~\(F^0\) to be a cover restricts to actions with
  a cover as anchor map.  An important counterexample is the action
  of~\(H\) by right translations on \(H^x\defeq \{h\in H\mid
  \rg(h)=x\}\) for \(x\in H^0\).
\end{remark}

\subsection{Groupoid coverings and actions on spaces}
\label{sec:groupoid_coverings}

\begin{example}
  \label{exa:act_space}
  Let~\(X\) be a topological space with an action of a topological
  groupoid~\(H\).  View~\(X\) as a groupoid with only identity arrows.
  The functor from the transformation groupoid \(H\ltimes X\) to~\(H\)
  which is the anchor map on objects and the map \((h,x)\mapsto h\) on
  arrows is a groupoid covering with fibre~\(X\).  The
  map~\eqref{eq:groupoid_fibration} is an isomorphism by the
  definition of~\(H\ltimes X\): \((H\ltimes X)^1 \defeq \{(h,x)\in
  H^1\times X \mid \s(h)=\rg(x)\}\).
\end{example}

\begin{proposition}
  \label{pro:groupoid_covering}
  Any groupoid covering is isomorphic to \(H\ltimes X\to H\) for some
  \(H\)\nb-action~\(X\).  A groupoid fibration is a groupoid covering
  if and only if its fibre is a groupoid with only identity arrows,
  that is, a space viewed as a groupoid.
\end{proposition}

This is the main result
in~\cite{Brown-Danesh-Hardy:Topological_groupoidsII}.  We sketch the
proof to show how the argument carries over to groupoids in categories
with pretopology as in Remark~\ref{rem:def_pretopology}.

\begin{proof}
  First let \(F\colon L\to H\) be a groupoid covering.  Its fibre
  is~\(L^0\) viewed as a groupoid with only identity arrows.  We
  construct an \(H\)\nb-action on~\(X\defeq L^0\).  Its anchor map
  is~\(F^0\).  Let \(\tau\colon H^1\times_{\s,H^0,F^0} L^0 \congto
  L^1\) be the inverse of the isomorphism
  in~\eqref{eq:groupoid_fibration}.  We define the \(H\)\nb-action to
  be \(\rg\circ\tau\colon H^1\times_{\s,H^0,F^0} L^0 \to L^0\); this
  is indeed a groupoid action.  The identity on objects and~\(\tau\)
  on arrows give a groupoid isomorphism \(H\ltimes L^0\congto L\).
  This is the only isomorphism for which \(F\colon L\to H\) becomes
  the canonical functor \(H\ltimes L^0\to H\).

  Now let \(F\colon L\to H\) be a groupoid fibration whose fibre~\(G\) is
  the space \(X=L^0\) of unit arrows.  The map~\eqref{eq:groupoid_fibration} is
  a bundle projection for a principal \(G\)\nb-bundle by
  Proposition~\ref{pro:fibration_basic_action_GL}, and~\(G\) has only
  identity arrows by assumption.  Hence~\eqref{eq:groupoid_fibration}
  must be a homeomorphism by
  \cite{Meyer-Zhu:Groupoids}*{Proposition~5.9}; that is, \(F\) is a
  groupoid covering.
\end{proof}

Thus a groupoid fibration \(L\to H\) where the fibre~\(G\) is a space
is equivalent to an \(H\)\nb-action on~\(G\) with transformation
groupoid~\(L\).  This suggests to view a groupoid fibration \(F\colon
L\to H\) with a groupoid~\(G\) as its fibre as a generalised action
of~\(H\) on~\(G\) with transformation groupoid~\(L\).

\subsection{Groupoid equivalences defined by a fibration}
\label{sec:equivalences_from_fibration}

Let \(F\colon L\to H\) be a groupoid fibration.  We are going to
construct an action of the groupoid~\(H\) on~\(G\) by equivalences.
For \(h\in H^1\), let \(L_h=(F^1)^{-1}(h)\subseteq L^1\).  If \(y\in
H^0\), then~\(L_{1_y}\) is contained in~\(G^1\) and consists of those
\(g\in G^1\) with \(F^0(\rg(g))=y\).  Since \(F^0(\rg(g))=F^0(\s(g))\)
for \(g\in G\), the subset \((F^0)^{-1}(y)\subseteq G^0\) is
\(G\)\nb-invariant, and \(G_y\defeq L_{1_y}\) is the restriction
of~\(G\) to this \(G\)\nb-invariant subset.  This is the subgroupoid
of~\(G\) with \(G^0_y=(F^0)^{-1}(y)\) and
\(G^1_y=\s^{-1}_G(G^0_y)=\rg^{-1}_G(G^0_y)\).

\begin{lemma}
  \label{lem:Lh_locally_closed}
  The subgroupoids \(G_{\rg(h)}\) and~\(G_{\s(h)}\) act on~\(L_h\) by
  left and right multiplication, respectively.  With these actions,
  \(L_h\) is an equivalence between \(G_{\rg(h)}\) and~\(G_{\s(h)}\).
  The inversion and multiplication in~\(L\) restrict to isomorphisms
  of groupoid equivalences \(L_h\congto L_{h^{-1}}^*\) for \(h\in
  H^1\) and \(L_{h_1}\times_{G_y} L_{h_2}\congto L_{h_1h_2}\) for
  composable \(h_1,h_2\in H^1\) and \(y\defeq \s(h_1)=\rg(h_2)\); here
  the star denotes the inverse equivalence with left and right actions
  exchanged, and~\(\times_{G_y}\) denotes the composition of
  equivalences.
\end{lemma}

\begin{proof}
  We have \(L_{h_1}\cdot L_{h_2} \subseteq L_{h_1h_2}\)
  for all composable \(h_1,h_2\in H^1\).
  In particular, \(L_{1_{\rg(h)}}\cdot L_h \subseteq L_h\)
  and \(L_h \cdot L_{1_{\s(h)}} \subseteq L_h\).
  Since \(G_y^1=L_{1_{y}}\) for all \(y\in H^0\), this gives commuting left
  and right actions of \(G_{\rg(h)}\) and~\(G_{\s(h)}\) on~\(L_h\), respectively.
  Their anchor maps are the restrictions of \(\rg\)
  and~\(\s\)
  to~\(L_h\),
  respectively.  The left action of~\(G_{\rg(h)}\)
  gives a principal bundle with bundle projection
  \(\s\colon L_h\onto G^0_{\s(h)}\)
  because of~\eqref{eq:fibration_principal_bundle}: we have simply
  restricted the principal \(G\)\nb-bundle
  \(L^1\onto H^1\times_{\s,H^0,F^0} L^0\)
  to \(\{h\}\times G^0_{\s(h)} \subseteq H^1\times_{\s,H^0,F^0} L^0\).
  Taking inverses everywhere gives the same statement for the right
  action of~\(G_{\s(h)}\) on~\(L_h\).

  The isomorphism \(L_h \cong L_{h^{-1}}^*\) is trivial.  The
  multiplication in~\(L^1\) restricts to a continuous map \(L_{h_1}
  \times_{\s,\rg} L_{h_2} \to L_{h_1 h_2}\) that equalises \((l_1,g
  l_2)\) and \((l_1 g, l_2)\) for all \(l_i\in L_{h_i}\), \(g\in G_y\)
  with \(\s(l_1)=\rg(g)\), \(\s(g)=\rg(l_2)\).  Hence it induces a
  continuous map \(L_{h_1} \times_{G_y} L_{h_2} \to L_{h_1 h_2}\).
  This map is equivariant with respect to the left action
  of~\(G_{\rg(h_1)}\) and the right action of~\(G_{\s(h_2)}\).  An
  equivariant, continuous map between two groupoid equivalences such
  as \(L_{h_1}\times_{G_y} L_{h_2}\) and~\(L_{h_1h_2}\) is
  automatically a homeomorphism; this follows from the statement in
  \cite{Meyer-Zhu:Groupoids}*{Theorem~7.15} that all \(2\)\nb-arrows
  in the bicategory of bibundle functors are invertible.
\end{proof}

If we equip~\(H\) with the discrete topology, then the
equivalences~\(L_h\) for \(h\in H\) with the isomorphisms of
equivalences \(L_{h_1}\times_{G_y} L_{h_2}\to L_{h_1h_2}\) for
\(y\defeq \s(h_1)=\rg(h_2)\) form an action of~\(H\) on the groupoid
\(\bigsqcup_{y\in H^0} G_y\) by equivalences,
compare~\cite{Buss-Meyer:Actions_groupoids}.  The continuity of this
action is expressed by putting a topology on \(L= \bigsqcup_{h\in H}
L_h\) such that the involution and multiplication are continuous and
the maps \(\rg,\s\colon L\onto L^0\) and~\eqref{eq:groupoid_fibration}
are open and surjective.

\subsection{Classical groupoid actions}
\label{sec:classical_action}

We define actions of one topological groupoid on another by
automorphisms and construct a transformation groupoid and a groupoid
fibration from it.  This corroborates our interpretation of groupoid
fibrations as generalised groupoid actions.

\begin{definition}
  \label{def:classical_action}
  Let \(H\) and~\(G\) be topological groupoids.  A \emph{classical
    action} of~\(H\) on~\(G\) consists of \(H\)\nb-actions on \(G^0\)
  and~\(G^1\) such that
  \begin{enumerate}
  \item \(\rg,\s\colon G^1\rightrightarrows G^0\) are
    \(H\)\nb-equivariant;
  \item the multiplication map \(G^1\times_{\s,G^0,\rg} G^1 \to G^1\)
    is \(H\)\nb-equivariant.  Here we use the diagonal action of~\(H\)
    on \(G^1\times_{\s,G^0,\rg} G^1\), which exists because \(\s\)
    and~\(\rg\) are equivariant.
  \end{enumerate}
\end{definition}

The equivariance of the multiplication in~\(G\) implies that the unit
and inversion maps are \(H\)\nb-equivariant as well.  Let
\(\rg_H\colon G^0\to H^0\) and \(\rg_H\colon G^1\to H^0\) denote the
anchor maps.  The map \(\rg_H\colon G^0\to H^0\) is \(G\)\nb-invariant
because \(\rg_H(\rg(g))=\rg_H(g) = \rg_H(\s(g))\) for all \(g\in
G^1\).

More explicitly, the \(H\)\nb-equivariance of the multiplication map
means that
\[
h\cdot (g_1\cdot g_2) = (h\cdot g_1) \cdot (h\cdot g_2)
\]
for \(h\in H^1\), \(g_1,g_2\in G^1\) with \(\s(h)=\rg_H(g_1)\),
\(\s(g_1)=\rg(g_2)\).  We check that \(h\cdot (g_1\cdot g_2)\) and
\((h\cdot g_1) \cdot (h\cdot g_2)\) are defined.  The product
\(g_1\cdot g_2\) is defined because \(\s(g_1)=\rg(g_2)\), and \(h\cdot
(g_1\cdot g_2)\) is defined because \(\s(h) = \rg_H(g_1) =
\rg_H(\rg(g_1)) = \rg_H(\rg(g_1 g_2)) = \rg_H(g_1 g_2)\).  The product
\(h\cdot g_1\) is defined because \(\s(h)=\rg_H(g_1)\), and \(h\cdot
g_2\) is defined because \(\s(h) = \rg_H(g_1) = \rg_H(\s(g_1)) =
\rg_H(\rg(g_2)) = \rg_H(g_2)\).  Since \(\s(h\cdot g_1) = h\cdot
\s(g_1) = h\cdot \rg(g_2) = \rg(h\cdot g_2)\), the product \((h\cdot
g_1) \cdot (h\cdot g_2)\) is defined.

\begin{remark}
  \label{rem:classical_action}
  Since \(\rg_H\colon G^0\to H^0\) is \(G\)\nb-invariant, the
  subspaces \(\rg_H^{-1}(y)\subseteq H^0\) are \(G\)\nb-invariant, so
  that we may restrict~\(G\) to topological subgroupoids \(G_y \defeq
  G|_{\rg_H^{-1}(y)}\) for \(y\in H^0\).  Then \(G=\bigsqcup_{y\in
    H^0} G_y\) as a set.  An arrow \(h\in H^1\) acts on~\(G^1\) by an
  isomorphism of topological groupoids \(\alpha_h\colon G^1_{\s(h)}
  \congto G^1_{\rg(h)}\), such that \(\alpha_{h_1 h_2} = \alpha_{h_1}
  \alpha_{h_2}\) for all \(h_1,h_2\in H^1\) with \(\s(h_1)=\rg(h_2)\).
  Moreover, the map \(H^1\times_{\s,H^0,\rg_H} G^1\to G^1\),
  \((h,g)\mapsto \alpha_h(g)\), is continuous.  Conversely, assume
  that we are given a decomposition \(G= \bigsqcup G_y\) through a
  \(G\)\nb-invariant continuous map \(G^0\to H^0\) and groupoid
  isomorphisms \(\alpha_h\colon G_{\s(h)} \congto G_{\rg(h)}\) for
  \(h\in H^1\), such that \(\alpha_{h_1}\alpha_{h_2} = \alpha_{h_1
    h_2}\) for composable \(h_1,h_2\in H^1\) and such that
  \(H^1\times_{\s,H^0,\rg_H} G^1\to G^1\), \((h,g)\mapsto
  \alpha_h(g)\), is continuous.  Then the map
  \(H^1\times_{\s,H^0,\rg_H} G^0\to G^0\), \((h,x)\mapsto
  \alpha_h(x)\), is continuous as well, and these actions of~\(H\) on
  \(G^0\) and~\(G^1\) form a classical action as in
  Definition~\ref{def:classical_action}.  Thus a classical action
  of~\(H\) on~\(G\) is the same as a decomposition of~\(G\) into a
  bundle of topological groupoids over~\(H^0\) (that is, a continuous,
  \(G\)\nb-equivariant map \(\rg_H\colon G^0\to H^0\)) and an action
  of~\(H\) by topological groupoid isomorphisms between the fibres of
  this bundle, such that the induced action on~\(G^1\) is continuous.

  In particular, if~\(H\) is a topological group, then a classical
  action is a group homomorphism from~\(H\) to the group of groupoid
  automorphisms of~\(G\) that is continuous in the sense that the
  action map \(H^1\times G^1\to G^1\) is continuous (this implies the
  continuity of \(H^1\times G^0\to G^0\)).  Actions of this type are
  used in \cites{Kaliszewski-Muhly-Quigg-Williams:Coactions_Fell,
    Kaliszewski-Muhly-Quigg-Williams:Fell_bundles_and_imprimitivity_theoremsI,
    Kaliszewski-Muhly-Quigg-Williams:Fell_bundles_and_imprimitivity_theoremsII,
    Kaliszewski-Muhly-Quigg-Williams:Fell_bundles_and_imprimitivity_theoremsIII}.
\end{remark}

Now we build a \emph{transformation groupoid}~\(L\) for a classical
action of~\(H\) on~\(G\); this may also be called a semidirect product
groupoid because that is what it is for an action of a topological
group on another topological group.  Let \(L^0 \defeq G^0\) and
\(L^1\defeq H^1\times_{\s,H^0,\rg_H} G^1\).  Define \(\s,\rg\colon
L^1\rightrightarrows L^0\) by \(\s(h,g) = \s(g)\), \(\rg(h,g)=h
\cdot\rg(g)\) for \(h\in H^1\), \(g\in G^1\) with \(\s(h) =
\rg_H(g)\).  Define the multiplication by
\[
(h_1,g_1)\cdot (h_2,g_2) \defeq
(h_1\cdot h_2, (h_2^{-1}\cdot g_1)\cdot g_2)
\]
for \(h_1,h_2\in H^1\), \(g_1,g_2\in G^1\) with
\(\s(h_1)=\rg_H(g_1)\), \(\s(h_2)=\rg_H(g_2)\), \(\s(g_1)= h_2\cdot
\rg(g_2)\).

\begin{lemma}
  \label{lem:classical_action_trafo}
  The data above defines a topological groupoid.
\end{lemma}

\begin{proof}
  The range map is well defined because \(\rg\colon G^1\to G^0\) is
  \(H\)\nb-equivariant.  We check that the multiplication is well
  defined.  Let \(h_1,h_2\in H^1\), \(g_1,g_2\in G^1\) satisfy
  \(\s(h_1)=\rg_H(g_1)\), \(\s(h_2)=\rg_H(g_2)\), \(\s(g_1)= h_2\cdot
  \rg(g_2)\).  Then \(\s(h_1) = \rg_H(g_1) = \rg_H(\s(g_1)) =
  \rg_H(h_2\cdot \rg(g_2)) = \rg(h_2) = \s(h_2^{-1})\), so
  \(h_2^{-1}\cdot g_1\) and \(h_1\cdot h_2\) are defined.  And
  \(\s(h_2^{-1}\cdot g_1) = h_2^{-1} \cdot \s(g_1) = \rg(g_2)\), so
  \((h_2^{-1}\cdot g_1)\cdot g_2\) is defined.  And
  \(\rg_H((h_2^{-1}\cdot g_1)\cdot g_2) = \rg_H(h_2^{-1}\cdot
  (g_1\cdot (h_2\cdot g_2))) = \rg(h_2^{-1}) = \s(h_1 h_2)\), so
  \((h_1\cdot h_2, (h_2^{-1}\cdot g_1)\cdot g_2) \in L^1\).

  Direct computations give the following:
  \begin{itemize}
  \item \(\s\bigl((h_1,g_1)\cdot (h_2,g_2)\bigr) = \s(g_2) =
    \s(h_2,g_2)\)
  \item \(\rg\bigl((h_1,g_1)\cdot (h_2,g_2)\bigr) = h_1\cdot \rg(g_1)
    = \rg(h_1,g_1)\);
  \item the arrow \((1_{\rg_H(x)},1_x)\) for \(x\in G^0\) is the unit
    arrow on~\(x\);
  \item the arrow \((h^{-1},h\cdot g^{-1})\) is inverse to~\((h,g)\)
    for \(h\in H^1\), \(g\in G^1\) with \(\s(h)=\rg_H(g)\), and the
    inversion map is a homeomorphism;
  \item the multiplication is associative.
  \end{itemize}
  The source map of~\(L\) is an open surjection as the product of two
  open surjections, namely, \(\s\colon G^1\onto G^0 = L^0\) and the
  coordinate projection \(L^1 = H^1\times_{\s,H^0,\rg_H} G^1 \onto
  G^1\), which is the pull-back of the open surjection \(\s\colon
  H^1\onto H^0\) along \(\rg_H\colon G^1\to H^0\).
\end{proof}

\begin{proposition}
  \label{pro:classical_action_fibration}
  The map \(\rg_H\colon L^0 = G^0\to H^0\) on objects and the
  coordinate projection \(\pr_1\colon L^1 = H^1\times_{\s,H^0,\rg_H}
  G^1 \to H^1\) on arrows define a continuous functor \(F\colon L\to
  H\), which is a groupoid fibration.
\end{proposition}

\begin{proof}
  Direct computations show that~\(F\) is a functor.  The map
  \((F^1,\s)\colon L^1 \to H^1\times_{\s,H^0,F^0} L^0\)
  in~\eqref{eq:groupoid_fibration} is equivalent to the map
  \(H^1\times_{\s,H^0,\rg_H} G^1 \to H^1\times_{\s,H^0,\rg_H} G^0\),
  \((h,g)\mapsto (h,\s(g))\).  This is equivalent to the pull-back of
  the open surjection \(\s\colon G^1\onto G^0\) along the map
  \(\pr_2\colon H^1\times_{\s,H^0,\rg_H} G^0 \to G^0\).  Hence
  \((F^1,\s)\) is an open surjection.
\end{proof}

\begin{example}
  \label{exa:bitransformation_groupoid}
  Let~\(X\) be a topological space with commuting left and right
  actions of topological groupoids \(G\) and~\(H\).  Then~\(H\) acts
  on the transformation groupoid \(G\ltimes X\) on the right by the
  given action on objects \(x\in X\) and by \((g,x)\cdot h \defeq (g,
  x\cdot h)\) for \((g,x)\in (G\ltimes X)^1\), \(h\in H^1\) with
  \(\s(x) = \s(g,x) = \rg(h)\).  The transformation groupoid for this
  classical action of \(H\) on~\(G\ltimes X\) is the bi-transformation
  groupoid \(G\ltimes X\rtimes H\), which has unit space~\(X\), arrow
  space \(G\times_{\s,\rg}X\times_{\s,\rg}H\), \(\rg(g,x,h) \defeq
  g\cdot x\), \(\s(g,x,h) = x\cdot h^{-1}\), and \((g_1,x_1,h_1)\cdot
  (g_2,x_2,h_2) \defeq (g_1\cdot g_2,x,h_1 h_2)\) with \(x= g_2^{-1}
  x_1 = x_2 h_1\) for \(g_1,g_2\in G^1\), \(x_1,x_2\in X\),
  \(h_1,h_2\in H\) with \(\s(g_i) =\rg(x_i)\), \(\s(x_i)=\rg(h_i)\)
  for \(i=1,2\) and \(x_1\cdot h_1^{-1} = g_2\cdot x_2\).  The
  canonical functor \(F\colon G\ltimes X\rtimes H\to H\) is the anchor
  map \(\s\colon X\to H^0\) on objects and the coordinate projection
  \(\pr_3\colon \colon G\times_{\s,\rg}X\times_{\s,\rg}H \to H\) on
  arrows.  This is a groupoid fibration by
  Proposition~\ref{pro:classical_action_fibration}.
\end{example}

We may characterise exactly which groupoid fibrations come from
classical actions as above.  The transformation groupoid \(L= H\ltimes
G\) for a classical action comes with a canonical action of~\(H\) on
\(L^1 \defeq H^1\times_{\s,H^0,\rg_H} G^1\) by \(h_1\cdot (h_2,g)
\defeq (h_1\cdot h_2,g)\), with anchor map \(\rg_H\colon L^1\to H^0\),
\((h,g)\mapsto \rg(h)\).  This action commutes with the action
of~\(L^1\) on itself by right multiplication.  Thus it is an
\emph{actor} from~\(H\) to~\(L\) as in
\cite{Meyer-Zhu:Groupoids}*{Definition~4.20}.  The action of~\(H\)
on~\(L^1\) makes~\(F^1\) equivariant: \(F^1(h\cdot l) = h\cdot
F^1(l)\) for all \(h\in H^1\), \(l\in L^1\) with \(\s(h) = \rg_H(l)\).
There is, however, no canonical functor \(H\to H\ltimes G\)
unless~\(F^0\) is a homeomorphism because we lack a canonical map
\(H^0\to L^0\).

\begin{proposition}
  \label{pro:fibration=classical}
  A groupoid fibration \(F\colon L\to H\) with fibre~\(G\) comes from
  a classical action of~\(H\) on~\(G\) if and only if there is an
  actor from~\(H\) to~\(L\) that makes~\(F^1\) \(H\)\nb-equivariant.
\end{proposition}

\begin{proof}
  We have already described the actor from~\(H\) to~\(L\) for a
  classical action.  Conversely, assume that an actor from~\(H\)
  to~\(L\) is given that makes~\(F^1\) \(H\)\nb-equivariant.  Hence
  \(\rg_H(l) = \rg(F^1(l)) = F^0(\rg(l))\).  Write~\(\bullet\) for the
  left \(H\)\nb-action on~\(L^1\) to distinguish it from the
  multiplication in~\(L\).  We claim that the maps
  \(H^1\times_{\s,H^0,F^0\circ\rg} G^1\to L^1\), \((h,g)\mapsto
  h\bullet g\), and \(L^1 \to H^1\times_{\s,H^0,F^0\circ\rg} G^1\),
  \(l\mapsto (F^1(l),F^1(l)^{-1}\bullet l)\), are well defined,
  continuous and inverse to each other.  Hence they are both
  homeomorphisms.  The first map is clearly well defined, and the
  second map is well defined because the anchor map \(L^1\to H^0\) is
  \(F^0\circ \rg\) and \(F^1(F^1(l)^{-1}\bullet l) = F^1(l)^{-1}\cdot
  F^1(l) = 1_{F^0(\s(l))}\), that is, \(F^1(l)^{-1}\bullet l\) belongs
  to~\(G^1\).  Both maps are clearly continuous.  They are inverse to
  each other because \(F^1(h\cdot g) = h\cdot F^1(g) = h\) for all
  \((h,g)\in H^1\times_{\s,H^0,F^0\circ\rg} G^1\).

  Identify \(L^1 \cong H^1\times_{\s,H^0,F^0\circ\rg} G^1\) as above.
  The multiplication in~\(L\) must satisfy
  \begin{multline*}
    (h_1,g_1)\cdot (h_2,g_2)
    = (h_1 \bullet g_1) \cdot (h_2\bullet g_2)
    = (h_1 h_2) \bullet (h_2^{-1} \bullet g_1) \cdot
    (h_2\bullet (1_{\s(g_1)}\cdot g_2))
    \\= (h_1 h_2) \bullet (h_2^{-1} \bullet 1_{\rg(g_1)}) \cdot
    g_1\cdot (h_2\bullet 1_{\s(g_1)}) \cdot g_2.
  \end{multline*}
  Since \(F^1\) is \(H\)\nb-equivariant, \(F^1(h_2^{-1}\bullet
  1_{\rg(g_1)}) = F^1(h_2)^{-1}\) and \(F^1(h_2\bullet 1_{\s(g_1)}) =
  F^1(h_2)\), so that the product \((h_2^{-1} \bullet 1_{\rg(g_1)})
  \cdot g_1\cdot (h_2\bullet 1_{\s(g_1)})\) belongs to~\(G^1\).
  Since~\(F^1\) is \(H\)\nb-equivariant, the multiplication in~\(L\)
  becomes \((h_1,g_1) \cdot (h_2, g_2) = (h_1 h_2, (h_2^{-1}\cdot
  g_1)\cdot g_2)\) with
  \[
  h\cdot g \defeq
  (h\bullet 1_{\rg(g)})\cdot g\cdot (h^{-1}\bullet 1_{\s(g)}).
  \]
  This is exactly as in the transformation groupoid for a classical
  action.  We also define an action of~\(H\) on~\(G^0\) by \(h\cdot x
  = \rg(h\cdot 1_x)\) as in the proof of
  \cite{Meyer-Zhu:Groupoids}*{Proposition~4.21}.  Reversing the
  computations in the proof of Lemma~\ref{lem:classical_action_trafo},
  we see that the formulas above must define a classical action
  of~\(H\) on~\(G\) by automorphisms because~\(L\) is a groupoid.
\end{proof}

\subsection{Translation action on the arrow space}
\label{sec:translation_arrows}

A motivating example in~\cite{Buss-Meyer:Actions_groupoids} is to
associate an action on a \(\Cst\)\nb-algebra to the translation action
of a locally Hausdorff, locally compact groupoid~\(H\) on its
arrow space~\(H^1\).  Since~\(H^1\) is non-Hausdorff, we cannot use
the commutative \(\Cst\)\nb-algebra of \(\Cont_0\)\nb-functions
on~\(H^1\).  Instead, we cover~\(H^1\) by Hausdorff, open subsets and
form the resulting \v{C}ech groupoid~\(G\).  It should carry an action
of~\(H\), which then induces an action of~\(H\) on the groupoid
\(\Cst\)\nb-algebra~\(\Cst(G)\).  This is accomplished
in~\cite{Buss-Meyer:Actions_groupoids} for étale groupoids.  Groupoid
fibrations allow to do the same for arbitrary locally Hausdorff
groupoids (the \(\Cst\)\nb-algebraic assertions also need a Hausdorff,
locally compact object space, a locally compact arrow space, and
a Haar system).  We construct the relevant example at the end of this
section, based on some simpler examples and a proposition on the
composition of groupoid fibrations.

\begin{example}
  \label{exa:hypercovers}
  Let~\(H\) be a topological groupoid and \(p\colon X\onto H^0\) an
  open, continuous surjection.  The \emph{pull-back~\(p^*(H)\) of~\(H\)
    along~\(p\)} is the topological groupoid with object space
  \(p^*(H)^0\defeq X\), arrow space
  \[
  p^*(H)^1\defeq
  X\times_{p,H^0,\rg}H^1\times_{\s,H^0,p}X=\{(x,h,y)\mid
  p(x)=\rg(h),\ p(y)=\s(h)\},
  \]
  \(\s(x,h,y)=y\), \(\rg(x,h,y)=x\), and \((x,h,y)\cdot
  (y,h',z)=(x,hh',z)\) (see
  \cite{Meyer-Zhu:Groupoids}*{Example~3.13}).  If~\(H\) is a space
  viewed as a groupoid with only
  identity arrows, this gives the \emph{\v{C}ech groupoid}~\(p^*(Y)\)
  of~\(p\), which has unit space~\(X\) and arrow space
  \(X\times_{p,Y,p}X = \{(x,x')\mid p(x)=p(y)\}\).

  The maps \(F^0\defeq p\) on objects and \(F^1(x,h,y)\defeq h\) on
  arrows give a functor \(F\colon p^*(H)\to H\) (see
  \cite{Meyer-Zhu:Groupoids}*{Example~3.18}).  It is always a groupoid
  fibration: identifying \(H^1 \times_{\s,H^0,F^0} p^*(H)^0 = H^1
  \times_{\s,H^0,p} X\), the map~\((F^1,\s)\)
  in~\eqref{eq:groupoid_fibration} becomes the coordinate projection
  \(X\times_{p,H^0,\rg} H^1 \times_{\s,H^0,p} X \onto H^1
  \times_{\s,H^0,p} X\), which is an open surjection because it is the
  pull-back of the open surjection~\(p\) along \(H^1
  \times_{\s,H^0,p} X\to H^0\), \((h,x)\mapsto \rg(h)\).  The fibre
  of~\(F\) is the \v{C}ech groupoid~\(p^*(H^0)\) of~\(p\).
\end{example}

\begin{lemma}
  \label{lem:fibration_to_space}
  Any functor from a topological groupoid~\(L\) to a space~\(Y\) is a
  groupoid fibration with fibre~\(L\).  It is a groupoid covering if
  and only if~\(L\) is a space viewed as a groupoid.
\end{lemma}

\begin{proof}
  A functor \(F\colon L\to Y\) is equivalent to a continuous map
  \(F\colon L^0\to Y\) that is \(L\)\nb-invariant, that is,
  \(f(\s(l))=f(\rg(l))\) for all \(l\in L^1\).  There is a
  homeomorphism \(Y\times_{\Id,f} L^0\congto L^0\), \((y,x)\mapsto
  x\), which identifies the map in~\eqref{eq:groupoid_fibration} with
  the source map \(\s\colon L^1\onto L^0\).  Since we assume~\(\s\) to
  be an open surjection, \(F\) is a groupoid fibration automatically.
  It is a groupoid covering if and only if \(\s\colon L^1\onto L^0\)
  is a homeomorphism, if and only if~\(L\) is a space viewed as a
  groupoid.  The fibre of~\(F\) is all of~\(L\) because \(Y^0=Y^1\).
\end{proof}

\begin{remark}
  A functor from a space~\(X\) to a topological groupoid~\(H\) is a
  fibration only if the relevant part of~\(H\) is a space viewed as a
  groupoid.  More precisely, such a functor is the same as a
  continuous map \(f\colon X\to H^0\), and this is a fibration if and
  only if any arrow in~\(H\) with source or range in the image
  of~\(f\) is an identity arrow.
\end{remark}

\begin{proposition}
  \label{pro:compose_fibrations}
  Let \(F_i\colon L_i\to H_i\) be groupoid fibrations with
  fibre~\(G_i\) for \(i=1,2\) with \(H_1=L_2\).  Let \(F\defeq
  F_2\circ F_1\colon L_1 \to H_1=L_2\to H_2\).  Then~\(F\) is a
  groupoid fibration as well; let~\(G\) be its fibre.  The
  functor~\(F_1\) restricts to a functor \(G\to G_2\), which is a
  groupoid fibration with fibre~\(G_1\).  If~\(F_1\) is a groupoid
  covering, then \(F|_G\colon G\to G_2\) is a groupoid covering.
  If~\(F_2\) is a groupoid covering, then \(G_1=G\).
\end{proposition}

\begin{proof}
  We have assumed that the maps \((F_i^1,s_{L_i})\colon L_i^1 \onto
  H_i^1 \times_{s_{H_i},H_i^0,F_i^0} L_i^0\) for \(i=1,2\) are open
  surjections.  The pull-back of the open
  surjection~\((F_2^1,s_{L_2})\) along the map \(F_1^0\colon
  L_1^0\to H_1^0=L_2^0\) remains an open surjection.  The
  homeomorphism
  \[
  (H_2^1 \times_{s_{H_2},H_2^0,F_2^0} L_2^0)
  \times_{\pr_2,L_2^0,F_1^0} L_1^0 \congto
  H_2^1 \times_{s_{H_2},H_2^0,F_2^0\circ F_1^0} L_1^0
  \]
  identifies this pull-back with the map
  \[
  H_1^1\times_{s,F_1^0} L_1^0 = L_2^1\times_{s,F_1^0} L_1^0\onto
  H_2^1\times_{s,F^0} L_1^0,\qquad
  (g,x)\mapsto (F_2^1(g),x).
  \]
  Composing with \((F_1^1,s_{L_1})\)
  gives \((F,s)\colon L_1^1\onto H_2^1\times_{s,F_1^0} L_1^0\).
  So~\(F\) is a groupoid fibration.

  Since \(F_2^1\circ F_1^1(g)\) is a unit if and only if
  \(F_1^1(g)\) belongs to the fibre~\(G_2\) of~\(F_2\), the preimage
  of \(G_2\times_{s,F_1^0} L_1^0\subseteq H_1^1\times_{s,F_1^0}
  L_1^0\) under the map~\((F^1_1,s)\) is naturally isomorphic to the
  fibre of~\(F\).  Restricting an open surjection to the preimage of
  a subspace is also a case of pull-back, so the restriction remains
  an open surjection.  This says exactly that the restriction
  \(F_1|_{G}\colon G\to G_2\) is a groupoid fibration.  Since
  \(G_1\subseteq G\subseteq L_1\), the fibre of~\(F_1|_{G}\) is the
  same as for~\(F_1\).  So the groupoid fibration \(F_1|_{G}\colon
  G\to G_2\) has fibre~\(G_1\) as asserted.

  If~\(F_1\) is a groupoid covering, then its fibre is just a space.
  Hence the fibration~\(F|_G\) also has a space as its fibre and thus
  is a groupoid covering by Proposition~\ref{pro:groupoid_covering}.
  If~\(F_2\) is a groupoid covering, then~\(G_2\) is just a space.
  Hence the fibre~\(G_1\) of the groupoid fibration~\(F|_G\colon G\to
  G_2\) is~\(G\) by Lemma~\ref{lem:fibration_to_space}.
\end{proof}

\begin{example}
  \label{exa:act_arrows1}
  Let~\(H\) be a topological groupoid and let \(p\colon X\onto H^1\)
  be an open surjection.  We let~\(H\) act on~\(H^1\) by the left
  translation action and form its transformation groupoid \(H\ltimes
  H^1\).  Example~\ref{exa:hypercovers} applied to the map~\(p\) and
  the groupoid~\(H\ltimes H^1\) gives a groupoid fibration \(F_1\colon
  p^*(H\ltimes H^1) \to H\ltimes H^1\) with the \v{C}ech groupoid
  of~\(p\) as its fibre.  Example~\ref{exa:act_space} gives a groupoid
  covering \(F_2\colon H\ltimes H^1 \to H\).
  Proposition~\ref{pro:compose_fibrations} shows that the composite
  functor \(F\defeq F_2\circ F_1\colon p^*(H\ltimes H^1) \to H\ltimes
  H^1 \to H\) is a groupoid fibration with the same fibre~\(p^*(H^1)\)
  as~\(F^1\).  The groupoid fibration \(F\colon p^*(H\ltimes H^1) \to
  H\) describes an action of~\(H\) on the \v{C}ech
  groupoid~\(p^*(H^1)\) of~\(p\) with transformation groupoid
  \(p^*(H\ltimes H^1)\).
\end{example}

\begin{lemma}
  \label{lem:action_space_trafo-groupoid}
  The groupoid \(p^*(H\ltimes H^1)\) is isomorphic to the \v{C}ech
  groupoid of the open surjection \(\s\circ p\colon X\onto H^1 \onto
  H^0\).
\end{lemma}

\begin{proof}
  There is a canonical homeomorphism \((\rg,\s)\colon (H\ltimes H^1)^1
  \congto H^1\times_{\s,H^0,\s} H^1\), \((h_1,h_2)\mapsto (h_1\cdot
  h_2,h_2)\).  This together with the identity on objects identifies
  \(H\ltimes H^1\) with the \v{C}ech groupoid~\(\s^*(H^0)\) of
  \(\s\colon H^1\onto H^0\).  Thus \(p^*(H\ltimes H^1) \cong p^*
  \s^*(H^0) \cong (\s\circ p)^*(H^0)\).
\end{proof}

Now let~\(H\) be a topological groupoid with locally Hausdorff arrow
space~\(H^1\).  Choose a cover \(H^1 = \bigcup_{U\in\mathcal{U}} U\)
by Hausdorff, open subsets and let \(X\defeq
\bigsqcup_{U\in\mathcal{U}} U\) with the canonical map \(p\colon
X\onto H^1\) that is the inclusion map on each component \(U\subseteq
Y\).  This is a local homeomorphism and \emph{a fortiori} an open
surjection.  The construction above gives a groupoid fibration from
the \v{C}ech groupoid \((\s\circ p)^*(H^0)\) of \(\s\circ p\colon
X\onto H^0\) to~\(H\) with the \v{C}ech groupoid~\(p^*(H^1)\) as its
fibre.  Both \((\s\circ p)^*(H^0)\) and~\(p^*(H^1)\) have object
space~\(X\), which is Hausdorff, and arrow spaces contained in
\(X\times X\), which forces their arrow spaces to be Hausdorff as
well.  The \v{C}ech groupoid~\(p^*(H^1)\) is étale because~\(p\) is a
surjective local homeomorphism.

\subsection{Groupoid extensions}
\label{sec:groupoid_extensions}

\begin{definition}
  \label{def:groupoid_extension}
  A \emph{\textup{(}topological\textup{)} groupoid extension} is a
  diagram \(G\hookrightarrow L\to H\) where the functor \(F\colon L\to
  H\) is a groupoid fibration of topological groupoids with fibre
  \(G\subseteq L\), such that \(F^0\colon L^0\congto H^0\) is a
  homeomorphism.
\end{definition}

\begin{lemma}
  \label{lem:fibration_iso_objects}
  A functor \(F\colon L\to H\) that is a homeomorphism on objects is a
  groupoid fibration is and only if \(F^1\colon L^1\to H^1\) is an
  open surjection, and a groupoid covering if and only if~\(F^1\) is a
  homeomorphism as well, that is, \(F\) is an isomorphism of
  topological groupoids.
\end{lemma}

\begin{proof}
  Simplify~\eqref{eq:groupoid_fibration} using the homeomorphism
  \(H^1\times_{\s,H^0,F^0} L^0\congto H^1\), \((h,x)\mapsto h\).
\end{proof}

\begin{proposition}
  \label{pro:group_fibration}
  Let \(G\) and~\(H\) be topological groups.  A groupoid fibration
  \(L\to H\) with fibre~\(G\) is the same as an extension of
  topological groups \(G\rightarrowtail L\onto H\).
\end{proposition}

\begin{proof}
  We have \(L^0=G^0=\star\), so~\(L\) is a group as well.  A group
  fibration is the same as a continuous, open, surjective group
  homomorphism by Lemma~\ref{lem:fibration_iso_objects}.  The fibre is
  the kernel of this homomorphism.
\end{proof}

A topological group extension \(G\rightarrowtail L\onto H\) comes from
a classical action of~\(H\) on~\(G\) by group automorphisms if and
only if it splits by a continuous group homomorphism \(H\to L\); this
is well known, and also follows from
Proposition~\ref{pro:fibration=classical}.  We consider \emph{any}
topological group extension as an ``action'' of~\(H\) on~\(G\).  For
Polish groups, such extensions may be classified by Borel measurable
\(2\)\nb-cocycles, see~\cite{Brown:Extensions}.

\begin{remark}
  \label{rem:pretopology_group}
  Proposition~\ref{pro:group_fibration} works in a category with
  pretopology if we assume that our category has a final
  object~\(\star\).  Then we may define a group as a groupoid with
  \(G^0=\star\).  (Any map from a non-empty space to the one-point
  space is an open surjection.  Some basic features of topological
  groups and their actions only work in the abstract setting if we
  assume that any map to the final object is a cover, except possibly
  for the map from the initial object, if that exists, compare
  \cite{Meyer-Zhu:Groupoids}*{Assumption 2.10 and Examples 3.14
    and~4.9}.)
\end{remark}

\begin{lemma}
  \label{lem:extension_normal_subgroup_bundle}
  Let \(G\hookrightarrow L\to H\) be a groupoid extension.  Then
  \(G\subseteq L\) is a normal subgroup bundle on which the range map
  is open, and \(H = L/G\).  Conversely, any normal subgroup bundle
  \(G\subseteq L\) on which the range map is open appears in a
  groupoid extension that is unique up to isomorphism.
\end{lemma}

\begin{proof}
  All three groupoids in a groupoid extension have the same or
  homeomorphic object spaces: \(G^0=L^0\cong H^0\).  If \(g\in G^1\),
  then \(F^0(\s(g))=\s(F^1(g))=\rg(F^1(g))=F^0(\rg(g))\), so~\(G\) is
  a bundle of groups contained in~\(L\).  If \(g\in G^1\), \(l\in
  L^1\) with \(\s(g)=\rg(l)\), then \(l g l^{-1}\in G^1\) as well,
  that is, the subgroup bundle~\(G\) is normal in~\(L\).  The range
  and source maps of~\(L\) restrict to open mappings on~\(G\) by
  Lemma~\ref{lem:fibre_top_groupoid}.  Conversely, let \(G\subseteq
  L\) be a normal subgroup bundle.  This is a topological groupoid
  with the subspace topology if and only if the range map on~\(L\)
  restricts to an open map on~\(G\).  Assume this.  Since~\(G\) is a
  normal subgroup bundle, there is a unique multiplication on \(H^1
  \defeq L^1/ G\) so as to give a groupoid~\(H\) with object
  set~\(H^0\) and such that the quotient map \(L\to H\) is a functor.
  We equip~\(H^1\) with the quotient topology.  Then~\(H\) is a
  topological groupoid.  The quotient map \(L^1\to L^1/G = H^1\) is
  automatically an open surjection by
  \cite{Meyer-Zhu:Groupoids}*{Proposition~9.34}.  Hence \(G\into
  L\onto H\) is an extension of topological groupoids by
  Lemma~\ref{lem:fibration_iso_objects}.  If \(G\hookrightarrow L\to
  H\) is a groupoid extension, then~\(H\) is canonically homeomorphic
  to the quotient~\(L/G\) described above by
  Proposition~\ref{pro:fibration_basic_action_GL}, simplified using
  \(H^1 \times_{\s,H^0,F^0} L^0 \cong H^1\) as in the proof of
  Lemma~\ref{lem:fibration_iso_objects}.  Thus any groupoid extension
  comes from a unique normal subgroup bundle \(G\subseteq L\) on which the
  range map is open.
\end{proof}

\begin{example}
  \label{exa:isotropy-ext}
  Let~\(L\) be an étale groupoid and consider the interior \(G\defeq
  \Iso^\circ(L)\) of its isotropy bundle \(\{g\in L\mid
  \s(g)=\rg(g)\}\).  This is an open, normal group bundle.  So we
  may form a groupoid extension \(G\into L\onto H\) with \(H^1\defeq
  L^1/G\) as in Lemma~\ref{lem:extension_normal_subgroup_bundle}.
\end{example}

\begin{example}
  \label{exa:extensions}
  A \emph{central groupoid extension} \(G\into L\onto H\) is an
  extension where \(G=L^0\times \Gamma\) is a trivial group bundle
  such that the conjugation action of arrows in~\(L\) induces the
  trivial map on~\(\Gamma\).  More precisely, let \(\iota\colon
  L^0\times\Gamma\into L\) be the embedding, then we require
  \(\iota(\rg(l),\gamma)\cdot l=l\cdot\iota(\s(l),\gamma)\) for all
  \(\gamma\in \Gamma\) and \(l\in L\).  This forces~\(\Gamma\) to be
  Abelian.  Such extensions have been extensively studied, especially
  for \(\Gamma=\Torus\), which leads to twisted groupoid
  \(\Cst\)\nb-algebras, see \cites{Muhly-Williams:Continuous-traceII,
    Renault:Cartan.Subalgebras}.  These extensions also appear in the
  study of groupoid cohomology (see, for instance,
  \cite{Renault:Transverse_dynamical}) and are closely related to
  gerbes and thus to twisted \(K\)\nb-theory, see
  \cite{Tu-Xu-Laurent-Gengoux:Twisted_K}*{Remark 2.14}.
\end{example}

The class of groupoid fibrations where~\(F^0\) is not a homeomorphism
but only an open surjection also deserves special attention (compare
Remark~\ref{rem:Lie_groupoid_fibration}).  They should behave like
groupoid extensions where the ``kernel'' is no longer a group bundle.

\begin{lemma}
  \label{lem:F1_F0_open_surjective}
  Let \(F\colon L\to H\)
  be a groupoid fibration.  The map \(F^1\colon L^1\to H^1\)
  is surjective or open if and only if \(F^0\colon L^0\to H^0\)
  is surjective or open, respectively.
\end{lemma}

\begin{proof}
  Since \(\s\colon H^1\onto H^0\)
  is an open surjection, so is the coordinate projection
  \(H^1\times_{\s,F^0} L^0 \onto L^0\).
  Therefore, the coordinate projection
  \(H^1\times_{\s,F^0} L^0 \to H^1\)
  is an open surjection if and only if \(F^0\colon L^0\to H^0\)
  is an open surjection (this is the locality of covers for the
  pretopology of open surjections in the notation
  of~\cite{Meyer-Zhu:Groupoids}).  Even more, it can be checked by
  hand that~\(F^0\)
  is open or surjective, respectively, if and only if the coordinate
  projection \(H^1\times_{\s,F^0} L^0 \to H^1\)
  is.  If~\(F\)
  is a groupoid fibration, then the map
  \((F^1,\s)\colon L^1\onto H^1\times_{\s,F^0} L^0\)
  is an open surjection as well.  The two-out-of-three property for
  the pretopology of open surjections says that the composite map
  \(F^1\colon L^1\to H^1\times_{\s,F^0} L^0 \to H^1\)
  is an open surjection if and only if the projection
  \(H^1\times_{\s,F^0} L^0 \to H^1\)
  is one; once again, it can be checked that the statement remains
  true for open maps and surjective maps separately.  Since the
  coordinate projection \(H^1\times_{\s,F^0} L^0 \to H^1\)
  is open or surjective if and only if the map
  \(F^0\colon L^0\to H^0\)
  is so, \(F^1\) is open or surjective if and only if~\(F^0\) is so.
\end{proof}

\begin{example}
  \label{exa:surjective_not_enough}
  A (continuous) open, surjective functor \(F\colon L\to H\) need not
  be a fibration.  For instance, let~\(K\) be the pair groupoid on the
  \(2\)\nb-element set.  This is a finite, discrete groupoid with two
  objects and four arrows \(K^1\defeq \{a,b,\gamma,\gamma^{-1}\}\),
  where \(a,b\) are the two unit arrows and~\(\gamma\) is the
  non-trivial arrow with \(\s(\gamma)=a\), \(\rg(\gamma)=b\).  Let
  \(K_n\defeq \{a_n,b_n,\gamma_n,\gamma_n^{-1}\}\) be a copy of~\(K\)
  for each \(n\in \N\) and \(L\defeq \bigsqcup_{n\in \N} K_n\).  Let
  \(F\colon L\to \Z\) be the unique functor that sends
  \(\gamma_n\mapsto n\).  This functor between discrete groupoids is
  surjective, but not a fibration.
\end{example}

\begin{remark}
  The notion of an ``extension'' of Borel groupoids in
  \cite{Renault_AnantharamanDelaroche:Amenable_groupoids}*{Definition~5.2.7}
  is closely related to our definition of a groupoid fibration with an
  open surjection~\(F^0\).  In the world of Borel structures, the
  condition of being an open surjection is replaced by the condition
  of being surjective, so the requirement is that the maps \(F^0\) and
  \((F^1,\s)\) be surjective.

  Under mild extra conditions, amenability of Borel groupoids is
  preserved under extensions by
  \cite{Renault_AnantharamanDelaroche:Amenable_groupoids}*{Theorem~5.2.14}.
  Moreover,
  Renault~\cite{Renault:Topological_amenability_is_a_Borel_property}
  proves that Borel amenability is equivalent to topological
  amenability for locally compact and locally Hausdorff topological
  groupoids with Haar systems and Hausdorff unit space.  Therefore, if
  \(L\to H\) is a groupoid fibration with fibre~\(G\) and the map
  \(F^0\colon L^0\to H^0\) is surjective, then \(L\) is amenable if
  \(G\) and~\(H\) are.  We expect this to remain true without the
  surjectivity assumption on~\(F^0\), but have not examined the matter
  closely.
\end{remark}

\section{Étale groupoid fibrations and inverse semigroup gradings}
\label{sec:fibrations_gradings}

Let~\(H\)
be an étale groupoid.  We are going to show that groupoid
fibrations~\(L\to H\)
are essentially equivalent to the gradings by inverse semigroups used
in~\cite{Buss-Meyer:Actions_groupoids} to model groupoid actions on
other groupoids.  Since this section will not be needed in the rest of
the paper, we assume that the reader is familiar with the relevant
notions from~\cite{Buss-Meyer:Actions_groupoids}.  Recall that every
étale groupoid~\(H\)
is isomorphic to a groupoid of germs~\(S\ltimes Z\)
for some action of a unital inverse semigroup~\(S\)
on a space~\(Z\)
(see~\cite{Exel:Inverse_combinatorial}).  We assume~\(H\)
to be of this form.

\begin{theorem}
  \label{the:S-action_groupoid_fibration}
  A groupoid fibration \(F\colon L\to S\ltimes Z\) with fibre~\(G\)
  is equivalent to an \(S\)\nb-grading on~\(L\) with \(L_1=G\) -- so
  that~\(L\) is the transformation groupoid \(S\ltimes G\) for an
  action of~\(S\) on~\(G\) by partial equivalences -- together with
  a \(G\)\nb-invariant continuous map \(G^0=L^0\to Z\) that is
  \(S\)\nb-equivariant for the induced action of~\(S\) on~\(G^0/G\).
\end{theorem}

\begin{proof}
  First let~\(F\) be a groupoid fibration.  We define an
  \(S\)\nb-grading on~\(L\) by \(L_t \defeq (F^1)^{-1}(t)\) for \(t\in
  S\), viewed as an open subset of \((S\ltimes Z)^1\).  More
  precisely, each \(t\in S\) is viewed as the set of germs \([t,x]\in
  S\ltimes Z\) with \(x\) in the domain \(\dom(t)=D_{t^*t}\subseteq
  Z\) of the \(S\)\nb-action.  The subspaces~\(L_t\) are open
  because~\(F^1\) is continuous.  The unit fibre~\(L_1\) of the
  grading is equal to the fibre~\(G\) of~\(F\) because the unit arrows
  of \(S\ltimes Z\) are exactly the germs of the form~\([1,x]\).  The
  following properties required for an \(S\)\nb-grading are trivial:
  \[
  L_t\cdot L_u\subseteq L_{tu},\qquad
  L_t^{-1} = L_{t^*},\qquad
  \bigcup_{t\in S} L_t=L^1.
  \]
  If \(l \in L_t\cap L_u\), then \(F^1(l)\in t\cap u\); by the
  definition of \(S\ltimes Z\), this means that \(F^1(l)\in v\) for
  some \(v\in S\) with \(v\le t, u\).  Thus
  \[
  L_t\cap L_u = \bigcup_{v\in S, v\le t,u} L_v.
  \]
  The only property of an \(S\)\nb-grading that requires the fibration
  condition is
  \[
  L_t\cdot L_u \supseteq L_{tu}
  \]
  for all \(t,u\in S\).  Let \(l\in L_{tu}\).  Then \(F^1(l)\in tu\),
  so we may factor \(F^1(l)=h_1h_2\) with \(h_1\in t\), \(h_2\in u\).
  Since \(\s(h_2)=\s(F^1(l))=F^0(\s(l))\)
  and~\eqref{eq:groupoid_fibration} is surjective, there is \(l_2\in
  L^1\) with \(\s(l_2)=\s(l)\) and \(F^1(l_2)=h_2\).  Then \(l_2\in
  L_u\) because \(h_2\in u\), and \(l_1\defeq l\cdot l_2^{-1}\in L_t\)
  because \(F^1(l_1)=h_1\in t\).  Thus \(l\in L_tL_u\) as desired.

  Since \(F^0(\s(g))=\s(F^1(g))=\rg(F^1(g))=F^0(\rg(g))\) for all
  \(g\in G^1\), the continuous map \(F^0\colon G^0=L^0\to Z\) is
  \(G\)\nb-invariant and hence descends to a continuous map \(G^0/G\to
  Z\).  We must show that this map is \(S\)\nb-equivariant.  We recall
  how the action of~\(S\) on~\(G^0/G\) by partial homeomorphisms is
  defined (see the comments before Remark~2.14
  in~\cite{Buss-Meyer:Actions_groupoids}).  For \(t\in S\), let
  \(U_{tt^*} \defeq \rg(L_t) = \s(L_{t^*})\); these are
  \(G\)\nb-invariant open subsets of~\(L^0\), which we view as open
  subsets of~\(G^0/G\).  If \(x\in U_{t^*t} = \s(L_t)\), then pick
  \(l\in L_t\) with \(\s(l)=x\) and define \(t\cdot [x] \defeq
  [\rg(l)]\), where the brackets mean that we take the \(G\)\nb-orbit.
  This does not depend on the choice of~\(l\) because all choices
  of~\(L\) are of the form \(g\cdot l\) with \(g\in G^1\) and we
  divided out the \(G\)\nb-action.  This is indeed a homeomorphism
  from~\([U_{t^*t}]\) onto~\([U_{tt^*}]\), and these partial
  homeomorphisms form an action of~\(S\).

  Since~\(F\) is a functor, \(F^1(l)\in t\) has range~\(F^0(\rg(l))\)
  and source~\(F^0(\s(l))\).  Since \(t\cdot F^0(\s(l))=F^0(\rg(l))\),
  the map \(G^0/G\to Z\) induced by~\(F^0\) is \(S\)\nb-equivariant.
  We have built an \(S\)\nb-grading and an \(S\)\nb-equivariant map
  from a groupoid fibration.

  Conversely, take an \(S\)\nb-grading on~\(L\) with \(L_1 = G\)
  and an \(S\)\nb-equivariant continuous map \(\rho\colon G^0/G\to Z\).
  We are going to define a groupoid fibration \(L\to S\ltimes Z\).  We
  let~\(F^0\) be the composite of~\(\rho\) with the orbit space
  projection \(L^0 = G^0\to G^0/G\).

  For \(l\in L\), there is \(t\in S\) with \(l\in L_t\).  We want to
  define~\(F^1(l)\in (S\ltimes Z)^1\) as the
  germ~\([t,\rho(\s(l))]\) of~\(t\) at \(\rho(\s(l))\).
  Since~\(\rho\) is \(S\)\nb-equivariant, \(\rho(\s(l))\) belongs to
  the domain~\(D_{t^*t}\) of the partial homeomorphism on~\(Z\)
  given by~\(t\), so \([t,\rho(\s(l))]\) is a well defined arrow
  in~\(S\ltimes Z\).  We must also check that \(F^1(l)\) does not
  depend on the choice of~\(t\).  Let \(l\in L_t\cap L_u\).  The
  assumptions for an \(S\)\nb-grading give \(v\in S\) with \(v\le
  t,u\) and \(l\in L_v\).  Then \(\rho(\s(l))\in D_{v^*v}\), so
  \([t,\rho(\s(l))] = [v,\rho(\s(l))] = [u,\rho(\s(l))]\).  The
  map~\(F^1\) is continuous because its restriction to~\(L_t\) is
  continuous for each \(t\in S\).  It is compatible with range maps
  by the equivariance condition \(t\cdot \rho(\s(l)) =
  \rho(\rg(l))\).  Multiplicativity follows from \(L_tL_u\subseteq
  L_{tu}\), so~\(F\) is a continuous functor.

  If \((x,[t,z])\in L^0\times_{F^0,Z,\s} (S\ltimes Z)^1\), then
  \(\rho(x)=z\in D_{t^*t}\).  Since~\(\rho\) is \(S\)\nb-equivariant,
  this implies \(x\in \rho^{-1}(D_{t^*t}) = U_{t^*t} = \s(L_{t^*t}) =
  \s(L_t)\).  Thus there is \(l\in L_t\) with \(\s(l)=x\).  That is,
  the map
  \[
  L^1\xrightarrow{(\s,F^1)} L^0\times_{F^0,Z,\s} (S\ltimes Z)^1
  \]
  is surjective.

  This map is open because~\(S\ltimes Z\) is étale and \(\s\colon
  L^1\onto L^0\) is open.  We now prove this claim in detail.  It
  suffices to check that the restriction of~\((\s,F^1)\) to~\(L_t\)
  is open because the subsets \(L_t\subseteq L^1\) form an open
  cover of~\(L\).  The \(F^1\)\nb-image of~\(L_t\) is contained in
  the bisection of~\(S\ltimes Z\) associated to~\(t\).  Since the
  source map of the étale groupoid~\(S\ltimes Z\) restricts to a
  homeomorphism on any bisection, the map
  in~\eqref{eq:groupoid_fibration} restricted to~\(L_t\) is open if
  and only if \(\s\colon L_t\onto L^0\) is open; this is assumed for
  all topological groupoids.

  Next, we check that the \(S\)\nb-grading on~\(L\) associated to the
  functor~\(F\) is the given one; that is, \(F^1(l)\in t\) if and only
  if \(l\in L_t\).  By construction, if \(l\in L_t\) then \(F^1(l)\in
  t\).  Conversely, let \(l\in L\) satisfy \(F^1(l)\in L_t\).  There
  is \(u\in S\) with \(l\in L_u\); so \(F^1(l)=[u,\rho(\s(l))] \in
  t\).  Hence there is an idempotent \(e\in S\) with \(\rho(\s(l))\in
  D_e\) and \(te=ue\).  Then \(\s(l)\in U_e=\s(L_e)\) and hence \(l=l
  \cdot 1_{\s(l)}\in L_u\cdot U_e = L_{ue} = L_{te} = L_t\cdot U_e
  \subseteq L_t\) as desired.  In particular, \(F^1(l)\) is a unit,
  that is, belongs to the bisection \(1\in S\), if and only if \(l\in
  L_1=G\).  Thus~\(G\) is the fibre of the groupoid fibration~\(F\).

  We have now turned an \(S\)\nb-grading
  with a compatible map \(L^0\to Z\)
  into a fibration \(L\to S\ltimes Z\)
  and vice versa.  And we have checked that when we turn an
  \(S\)\nb-grading
  into a groupoid fibration and back, this gives the same
  \(S\)\nb-grading
  we started with.  Conversely, let us start with a groupoid fibration
  \(F\colon L\to S\ltimes Z\),
  turn it into an \(S\)\nb-grading
  \(L_t = (F^1)^{-1}(t)\)
  on~\(L\)
  with compatible map \(G^0/G\to Z\),
  and then construct a groupoid fibration~\(\tilde{F}\)
  from this.  The new groupoid fibration has the same map~\(F^0\)
  on objects, and it has \((\tilde{F}^1)^{-1}(t) = (F^1)^{-1}(t)\)
  for all \(t\in S\).
  This implies \(\tilde{F}^1=F^1\)
  because elements of~\(t\)
  are distinguished by their source or range.
\end{proof}

\begin{remark}
  \label{rem:openness_redundant}
  The proof of Theorem~\ref{the:S-action_groupoid_fibration} shows
  that the map in~\eqref{eq:groupoid_fibration} is open for any
  functor \(F\colon L\to H\) if~\(H\) is an étale groupoid.
  So~\(F\) is a groupoid fibration if and only if the map
  in~\eqref{eq:groupoid_fibration} is surjective.

  Thus the fibrations between \'etale groupoids defined
  in~\cite{Deaconu-Kumjian-Ramazan:Fell_groupoid_morphism} are the
  same as groupoid fibrations in our sense with the extra property
  that~\(F^1\) is an open surjection.  This is equivalent to~\(F^0\)
  being an open surjection by Lemma~\ref{lem:F1_F0_open_surjective}.
\end{remark}

\begin{example}
  \label{exa:linking-groupoid}
  Let~\(X\) be a groupoid equivalence between two topological
  groupoids \(G\) and~\(H\).  Its linking groupoid~\(L\) is a
  topological groupoid with unit space \(G^0\sqcup H^0\) and arrow
  space \(G^1\sqcup X\sqcup X^*\sqcup H^1\), where~\(X^*\) denotes the
  dual (or inverse) equivalence of~\(X\).  Here \(G\) acts on the left
  and~\(H\) on the right of~\(X\).  The groupoid structure is the
  canonical one involving the groupoid structures of \(G\) and~\(H\)
  and the structure of the equivalence bibundle~\(X\).  Let \(K\defeq
  \{a,b,\gamma,\gamma^{-1}\}\) be the pair groupoid on the
  \(2\)\nb-element set as in Example~\ref{exa:surjective_not_enough}.
  There is an obvious functor \(F\colon L\to K\), where~\(F^0\)
  maps~\(G^0\) to~\(a\) and~\(H^0\) to~\(b\) and \(F^1\colon
  L^1\to K^1\) maps~\(G^1\) to~\(a\), \(H^1\) to~\(a\), \(X\)
  to~\(\gamma^{-1}\) and~\(X^*\) to~\(\gamma\).  This is a groupoid
  fibration with fibre \(G\sqcup H\).  Both \(F^0\) and~\(F^1\) are
  open surjections.
\end{example}

\section{Locally Hausdorff and locally compact groupoids}
\label{sec:locally_Hausdorff_compact}

The main result in this section says that~\(L\) inherits certain
topological properties from \(G\) and~\(H\).  Furthermore, we relate
some properties of~\(H\) to the property of~\(G\) being open or closed
in~\(L\).

\begin{theorem}
  \label{the:locally_Hausdorff_compact}
  Let \(F\colon L\to H\) be a groupoid fibration with fibre~\(G\).  If
  both \(H\) and~\(G\) are Hausdorff or locally Hausdorff,
  respectively, then so is~\(L\).  If \(G\) and~\(H\) are locally
  Hausdorff and locally compact, then so is~\(L\).
\end{theorem}

\begin{proof}
  The proofs for the Hausdorff and locally Hausdorff case are the
  same, merely adding or removing the word ``locally'' where needed.
  The main tool is the following.  Let \(f\colon X\onto Y\)
  be a continuous open surjection.  Then~\(Y\)
  is (locally) Hausdorff if and only if
  \(\{(x_1,x_2)\in X\times X\mid f(x_1)=f(x_2)\}\)
  is (locally) closed in \(X\times X\) (see
  \cite{Buss-Meyer:Actions_groupoids}*{Proposition~2.15} and
  \cite{Meyer-Zhu:Groupoids}*{Proposition~9.18}).

  Assume first that \(G^i\) and~\(H^i\) are (locally) Hausdorff for
  \(i=0,1\).  Hence \(L^0=G^0\) is (locally) Hausdorff.  The
  space~\(L^1\) is (locally) Hausdorff if and only if the diagonal in
  \(L^1\times L^1\) is (locally) closed; this is the above criterion
  applied to~\(\Id_{L^1}\).

  Since~\(G^1\) is (locally) Hausdorff, the diagonal in \(G^1\times
  G^1\) is (locally) closed.  Its preimage under the continuous map
  \(G^1\to G^1\times G^1\), \(g\mapsto (g,1_{\s(g)})\), is
  \(G^0\subseteq G^1\).  Since preimages of (locally) closed subsets are
  (locally) closed, \(G^0\subseteq G^1\) is (locally) closed.  Then the
  preimage of \(G^0\subseteq G^1\) under the coordinate projection
  \(G^1\times_{\s,L^0,\rg} L^1\to G^1\) is also (locally) closed in
  \(G^1\times_{\s,L^0,\rg} L^1\).  The
  homeomorphism~\eqref{eq:fibration_principal_bundle} identifies this
  with the diagonal in~\(L^1\) as a subspace of the fibre product
  \(L^1\times_{H^1\times_{H^0} L^0} L^1\).

  The criterion above also applies to the open surjection
  in~\eqref{eq:groupoid_fibration}.  The space
  \(H^1\times_{\s,H^0,F^0} L^0\)
  is (locally) Hausdorff because it is a subspace of the (locally)
  Hausdorff space \(H^1\times L^0\).
  Hence \(L^1\times_{H^1\times_{H^0} L^0} L^1\)
  is (locally) closed in \(L^1\times L^1\).

  If~\(A\) is (locally) closed in~\(B\) and~\(B\) is (locally) closed
  in~\(C\), then~\(A\) is (locally) closed in~\(C\).  Thus the results
  of the previous two paragraphs together say that the diagonal is
  (locally) closed in \(L^1\times L^1\) as desired.

  Now assume that \(G^i\)
  and~\(H^i\)
  are locally Hausdorff and locally compact for \(i=0,1\).
  We have already seen that \(L^0\)
  and~\(L^1\)
  are locally Hausdorff.  And \(L^0=G^0\)
  is locally compact as well.  We must show that each
  \(l\in L^1\)
  has a compact, Hausdorff neighbourhood; then any neighbourhood
  of~\(l\) contains a compact, Hausdorff neighbourhood.  Let
  \(A\subseteq H^1\times_{H^0} L^0\)
  be a compact, Hausdorff neighbourhood of~\((F^1(l),\s(l))\).
  Then \(\hat{A}\defeq (F^1,\s)^{-1}(A)\)
  is a neighbourhood of~\(l\),
  so we may restrict attention to this subspace of~\(L^1\).
  Choose a compact, Hausdorff neighbourhood~\(B\)
  of~\(1_{\rg(l)}\)
  in~\(G^1\).
  The homeomorphism~\eqref{eq:fibration_principal_bundle} maps
  \(B\times_{\s,L^0,\rg}\hat{A}\)
  onto a neighbourhood of~\((l,l)\)
  in \(L^1\times_{H^1\times_{H^0} L^0} L^1\).
  Hence there is an open neighbourhood~\(C\)
  of~\(l\)
  in~\(\hat{A}\)
  such that all \((l_1,l_2)\in C\times C\)
  with \((F^1,\s)(l_1)=(F^1,\s)(l_2)\)
  have \(l_1l_2^{-1}\in B\).
  We may assume~\(C\)
  Hausdorff because we already know that~\(L^1\)
  is locally Hausdorff.  Since \((F^1,\s)\)
  is open, the subset \((F^1,\s)(C)\)
  is open in \(A\subseteq H^1\times_{H^0} L^0\),
  so it contains a compact neighbourhood~\(D'\)
  of \((F^1(l),\s(l))\).
  Since~\(A\)
  is Hausdorff, compact subsets are closed.  So
  \(D\defeq \{k\in C \mid (F^1,\s)(k)\in D'\}\)
  is relatively closed in~\(C\).

  Let~\((k_i)_{i\in I}\) be a net in~\(D\).  We claim that some subnet
  converges in~\(C\).  First, since~\(D'\) is compact, we can choose a
  subnet~\((k'_j)_{j\in J}\) such that the net \((F^1,\s)(k'_j)\)
  converges in~\(D'\).  To simplify notation, we assume that already
  \((F^1,\s)(k_i)\) converges.  Since the map
  in~\eqref{eq:groupoid_fibration} is open, we may lift
  \((F^1,\s)(k_i)\) to a convergent net~\((k'_j)_{j\in J}\) in~\(C\)
  (see \cite{Williams:crossed-products}*{Proposition~1.15}).  Lifting
  means that the index set~\(J\) maps to~\(I\) by a cofinal map and
  \((F^1,\s)(k_{i(j)}) = (F^1,\s)(k'_j)\).  Once again, we simplify
  notation by assuming that the net~\(k'\) is indexed by the same
  directed set~\(I\), so we have \((F^1,\s)(k_i) = (F^1,\s)(k'_i)\).
  Hence~\eqref{eq:fibration_principal_bundle} gives \(g_i\in G^1\)
  with \(g_i\cdot k'_i = k_i\).  Since \(k_i,k'_i\in C\), even
  \(g_i\in B\).  Since~\(B\) is compact, we may choose a convergent
  subnet of~\((g_i)\).  As before, we simplify notation by assuming
  that~\((g_i)\) itself converges.  Since the multiplication is
  continuous and \((g_i)\) and~\((k'_i)\) converge, it follows
  that~\((k_i)\) converges towards some limit point in~\(C\).

  Since~\(D\) is relatively closed in~\(C\), the limits of nets
  in~\(D\) that converge in~\(C\) belong to~\(D\).  Thus every net
  in~\(D\) has a convergent subnet, that is, \(D\) is compact.
  It is Hausdorff as well by construction.
\end{proof}

The following proposition describes when~\(H\) is Hausdorff.  This is
unrelated to \(L\) or~\(G\) being Hausdorff because in the example
after Lemma~\ref{lem:action_space_trafo-groupoid}, \(L\) and~\(G\) are
always Hausdorff, but~\(H\) is only locally Hausdorff.

\begin{proposition}
  \label{pro:condition-Hausdorff}
  Let \(F\colon L\to H\) be a fibration of topological groupoids with
  Hausdorff object spaces.  If~\(H\) is Hausdorff, then~\(G^1\) is
  closed in~\(L^1\).  Conversely, if~\(F^0\) is open and surjective
  and~\(G^1\) is closed in~\(L^1\), then~\(H\) is Hausdorff.
\end{proposition}

\begin{proof}
  The topological groupoid~\(H\) is Hausdorff if and only if its unit
  space~\(H^0\) is Hausdorff and closed in~\(H^1\) by
  \cite{Buss-Exel-Meyer:Reduced}*{Lemma~5.2}.  By definition, \(G^1\)
  is the inverse image of the units of~\(H\) under~\(F^1\).
  So~\(G^1\) is closed in~\(L^1\) if~\(H\) is Hausdorff.  Conversely,
  if \(F^1\) is open and surjective and~\(G^1\) is closed, then
  \(F^1(L^1\backslash G^1) = H^1\backslash H^0\) is open in~\(H^1\)
  because~\(F^1\) is open.  Since we assume~\(H^0\) to be Hausdorff,
  \cite{Buss-Exel-Meyer:Reduced}*{Lemma~5.2} shows that~\(H^1\) is
  Hausdorff.  Lemma~\ref{lem:F1_F0_open_surjective} shows that~\(F^0\)
  is an open surjection if and only if~\(F^1\) is.
\end{proof}

\begin{example}
  \label{exa:Hausdorffness-isotropy-quotient}
  Let \(G\into L\onto H=L/G\) be the extension (hence fibration)
  associated to an étale groupoid~\(L\) and its open isotropy
  subgroupoid~\(G\) as in~§\ref{sec:groupoid_extensions}.  By
  Proposition~\ref{pro:condition-Hausdorff}, \(H\)~is Hausdorff if and
  only if~\(G^1\) is closed in~\(L^1\) (compare
  \cite{Sims-Williams:primitive_ideals}*{Proposition~2.5}).  The
  space~\(L^1\) need not be Hausdorff for this to hold.
\end{example}

\begin{remark}
  \label{rem:H_etale}
  A groupoid~\(H\) is étale if and only if~\(H^0\) is open in~\(H^1\)
  (see \cite{Resende:Etale_groupoids}*{Theorem~1.4} and recall our
  standing assumption that \(\s\) and~\(\rg\) be open).  As in the
  proof of Proposition~\ref{pro:condition-Hausdorff}, this implies
  that~\(G^1\) is open in~\(L^1\) if~\(G\) is the fibre of a fibration
  \(F\colon L\to H\).  Conversely, if~\(F^0\) is an open surjection
  and~\(G^1\) is open in~\(L^1\), then \(F^1(G^1) = \{1_x \mid x\in
  H^0\}\) is open in~\(H^1\), so that~\(H\) is étale.
\end{remark}

\section{Haar systems and groupoid fibrations}
\label{sec:Haar}

Let \(F\colon L\to H\) be a groupoid fibration with fibre~\(G\).  We
assume that \(G\) and~\(H\) are locally Hausdorff, locally compact
groupoids with Hausdorff object spaces.  Then~\(L\) is also locally
Hausdorff and locally compact by
Theorem~\ref{the:locally_Hausdorff_compact}.

\begin{theorem}
  \label{the:inherit_Haar_measures}
  Let \((\lambda^x)_{x\in G^0}\)
  and~\((\mu^y)_{y\in H^0}\)
  be Haar systems on \(G\)
  and~\(H\).
  These induce a Haar system~\((\nu^x)_{x\in L^0}\) on~\(L\),
  given by
  \begin{equation}
    \label{eq:Haar_L}
    \int_{L^1} f(l)\, \dd\nu^x(l)
    = \int_{H^1} \int_{L^1} f(l)
    \,\dd \dot\lambda^{(h,x)}(l)\,\dd \mu^{F^0(x)}(h)
  \end{equation}
  for the continuous family of measures~\(\dot\lambda\) along the
  fibres of the open surjection \((F^1,\rg)\colon L^1\onto
  H^1\times_{\rg,H^0,F^0} L^0\) given by \(\int_L f(l) \,\dd
  \dot\lambda^{(h,x)}(l) = \int_G f(k g) \,\dd \lambda^{\s(k)}(g)\)
  for any \(k \in L^1\) with \(F^1(k)=h\) and \(\rg(k)=x\).
\end{theorem}

\begin{proof}
  The Haar systems for \(G\)
  and~\(H\)
  are continuous families of measures along the fibres of the maps
  \(\rg\colon G^1 \onto G^0\)
  and \(\rg\colon H^1 \onto H^0\).
  We need a continuous family of measures along the fibres of
  \(\rg\colon L^1 \onto L^0\).
  To use our data, we are going to factorise \(\rg\colon L^1 \onto L^0\)
  into two maps related to \(G\) and~\(H\).

  We apply the inversion in~\(L\) to the principal bundle in
  Proposition~\ref{pro:fibration_basic_action_GL} to see that the
  right action of~\(G\) on~\(L^1\) also gives a principal bundle,
  with bundle projection
  \begin{equation}
    \label{eq:Haar_1}
    (F^1,\rg)\colon L^1\onto H^1\times_{\rg,H^0,F^0} L^0.
  \end{equation}
  The Haar system for~\(G\) induces a continuous family of measures on
  the fibres of any principal \(G\)\nb-bundle,
  see~\cite{Renault:Representations}.  In our case, this gives the
  following family of measures~\(\dot\lambda\).  For \((h,x)\in
  H^1\times_{\s,H^0,F^0} L^0\), choose some
  \(k\in L^1\) with \(F^1(k)=h\) and \(\rg(k)=x\).
  Then the map \(G^{\s(k)} \congto (F^1,\rg)^{-1}(h,x)\),
  \(g\mapsto k\cdot g\), is a homeomorphism.  Thus we may
  transfer the measure~\(\lambda^{\s(k)}\) to a
  measure~\(\dot\lambda^{(h,x)}\) on this fibre.  This measure does
  not depend on the choice of~\(k\) because~\((\lambda^x)_{x\in
    G^0}\) is left invariant and~\(k\) is unique up to right
  multiplication by some \(g_0\in G^1\).  We are going to use
  \cite{Renault:Representations}*{Lemmas 1.2 and 1.3} to check that
  the family~\(\dot\lambda\) is continuous.

  First we pull back~\((\lambda^x)\)
  to the family of measures \(l\mapsto \lambda^{\s(l)}\)
  on the fibres of the map
  \(m\colon L^1\times_{\s,G^0,\rg} G^1\onto L^1\),
  \((l,g)\mapsto l\cdot g\).
  This family is continuous by \cite{Renault:Representations}*{Lemma
    1.2}.  The map~\(m\)
  is a \(G\)\nb-equivariant
  map if the \(G\)\nb-action
  on \(L^1\times_{\s,G^0,\rg} G^1\)
  is defined by \((l,g_1)\cdot g_2\defeq (l,g_1g_2)\)
  and on~\(L^1\)
  as usual.  Both \(G\)\nb-actions
  are parts of principal bundles, and the induced map on orbit spaces
  is the orbit space projection of~\(L^1\).
  \cite{Renault:Representations}*{Lemma 1.3} says that the induced
  family of measures for the orbit space projection \(L^1\onto L^1/G\)
  is also continuous.  More precisely, Renault defines principal
  bundles using free and proper actions; we use basic actions instead,
  that is, we only require a homeomorphism
  \(X\times_{\s,G^0,\rg} G^1 \congto X\times_{X/G} X\),
  \((x,g)\mapsto (x,x\cdot g)\).  What we call a principal bundle is one
  in Renault's notation if and only if the orbit space is Hausdorff by
  \cite{Meyer-Zhu:Groupoids}*{Corollary~9.35}.  Since
  \(H^1\times_{\s,H^0,F^0} L^0\)
  is locally Hausdorff, we may apply Renault's result to the
  restrictions of our principal bundle to all Hausdorff open subsets
  of \(H^1\times_{\s,H^0,F^0} L^0\).
  This is exactly the meaning of ``continuity'' for a family of
  measures over a locally Hausdorff base space.

  The groupoid~\(L\) acts on \(H^1\times_{\s,H^0,F^0} L^0\) with
  anchor map \((h,x)\mapsto x\) and multiplication \(l\cdot (h,\s(l))
  = (F^1(l)\cdot h,\rg(l))\).  The map \((F^1,\rg)\)
  in~\eqref{eq:Haar_1} is \(L\)\nb-equivariant.  The family of
  measures~\(\dot\lambda\) on the fibres of this map is
  \(L\)\nb-invariant because~\(\lambda\) is \(G\)\nb-invariant.

  Next we construct a continuous family of measures along the fibres
  of
  \begin{equation}
    \label{eq:Haar_2}
    H^1\times_{\rg,H^0,F^0} L^0 \onto L^0,\qquad (h,x)\mapsto x.
  \end{equation}
  We simply take the measure~\(\mu^{F^0(x)}\) on the fibre of~\(x\).
  This family is the pull-back of~\((\mu^y)_{y\in H^0}\) along \(F^0\colon
  L^0\to H^0\), so it is continuous by
  \cite{Renault:Representations}*{Lemma 1.2}.  The map
  in~\eqref{eq:Haar_2} is also \(L\)\nb-equivariant.  The family of
  measures~\((\mu^{F^0(x)})_{x\in L^0}\) is \(L\)\nb-invariant
  because~\((\mu^y)_{y\in H^0}\) is \(H\)\nb-invariant.

  Now we combine the continuous families of measures in
  \eqref{eq:Haar_1} and~\eqref{eq:Haar_2} to a continuous family of
  measures~\((\nu^x)_{x\in L^0}\)
  along the fibres of the composite map \(\rg\colon L^1\onto L^0\).
  The integral of a Borel function \(f\colon L^1\to\C\)
  with quasi-compact support against this family of measures is given
  by~\eqref{eq:Haar_L}.
  The integration over~\(\dot\lambda\) in~\eqref{eq:Haar_L}
  maps quasi-continuous functions on~\(L^1\)
  to quasi-continuous functions on \(H^1\times_{\rg,H^0,F^0} L^0\),
  and the second integration over~\(\mu\)
  maps quasi-continuous functions on \(H^1\times_{\rg,H^0,F^0} L^0\)
  to (quasi)continuous functions on~\(H^0\).
  Hence~\(\nu\) is a continuous family of measures.  The
  \(L\)\nb-invariance
  of \(\dot\lambda\)
  and~\(\mu\)
  implies that~\(\nu\)
  is \(L\)\nb-invariant.
  It is also clear that the support of~\(\nu^x\) is all of~\(L^x\).
\end{proof}

\begin{example}
  Our theorem applies, in particular, to extensions of locally
  Hausdorff and locally compact groupoids \(G\into L\onto H\)
  with the same unit space \(G^0=L^0\congto H^0\).  For an extension
  of locally compact groups \(G\into L\onto H=L/G\),
  Theorem~\ref{the:inherit_Haar_measures} gives the usual formula for
  the Haar measures on~\(L\) in terms of the Haar measures on \(G\)
  and~\(H\).  For a twist, that is, \(G=L^0\times \Torus\) with
  trivial conjugation action of~\(L\) on~\(\Torus\), it is already
  observed in~\cite{Muhly-Williams:Continuous-traceII} that a Haar
  system on~\(H\) induces one on~\(L\).  Implicitly, this uses
  the normalised Haar measure on the compact group~\(\Torus\).

  For the transformation groupoid (or semidirect product) \(L=H\ltimes
  G\) of a classical action of a locally compact group~\(H\) on a
  Hausdorff, locally compact groupoid~\(G\), the existence of the Haar
  system on~\(L\) is proved in
  \cite{Kaliszewski-Muhly-Quigg-Williams:Coactions_Fell}*{Proposition~6.4}
  under an extra ``invariance'' condition for the \(H\)\nb-action
  on~\(G\).  This condition rules out some basic examples such as the
  \(ax+b\)-group \(\R\ltimes \R_{>0}\).  It is not necessary by
  Theorem~\ref{the:inherit_Haar_measures}, which needs no condition on
  the Haar systems of \(G\) and~\(H\) and even allows \(G\) and~\(H\)
  to be locally Hausdorff.
\end{example}

\section{Crossed products}
\label{sec:crossed}

As before, let \(F\colon L\to H\) be a groupoid fibration with
fibre~\(G\).  We assume that \(G\) and~\(H\) are locally Hausdorff,
locally compact groupoids with Hausdorff object spaces and with Haar
systems \(\lambda\) and~\(\mu\), respectively.  Then~\(L\) is locally
Hausdorff and locally compact by
Theorem~\ref{the:locally_Hausdorff_compact}.  Let~\(\nu\) be the
canonical Haar system on~\(L\) constructed in
Theorem~\ref{the:inherit_Haar_measures}.   The available
disintegration theory of representations also requires that all our
groupoids are second countable and all Fell bundles separable, but
these assumptions may probably be removed with better technology, on
which we are working at the moment.

Let~\(\Banb\)
be a Fell bundle over~\(L\).
By convention, \emph{all Fell bundles are separable, saturated and upper
semicontinuous}.
As in~\cite{Buss-Meyer:Actions_groupoids}, we denote the space of
quasi-continuous sections of~\(\Banb\)
on~\(L\)
by~\(\Sect(L,\Banb)\);
these are finite linear combinations of compactly supported continuous
sections on Hausdorff open subsets of~\(L\),
extended by~\(0\)
outside.
With standard formulas for convolution and involution,
\(\Sect(L,\Banb)\) becomes a \Star{}algebra.  It carries a canonical
bornology, that is, a collection of bounded subsets such that the
convolution and involution are bounded, see
\cite{Buss-Meyer:Actions_groupoids}*{Appendix~B}.  We call a seminorm
or representation on~\(\Sect(L,\Banb)\) bounded if it is bounded on
all bounded subsets.  The so-called inductive limit topology is the
locally convex topology generated by the bounded seminorms.

The \emph{section \(\Cst\)\nb-algebra} \(\Cst(L,\Banb)\) for the Fell
bundle \(\Banb\to L\) is the completion of~\(\Sect(L,\Banb)\) for the
maximal bounded \(\Cst\)\nb-seminorm on~\(\Sect(L,\Banb)\) or,
equivalently, for the maximal \(\Cst\)\nb-seminorm that is continuous
in the inductive limit topology on~\(\Sect(L,\Banb)\).

\emph{We assume that the Disintegration Theorem holds for~\(\Banb\).}
It says that any bounded representation of~\(\Sect(L,\Banb)\) comes
from a representation of the Fell bundle~\(\Banb\).  This implies that
any bounded \(\Cst\)\nb-seminorm is already dominated by a certain
norm called the \(I\)\nb-norm; hence the maximum of all bounded
\(\Cst\)\nb-norms on~\(\Sect(L,\Banb)\) exists.

We may restrict~\(\Banb\) to a Fell bundle over \(G\subseteq L\) and
get a bornological \Star{}algebra \(\Sect(G,\Banb|_G) =
\Sect(G,\Banb)\) and a \(\Cst\)\nb-algebra \(\Cst(G,\Banb|_G)=
\Cst(G,\Banb)\) in the same way.  Similarly, we may restrict to the
subgroupoids \(G_y\subseteq G\) for \(y\in H^0\)
(see~§\ref{sec:equivalences_from_fibration}) and construct
\(\Cst\)\nb-algebras \(\Cst(G_y,\Banb)\).

\emph{We also assume that the Equivalence Theorem holds for
  \(\Cst(G_y,\Banb)\), that is, the groupoid equivalences~\(L_h\) in
  Lemma~\textup{\ref{lem:Lh_locally_closed}} induce Morita--Rieffel
  equivalences between \(\Cst(G_{\rg(h)},\Banb)\) and
  \(\Cst(G_{\s(h)},\Banb)\).}

The Disintegration Theorem and the Equivalence Theorem have been shown
for several classes of Fell bundles over groupoids: for Green twisted
actions of non-Hausdorff groupoids on continuous fields of
\(\Cst\)\nb-algebras over~\(L^0\) in~\cite{Renault:Representations};
for arbitrary (separable and saturated) upper semicontinuous Fell
bundles over Hausdorff groupoids
in~\cite{Muhly-Williams:Equivalence.FellBundles}; and for ordinary
actions of non-Hausdorff groupoids in
\cite{Muhly-Williams:Renaults_equivalence}.  So far, there seems to be
no source that covers Fell bundles and non-Hausdorff groupoids
simultaneously.  We are working on proving these results in this
generality; for now, we must assume these basic technical
results to hold for our proofs to work.

We now construct a pre-Fell bundle over~\(H\), which we will later
complete to a Fell bundle using appropriate \(\Cst\)\nb-norms.  For
\(h\in H^1\), let \(L_h\defeq (F^1)^{-1}(h)\subseteq L^1\) as in
Lemma~\ref{lem:Lh_locally_closed}.  Since~\(H\) is locally Hausdorff,
points in~\(H^1\) are closed.  So~\(L_h\) is closed in~\(L^1\) and
hence locally Hausdorff and locally compact.  Therefore, the space
\(\FellH_h\defeq \Sect(L_h,\Banb)\) of quasi-continuous sections
\(L_h\to \Banb\) is well defined.  We define an involution
\[
\FellH_h\to \FellH_{h^{-1}},\qquad
f^*(l) \defeq \conj{f(l^{-1})},
\]
using that the inversion map restricts to a homeomorphism from~\(L_h\)
to~\(L_{h^{-1}}\) by Lemma~\ref{lem:Lh_locally_closed}.  This map is
bounded, conjugate-linear, and
involutive, that is, \(f^{**} = f\).  We want to define a convolution
\(\FellH_{h_1}\times \FellH_{h_2} \to \FellH_{h_1h_2}\) for
\(h_1,h_2\in H\) by
\begin{equation}
  \label{eq:convolution_FellH}
  (f_1*f_2)(l) \defeq
  \int_{L^1} f_1(l_1)\cdot f_2(l_1^{-1}l) \dd\dot\lambda^{(h_1,\rg(l))}(l_1)
\end{equation}
for \(f_1\in\FellH_{h_1}\), \(f_2\in\FellH_{h_2}\).  We explain why
this formula works.  We integrate over the \(L\)\nb-invariant family
of measures~\(\dot\lambda\) on the fibres of the map
\[
(F^1,\rg)\colon L^1\onto H^1\times_{\rg,H^0,F^0} L^0
\]
as in Theorem~\ref{the:inherit_Haar_measures}.  The support
of~\(\dot\lambda^{(h_1,\rg(l))}\) is the set of all \(l_1\in L^1\)
with \(F^1(l_1)=h_1\) and \(\rg(l_1)=\rg(l)\) or, equivalently,
\(l_1\in L_{h_1}\) with \(\rg(l_1)=\rg(l)\).  Then
\(F^1(l_1^{-1}l) = h_1^{-1}h = h_2\), so the integral
in~\eqref{eq:convolution_FellH} only sees values of~\(f_1\)
on~\(L_{h_1}\) and of~\(f_2\) on~\(L_{h_2}\), respectively.  The
product \(f_1(l_1)\cdot f_2(l_1^{-1}l)\) belongs to \(\Banb_{l_1}\cdot
\Banb_{l_1^{-1}l} \subseteq\Banb_l\).  Thus \((f_1*f_2)(l)\) is a well
defined element of~\(\Banb_l\).

The same argument as for the convolution in~\(\Sect(L,\Banb)\) shows
that \(f_1*f_2\) is quasi-continuous, that is, belongs to
\(\Sect(L_{h_1 h_2},\Banb)\) (see
\cite{Muhly-Williams:Renaults_equivalence}*{Proposition~4.4}).  It
suffices to prove this if both \(f_1\) and~\(f_2\) are compactly
supported continuous functions on some Hausdorff open subsets of
\(L_{h_1}\) and~\(L_{h_2}\), respectively.  We may further use
partitions of unity to decompose \(f_1\) and~\(f_2\) into functions
with smaller supports, so that the product of the Hausdorff open
subsets on which \(f_1\) and~\(f_2\) live is again Hausdorff
in~\(L_{h_1h_2}\).  Then the continuity of~\(\dot\lambda\) implies
that~\(f_1*f_2\) is continuous with
compact support on the relevant Hausdorff open subset
of~\(L_{h_1h_2}\).  Since we may choose the same partition of unity
for all functions \(f_1,f_2\) with a given support, the convolution is
bounded; that is, if~\(f_i\) for \(i=1,2\) run through bounded subsets
of \(\Sect(L_{h_i},\Banb)\), then the set of products \(f_1*f_2\) is
bounded.

The convolution in~\eqref{eq:convolution_FellH} is bilinear.  It is
associative because~\(\dot\lambda\) is \(L\)\nb-invariant:
\begin{align*}
  (f_1*f_2)*f_3(l)
  &= \int (f_1*f_2)(l_2) f_3(l_2^{-1} l)
  \dd\dot\lambda^{(h_1h_2,\rg(l))}(l_2)
  \\ &= \iint f_1(l_1) f_2(l_1^{-1}l_2) f_3(l_2^{-1} l)
  \dd\dot\lambda^{(h_1,\rg(l_2))}(l_1)
  \dd\dot\lambda^{(h_1h_2,\rg(l))}(l_2),\\
  f_1*(f_2*f_3)(l)
  &= \int f_1(l_1) (f_2*f_3)(l_1^{-1} l)
  \dd\dot\lambda^{(h_1,\rg(l))}(l_1)
  \\ &= \iint f_1(l_1) f_2(l_2) f_3(l_2^{-1} l_1^{-1} l)
  \dd\dot\lambda^{(h_2,\s(l_1))}(l_2)
  \dd\dot\lambda^{(h_1,\rg(l))}(l_1)
  \\&= \iint f_1(l_1) f_2(l_1^{-1}l_2) f_3(l_2^{-1} l)
  \dd\dot\lambda^{(h_1h_2,\rg(l_1))}(l_2)
  \dd\dot\lambda^{(h_1,\rg(l))}(l_1).
\end{align*}
The identity \(f_1^**f_2^* = (f_2*f_1)^*\) follows from the
\(L\)\nb-invariance of~\(\dot\lambda\):
\begin{multline*}
  (f_1^**f_2^*)(l)
  = \int f_1^*(l_1) f_2^*(l_1^{-1}l) \dd\dot\lambda^{(h_1,\rg(l))}(l_1)
  = \int \conj{f_1(l_1^{-1}) f_2(l^{-1}l_1)}
  \dd\dot\lambda^{(h_1,\rg(l))}(l_1)
  \\= \int \conj{f_1(l_1^{-1}l^{-1}) f_2(l_1)}
  \dd\dot\lambda^{(h_1,\s(l))}(l_1)
  = (f_2*f_1)^*(l).
\end{multline*}

Instead of a topology on the bundle~\((\FellH_h)_{h\in H^1}\), we
specify its space of quasi-continuous sections.  For a Hausdorff, open
subset \(U\subseteq H^1\), let \(V\defeq (F^1)^{-1}(U)\), which is an
open subset of~\(L^1\).  To a quasi-continuous section \(\xi\in
\Sect(V,\Banb)\), we assign a section~\(\tilde\xi\)
of~\((\FellH_h)_{h\in H^1}\) over~\(U\) by \(\tilde\xi(h)\defeq
\xi|_{L_h}\).  We take this as the space of continuous sections with
compact support of the bundle~\((\FellH_h)_{h\in H^1}\) over the
Hausdorff subset~\(U\).  The space of quasi-continuous sections
of~\((\FellH_h)_{h\in H^1}\) over an arbitrary open subset
\(U\subseteq H^1\) is defined as for Banach bundles, as the space of
finite linear combinations of continuous sections over Hausdorff open
subsets.  This is canonically isomorphic to
\(\Sect((F^1)^{-1}(U),\Banb)\) by
\cite{Buss-Meyer:Actions_groupoids}*{Proposition B.2}.  This finishes
the construction of the pre-Fell bundle~\((\FellH_h)_{h\in H^1}\)
over~\(H\).

The assignment \(\xi\mapsto \tilde\xi\) used above also preserves the
algebraic structure, that is, we have \(\widetilde{\xi*\eta} =
\tilde\xi*\tilde\eta\) and \((\tilde\xi)^* = \widetilde{\xi^*}\) for
all \(\xi,\eta\in \Sect(L,\Banb)\).  The equality \((\tilde\xi)^* =
\widetilde{\xi^*}\) follows easily from the definitions of the
involutions.  The equality \(\widetilde{\xi*\eta} =
\tilde\xi*\tilde\eta\) uses how the Haar system~\(\nu\) in
Theorem~\ref{the:inherit_Haar_measures} is built.  Let \(h\in H^1\).
Then \(\widetilde{\xi*\eta}(h)\) is the restriction of~\(\xi*\eta\) to
\(L_h\subseteq L^1\).  And for \(l\in L_h\), that is, \(l\in L^1\) with
\(F^1(l)=h\), this equals
\begin{multline*}
  \bigl( \widetilde{\xi*\eta}(h)\bigr)(l)
  = \xi*\eta(l)
  = \int_L \xi(l_1)\eta(l_1^{-1}l)\,\dd\nu^{\rg(l)}(l_1)
  \\= \int_H \int_L \xi(l_1)\eta(l_1^{-1}l)
  \,\dd\dot\lambda^{(x,\rg(l))}(l_1)\,\dd\mu^{F^0(\rg(l))}(x).
\end{multline*}
The assumption \(F^1(l)=h\) gives \(F^0(\rg(l))=\rg(h)\), so that this
is equal to the section \((\tilde\xi*\tilde\eta)(h) \in
\Sect(L_h,\Banb)\) evaluated at the point~\(l\)
via~\eqref{eq:convolution_FellH}; hence \(\widetilde{\xi*\eta} =
\tilde\xi*\tilde\eta\).  Therefore, the map \(\xi\mapsto \tilde\xi\)
gives a \Star{}algebra isomorphism \(\Sect(L,\Banb)\congto
\Sect(H,\FellH)\).

Next we complete~\((\FellH_h)_{h\in H^1}\) to a
\(\Cst\)\nb-algebraic Fell bundle~\((\bar{\FellH}_h)_{h\in H^1}\).
The groupoids~\(G_y\) inherit Haar systems from~\(G\)
by restricting~\((\lambda^x)\)
to \(x\in (F^0)^{-1}(y)\).
Our space~\(\FellH_{1_y}\)
is exactly the space~\(\Sect(G_y,\Banb)\)
of quasi-continuous sections of the restriction of~\(\Banb\)
to~\(G_y\).
Thus we may complete~\(\Sect(G_y,\Banb)\)
to a \(\Cst\)\nb-algebra
using the maximal bounded \(\Cst\)\nb-norm.

The \(\Cst\)\nb-algebras
\(\Cst(G_y,\Banb)\)
for \(y\in H^0\)
are the fibres of an upper semicontinuous field on~\(H^0\)
with section algebra \(\Cst(G,\Banb)\) or, more precisely,
\(\Cst(G,\Banb|_G)\).
To construct this field, we use the \(G\)\nb-invariant
function \(F^0\colon G^0 = L^0 \to H^0\).
The resulting non-degenerate \Star{}homomorphism
\((F^0)^*\colon \Cont_0(H^0)\to \Mult(\Cst(G,\Banb))\)
takes values in the centre because~\(F^0\)
is \(G\)\nb-invariant.
Thus it turns \(\Cst(G,\Banb)\) into a
\(\Cont_0(H^0)\)-\(\Cst\)\nb-algebra.  The fibre at \(y\in H^0\) for
this \(\Cont_0(H^0)\)-\(\Cst\)\nb-algebra structure
is~\(\Cst(G_y,\Banb)\) because
\(\Sect(G,\Banb)/\Cont_0(H^0\setminus\{y\})\cdot \Sect(G,\Banb) \cong
\Sect(G_y,\Banb)\).

Recall that~\(L_h\) is an equivalence between the groupoids
\(G_{\rg(h)}\) and~\(G_{\s(h)}\).  The formulas we used to define the
pre-Hilbert bimodule \(\FellH_h=\Sect(L_h,\Banb)\) are the usual ones
for an equivalence of Fell bundles.  By assumption, Renault's
Equivalence Theorem holds for the groupoids~\(G_y\) for all \(y\in
H^0\), see \cite{Renault:Representations}*{Corollaire 5.4} and
\cite{Muhly-Williams:Renaults_equivalence}*{Theorem 5.5}.  Hence we
may complete~\(\FellH_h\) to an imprimitivity bimodule
\(\bar{\FellH}_h = \Cst(L_h,\Banb)\) between
\(\Cst(G_{\rg(h)},\Banb)\) and~\(\Cst(G_{\s(h)},\Banb)\).  The
convolution products \(\FellH_{h_1}\times\FellH_{h_2}\to
\FellH_{h_1h_2}\) and the involutions \(\FellH_h\to \FellH_{h^{-1}}\)
defined in~\eqref{eq:convolution_FellH} extend to the
completions~\(\bar\FellH_h\).  The proof uses the \cstar{}identity and
that these convolution products are the left or right bimodule actions
on~\(\FellH_h\) if restricted to units.  Hence the spaces~\(\FellH_h\)
for \(h\in H\) are the fibres of a (saturated) Fell
bundle~\(\bar{\FellH}\) over~\(H\) if we specify a suitable space of
quasi-continuous sections.  This is done via the previously defined
map \(\xi\mapsto \tilde\xi\), which identifies \(\Sect(L,\Banb)\cong
\Sect(H,\FellH)\).  This determines a unique topology on the
bundle~\(\bar\FellH\) with fibres~\(\FellH_h\) and turns it into a
Fell bundle (see
\cite{BussExel:Fell.Bundle.and.Twisted.Groupoids}*{Propositions 2.4
  and~2.7}).

By construction, the map \(\xi\mapsto \tilde\xi\) identifies
\(\Sect(L,\Banb)\) with the dense \Star{}subalgebra
\(\Sect(H,\FellH)\) of the section \(\Cst\)\nb-algebra
\(\Cst(H,\bar{\FellH})\).  This extends to an isomorphism
\(\Cst(L,\Banb)\congto \Cst(H,\FellH)\) of \cstar{}algebras:

\begin{theorem}[Iterated crossed-product decomposition]
  \label{the:sections_full}
  Let \(L\) and \(H\) be second countable, locally Hausdorff and
  locally compact groupoids and let \(F\colon L\to H\) be a groupoid
  fibration with fibre~\(G\).  Assume that \(G\) and~\(H\) carry Haar
  systems, and endow~\(L\) with the Haar system constructed in
  Theorem~\textup{\ref{the:inherit_Haar_measures}}.  Let~\(\Banb\) be
  a saturated, separable Fell bundle over~\(L\).
  Then~\(\Cst(L,\Banb)\) is isomorphic to the section
  \(\Cst\)\nb-algebra of a saturated Fell bundle over~\(H\) with
  fibres \(\Cst(L_h,\Banb)\) at \(h\in H^1\) and unit
  fibre~\(\Cst(G,\Banb)\).  In particular, \(\Cst(L)\) is isomorphic
  to the section \cstar{}algebra of a saturated Fell bundle over~\(H\)
  with fibres~\(\Cst(L_h)\) at \(h\in H^1\) and unit
  fibre~\(\Cst(G)\).
\end{theorem}

Before we prove this theorem, we interpret it.  We view a fibration
\(F\colon L\to H\) with fibre~\(G\) as a continuous action by groupoid
equivalences of~\(H\) on~\(G\) with transformation groupoid~\(L\),
see~§\ref{sec:equivalences_from_fibration}.  As in
\cites{Buss-Meyer:Crossed_products, Buss-Meyer:Actions_groupoids,
  Buss-Meyer-Zhu:Higher_twisted}, we view a saturated Fell
bundle~\(\Banb\) over~\(L\) as an action (by \cstar{}algebra
equivalences) of~\(L\) on the \cstar{}algebra \(A\defeq
\Cst(L^0,\Banb)\) of the restriction of~\(\Banb\) to~\(L^0\), which we
called \emph{unit fibre} above.  We view the section \cstar{}algebra
\(\Cst(L,\Banb)\) as the ``crossed product'' \(A\rtimes L\) of this
action.  We may restrict the action (that is, the Fell bundle)
from~\(L\) to \(G\subseteq L\).  Theorem~\ref{the:sections_full} says
that there is a new action (in the form of a Fell bundle) of~\(H\) on
\(A\rtimes G=\Cst(G,\Banb)\) such that
\[
(A\rtimes G)\rtimes H\cong A\rtimes L.
\]

We begin to prove Theorem~\ref{the:sections_full}.  The proof will be
finished after Lemma~\ref{lem:norm_estimate}.

The \(\Cst\)\nb-algebra~\(\Cst(L,\Banb)\) is the completion
of~\(\Sect(L,\Banb)\) for the maximal bounded \(\Cst\)\nb-seminorm
on~\(\Sect(L,\Banb)\).  Similarly, the section \(\Cst\)\nb-algebra of
the Fell bundle~\(\bar{\FellH}\) over~\(H\) is the completion of
\(\Sect(H,\bar{\FellH})\) in the maximal bounded \(\Cst\)\nb-seminorm.
By construction, \(\Sect(L,\Banb)\) is a dense \Star{}subalgebra in
\(\Sect(H,\bar{\FellH})\), and the inclusion map is bounded.  It
remains to prove that this dense inclusion extends to an isomorphism
between the \(\Cst\)\nb-completions.  Equivalently, any bounded
\Star{}representation of~\(\Sect(L,\Banb)\) on a Hilbert space extends
to a bounded \Star{}representation of \(\Sect(H,\bar{\FellH})\).  This
result looks plausible, but the proof is rather technical, and it took
us some time to finish this argument.

We prove that any bounded Hilbert space representation
of~\(\Sect(L,\Banb)\) is bounded in norm by a variant of the
\(I\)\nb-norm on \(\Sect(H,\bar{\FellH})\).  This is
non-trivial because of the \(\Cst\)\nb-completion~\(\bar{\FellH}\)
of~\(\FellH\) in the direction
of~\(G\).  A first step is to construct a morphism
\(\Cst(G,\Banb)\to\Cst(L,\Banb)\), that is, a nondegenerate
\Star{}homomorphism from~\(\Cst(G,\Banb)\) to the multiplier
\(\Cst\)\nb-algebra of~\(\Cst(L,\Banb)\).

\begin{lemma}
  \label{lem:left_action_G_L}
  Define
  \[
  (\varphi*\psi)(l) \defeq
  \int_G \varphi(g)\psi(g^{-1} l)\,\dd\lambda^{\rg(l)}(g)
  \]
  for \(\varphi\in\Sect(G,\Banb)\), \(\psi\in\Sect(L,\Banb)\), \(l\in
  L\), where~\(\lambda\) denotes the Haar system on~\(G\).  This
  defines a \Star{}homomorphism from \(\Sect(G,\Banb)\) into the
  multiplier \Star{}algebra of \(\Sect(L,\Banb)\), that is, the
  convolution is bilinear and satisfies \((\varphi_1 *\varphi_2) *\psi
  = \varphi_1* (\varphi_2 *\psi)\), \(\varphi * (\psi_1 *\psi_2) =
  (\varphi * \psi_1) *\psi_2\), and \(\psi_1^* * (\varphi *\psi_2) =
  (\varphi^**\psi_1)^* *\psi_2\) for
  \(\varphi,\varphi_1,\varphi_2\in\Sect(G,\Banb)\) and
  \(\psi,\psi_1,\psi_2\in\Sect(L,\Banb)\).
\end{lemma}

\begin{proof}
  Bilinearity is trivial, and \((\varphi_1 *\varphi_2) *\psi =
  \varphi_1* (\varphi_2 *\psi)\) and \(\varphi * (\psi_1 *\psi_2) =
  (\varphi * \psi_1) *\psi_2\) follow from the left invariance of the
  Haar systems on \(G\) and~\(L\).  We prove \(\psi_1^* * (\varphi
  *\psi_2) = (\varphi^**\psi_1)^* *\psi_2\) in detail.  Let \(l\in
  L\).  Let~\(\dot\lambda\) and the Haar systems \(\lambda\), \(\mu\)
  and~\(\nu\) on \(G\), \(H\) and~\(L\) be as in
  Theorem~\ref{the:inherit_Haar_measures}.  By definition,
  \begin{multline*}
    \psi_1^* * (\varphi *\psi_2) (l) =
    \iint
    \psi_1^*(l_2) \varphi(g) \psi_2(g^{-1} l_2^{-1} l)
    \,\dd\lambda^{\s(l_2)}(g) \,\dd\nu^{\rg(l)}(l_2)\\
    = \iiint
    \psi_1(l_2^{-1})^* \varphi(g) \psi_2(g^{-1} l_2^{-1} l)
    \,\dd\lambda^{\s(l_2)}(g) \,\dd\dot\lambda^{(h,\rg(l))}(l_2)
    \,\dd\mu^{F^0(\rg(l))}(h),
  \end{multline*}
  \begin{multline*}
    (\varphi^* * \psi_1)^* * \psi_2 (l) =
    \iint
    \psi_1(\tilde{g}^{-1} \tilde{l}_2^{-1})^* \varphi(\tilde{g}^{-1})
    \psi_2(\tilde{l}_2^{-1} l)
    \,\dd\lambda^{\s(\tilde{l}_2)}(\tilde{g})
    \,\dd\nu^{\rg(l)}(\tilde{l}_2)\\
    = \iiint
    \psi_1((\tilde{l}_2\tilde{g})^{-1})^* \varphi(\tilde{g}^{-1})
    \psi_2(\tilde{l}_2^{-1} l)
    \,\dd\lambda^{\s(\tilde{l}_2)}(\tilde{g})
    \,\dd\dot\lambda^{(h,\rg(l))}(\tilde{l}_2)
    \,\dd\mu^{F^0(\rg(l))}(h).
  \end{multline*}
  The equality of these two triple integrals is proved by three
  substitutions for fixed~\(h\).  First, we use the homeomorphism
  \(L_h^{\rg(l)}\times_{\s,L^0,\rg} G \congto L_h^{\rg(l)} \times
  L_h^{\rg(l)}\), \((l_2,g)\mapsto (l_2,l_2 g)\).  Secondly, we use
  the homeomorphism \(L_h^{\rg(l)} \times L_h^{\rg(l)} \congto G
  \times_{\rg,L^0,\s} L_h^{\rg(l)}\), \((l_2,l_3)\mapsto (l_3^{-1}
  l_2,l_3)\), which maps \((l_2,l_2 g)\mapsto (g^{-1},l_2 g)\).
  Finally, we use the coordinate flip \(G \times_{\rg,L^0,\s}
  L_h^{\rg(l)} \congto L_h^{\rg(l)} \times_{\s,L^0,\rg} G\),
  \((\tilde{g},\tilde{l}_2) \mapsto(\tilde{l}_2,\tilde{g})\).
  Altogether, these substitutions take \((l_2,g)\mapsto
  (\tilde{l}_2,\tilde{g})\) with \(\tilde{l}_2 \defeq l_2 g\) and
  \(\tilde{g} \defeq g^{-1}\).  Since the family of
  measures~\(\dot\lambda\) is left \(L\)\nb-invariant and
  \(\dot\lambda^{1,x} = \lambda^x\), the measure
  \(\lambda^{\s(l_2)}\times \dot\lambda^{(h,\rg(l))}\) on
  \(L_h^{\rg(l)}\times_{\s,L^0,\rg} G\) is transformed by these
  substitutions first to \(\dot\lambda^{(h,\rg(l))}\times
  \dot\lambda^{(h,\rg(l))}\), secondly to \(\lambda^{\s(l_2))}\times
  \dot\lambda^{(h,\rg(l))}\) on \(G \times_{\rg,L^0,\s}
  L_h^{\rg(l)}\), and finally to \(\dot\lambda^{(h,\rg(l))} \times
  \lambda^{\s(l_2)}\) on \(L_h^{\rg(l)} \times_{\s,L^0,\rg} G\).
  Hence the substitutions above identify the two triple integrals
  expressing \(\psi_1^* * (\varphi *\psi_2) (l)\) and \((\varphi^* *
  \psi_1)^* * \psi_2 (l)\).
\end{proof}

To construct a morphism \(\Cst(G,\Banb)\to\Mult(\Cst(L,\Banb))\) from
the above lemma, we must extend multipliers defined only
on~\(\Sect(L,\Banb)\) to~\(\Cst(L,\Banb)\).  This can be done assuming
a stronger form of the Disintegration Theorem for densely defined
representations.  To reduce the number of technical assumptions, we
construct this morphism only for the special case of a trivial Fell
bundle, where the relevant technical result is already proved.  This
special case provides some information on quasi-invariant measures and
modular functions which, together with the Disintegration Theorem for
representations by globally defined bounded operators, gives the
morphism \(\Cst(G,\Banb)\to\Mult(\Cst(L,\Banb))\) also for Fell bundle
coefficients.

\begin{lemma}
  \label{lem:G_to_L_no_Fell}
  There is a unique non-degenerate \Star{}homomorphism \(\iota\colon
  \Cst(G)\to\Mult(\Cst(L))\) such that \(\iota(\phi)\psi = \phi*\psi\)
  for \(\phi\in\Sect(G)\), \(\psi\in\Sect(L)\).
\end{lemma}

\begin{proof}
  Since~\(\Cst(L)\) is separable, it has a faithful representation
  \(\pi\colon \Cst(L)\to\Bound(\Hils)\) on a separable Hilbert
  space~\(\Hils\).  The extension of~\(\pi\) to~\(\Mult(\Cst(L))\)
  remains faithful.  We want to define a densely defined
  representation~\(\pi_G\) of~\(\Sect(G)\) on~\(\Hils\) by
  \(\pi_G(\phi) \xi = \sum_{i=1}^n \pi(\phi*\psi_i)\xi_i\) if
  \(\phi\in\Sect(G)\) and \(\xi = \sum_{i=1}^n \pi(\psi_i)\xi_i\) with
  \(\psi_i\in\Sect(L)\), \(\xi_i\in\Hils\) for \(i=1,\dotsc,n\).  This
  is well defined, that is, \(\pi_G(\phi)\xi\) does not depend on how
  we decompose~\(\xi\) because
  \[
  \biggl< \sum_{j=1}^m \pi(\tau_j)\eta_j \biggm|
  \sum_{i=1}^n \pi(\phi*\psi_i)\xi_i \biggr>
  = \biggl< \sum_{j=1}^m \pi(\phi^**\tau_i)\eta_i \biggm|
  \sum_{i=1}^n \pi(\psi_i)\xi_i \biggr>
  \]
  for any \(\tau_j\in\Sect(L)\), \(\eta_j\in\Hils\).  The right hand
  side no longer depends on the decomposition of~\(\xi\) and
  determines \(\sum_{i=1}^n \pi(\phi*\psi_i)\xi_i\) because vectors of
  the form \(\sum_{j=1}^m \pi(\tau_j)\eta_j\) are dense in~\(\Hils\).

  The densely defined representation~\(\pi_G\) of~\(\Sect(G)\)
  on~\(\Hils\) satisfies the assumptions in
  \cite{Renault:Representations}*{Proposition 4.2}.  Hence the
  operators~\(\pi_G(\phi)\) extend to bounded operators on~\(\Hils\)
  that define a \Star{}homomorphism \(\pi_G\colon
  \Cst(G)\to\Bound(\Hils)\).  Since \(\pi_G(\phi)\pi(\psi) =
  \pi(\phi*\psi)\) for all \(\phi\in\Sect(G)\), \(\psi\in\Sect(L)\)
  and \(\pi_G(\Cst(G))\) is a \Star{}algebra, it is contained in
  \(\Mult(\Cst(L))\subseteq \Bound(\Hils)\).
\end{proof}

Let \(\pi\colon \Cst(L,\Banb)\to\Bound(\Hils)\) be a faithful
representation on a separable Hilbert space~\(\Hils\).  The
Disintegration Theorem gives an \(L\)\nb-quasi-invariant
measure~\(\alpha\) on~\(L^0\), an \(\alpha\)\nb-measurable field of
Hilbert spaces~\((\Hils_x)_{x\in L^0}\) over~\(L^0\), and a measurable
representation \(\Pi_l\colon \Banb_l \to
\Bound(\Hils_{\s(l)},\Hils_{\rg(l)})\), \(l\in L^1\), of the Fell
bundle, such that \(\Hils \cong \LL^2(L^0,(\Hils_x),\alpha)\) is the
space of \(\LL^2\)\nb-sections of the field~\((\Hils_x)\)
over~\(L^0\), and
\begin{equation}
  \label{eq:integrate_rep_Fell}
  \pi(\psi) \xi(x)
  = \int_{L^x} \Pi_l(\psi(l)) \xi(\s(l))
  \sqrt{\frac{\dd(\alpha\circ\tilde\nu)}{\dd(\alpha\circ\nu)}}(l)
  \,\dd\nu^x(l)
\end{equation}
for all \(\psi\in\Sect(L,\Banb)\), \(\xi\in\Hils\), and
\(\alpha\)\nb-almost all \(x\in L^0\).  Here \(\tilde\nu(l) =
\nu(l^{-1})\), and the quasi-invariance of~\(\alpha\) says that the
measures \(\alpha\circ\tilde\nu\) and~\(\alpha\circ\nu\) on~\(L^1\)
are equivalent, so that the Radon--Nikodym derivative
\(\frac{\dd(\alpha\circ\tilde\nu)}{\dd(\alpha\circ\nu)}\) is defined
\(\alpha\circ\nu\)-almost everywhere.  This theorem is proved in
\cites{Renault:Representations, Muhly-Williams:Renaults_equivalence,
  Muhly-Williams:Equivalence.FellBundles} for many cases, but no proof
for arbitrary Fell bundles over locally Hausdorff groupoids seems to
be published yet.  More precisely, there is an \(L\)\nb-invariant
\(\alpha\)\nb-nullset \(E\subseteq L^0\) such that the
representation~\(\Pi_l\) is defined for all \(l\in L^1\) with
\(\s(l),\rg(l)\notin E\), and has the properties required of a
representation for all such~\(l\); most notably, \(\Pi_{l_1 l_2}(b_1
b_2) = \Pi_{l_1}(b_1) \Pi_{l_2}(b_2)\) for composable \(l_1,l_2\in
L^1\) and \(b_i\in \Banb_{l_i}\) for \(i=1,2\) and \(\Pi_{l^{-1}}(b^*)
= \Pi_l(b)^*\) for \(l\in L^1\), \(b\in\Banb_l\) provided the source
and range objects of \(l,l_1,l_2\) do not belong to~\(E\).  It is
important that the nullsets where things do not work are of this
particular form.

\begin{theorem}
  \label{the:quasi-invariant}
  A quasi-invariant measure on~\(L\) is also quasi-invariant
  for~\(G\), and
  \begin{equation}
    \label{eq:modular_functions_equal}
    \frac{\dd(\alpha\circ\tilde\lambda)}{\dd(\alpha\circ\lambda)}(g)
    \frac{\dd(\alpha\circ\tilde\nu)}{\dd(\alpha\circ\nu)}(l)
    = \frac{\dd(\alpha\circ\tilde\nu)}{\dd(\alpha\circ\nu)}(g l)
  \end{equation}
  holds for almost all \((g,l)\in G^1\times_{\s,G^0,\rg} L^1\) with
  respect to the measure defined by \(\Sect(G^1\times_{\s,G^0,\rg}
  L^1) \ni f\mapsto \iint f(g,l) \,\dd\nu^{\s(g)}(l)
  \,\dd(\alpha\circ\lambda)(g)\).
\end{theorem}

Both Radon--Nikodym derivatives in~\eqref{eq:modular_functions_equal}
are characters almost everywhere.  If they are continuous,
then~\eqref{eq:modular_functions_equal} is equivalent to
\(\frac{\dd(\alpha\circ\tilde\lambda)}{\dd(\alpha\circ\lambda)}(g) =
\frac{\dd(\alpha\circ\tilde\nu)}{\dd(\alpha\circ\nu)}(g)\) for \(g\in
G^1\).  Without continuity, this may be meaningless because \(G^1\)
may be an \(\alpha\circ\nu\)-null set.

\begin{proof}
  Any quasi-invariant measure~\(\alpha\) on~\(L^0\) appears in some
  representation of the trivial Fell bundle~\(\C\) and thus comes from
  some representation of~\(\Cst(L)\): we may take the ``regular
  representation'' associated to~\(\alpha\).
  Lemma~\ref{lem:G_to_L_no_Fell} shows that this representation
  of~\(\Cst(L)\) induces a representation of~\(\Cst(G)\).  The
  representations of \(\Cont_0(G^0)=\Cont_0(L^0)\) associated to the
  representations of \(\Cst(G)\) and~\(\Cst(L)\) coincide.  Hence the
  Disintegration of the representation of~\(\Cst(G)\) gives the same
  measure (class)~\(\alpha\) and the same measurable field of Hilbert
  spaces over~\(G^0\) as for~\(\Cst(L)\).  Thus~\(\alpha\) is also
  quasi-invariant for~\(G\).  The representation of~\(\Sect(G)\) is
  given by an integral formula like~\eqref{eq:integrate_rep_Fell}, but
  with functions on~\(G\) and the modular function
  \(\frac{\dd(\alpha\circ\tilde\lambda)}{\dd(\alpha\circ\lambda)}\).
  Thus
  \[
  \pi_G(\varphi)\pi(\psi) \xi(x)
  = \iint \varphi(g)\psi(l)\xi(\s(l))
  \frac{\dd(\alpha\circ\tilde\lambda)}{\dd(\alpha\circ\lambda)}(g)
  \frac{\dd(\alpha\circ\tilde\nu)}{\dd(\alpha\circ\nu)}(l)
  \,\dd\nu^{\s(g)}(l) \,\dd\lambda^x(g).
  \]
  The construction of the morphism \(\Cst(G)\to\Mult(\Cst(L))\) in
  Lemma~\ref{lem:G_to_L_no_Fell} and one obvious substitution give
  \[
  \pi_G(\varphi)\pi(\psi) \xi(x)
  = \pi(\varphi* \psi) \xi(x)
  = \iint \varphi(g)\psi(l)\xi(\s(l))
  \frac{\dd(\alpha\circ\tilde\nu)}{\dd(\alpha\circ\nu)}(g l)
  \,\dd\nu^{\s(g)}(l) \,\dd\lambda^x(g).
  \]
  Both formulas coincide for \(\alpha\)\nb-almost all \(x\in L^0\) if
  and only if~\eqref{eq:modular_functions_equal} holds.
\end{proof}

The measure~\(\alpha\) on \(G^0=L^0\) is quasi-invariant for~\(G\) by
Theorem~\ref{the:quasi-invariant}.  The representation~\(\Pi_l\)
restricts to a representation of~\(\Banb|_G\).  This works even if
\(G\subseteq L\) is an \(\alpha\circ\nu\)-nullset because of the
special form of the set where~\(\Pi\) is not defined.  The
measure~\(\alpha\) on~\(G^0\), the measurable field of Hilbert
spaces~\((\Hils_x)_{x\in G^0}\), and~\(\Pi|_G\) form a representation
of~\(\Banb|_G\).  It integrates to a \Star{}representation~\(\pi_G\)
of~\(\Sect(G,\Banb|_G)\), given by
\[
\pi_G(\varphi)\xi(x) \defeq
\int_{G^x} \Pi_g(\varphi(g)) \xi(\s(g))
\sqrt{\frac{\dd(\alpha\circ\tilde\nu)}{\dd(\alpha\circ\nu)}}(g)
\,\dd\lambda^x(g)
\]
for all \(\varphi\in\Sect(G,\Banb)\), \(\xi\in\Hils\), and \(x\in
G^0\setminus E\) for the nullset~\(E\) above.

Let \(\varphi\in\Sect(G,\Banb)\), \(\psi\in\Sect(L,\Banb)\),
\(\xi\in\Hils\), and \(x\in G^0\setminus E\).  Then
\begin{align*}
  &\phantom{{}={}} \pi_G(\varphi)\pi(\psi)\xi(x)
  \\&= \int_{G^x} \int_{L^{\s(g)}}
  \Pi_g(\varphi(g)) \Pi_l(\psi(l)) \xi(\s(l))
  \sqrt{\frac{\dd(\alpha\circ\tilde\lambda)}{\dd(\alpha\circ\lambda)}}(g)
  \sqrt{\frac{\dd(\alpha\circ\tilde\nu)}{\dd(\alpha\circ\nu)}}(l)
  \,\dd\nu^{\s(g)}(l) \,\dd\lambda^x(g)
  \\&= \int_{L^x} \int_{G^x}
  \Pi_{l'}(\varphi(g) \psi(g^{-1} l')) \xi(\s(l'))
  \,\dd\lambda^x(g)
  \sqrt{\frac{\dd(\alpha\circ\tilde\nu)}{\dd(\alpha\circ\nu)}}(l')
  \,\dd\nu^x(l)
  = \pi(\phi*\psi)\xi(x).
\end{align*}
We compute as in the proof of Theorem~\ref{the:quasi-invariant}, but
now we know~\eqref{eq:modular_functions_equal} and deduce backwards
that \(\pi_G(\varphi)\pi(\psi) = \pi(\varphi*\psi)\).  Thus~\(\pi_G\)
gives a \Star{}homomorphism \(\Cst(G,\Banb)\to\Mult(\Cst(L,\Banb))\)
as in the proof of Lemma~\ref{lem:G_to_L_no_Fell}.  It is
nondegenerate because the convolution maps
\(\Sect(G,\Banb)\otimes\Sect(L,\Banb)\) to a dense subspace
of~\(\Sect(L,\Banb)\).

\begin{lemma}
  \label{lem:norm_estimate}
  Let \(\psi\in\Sect(L,\Banb)\).  Then
  \(\norm{\pi(\psi)}_{\Bound(\Hils)} \le \norm{\psi}_H^{1/2}
  \norm{\psi^*}_H^{1/2}\) with
  \[
  \norm{\psi}_H \defeq
  \sup_{y\in H^0} \int_{H^y} \norm{\psi|_{L_h}}_{\Cst(L_h)} \,\dd\mu^y(h).
  \]
\end{lemma}

\begin{proof}[Proof of Lemma~\textup{\ref{lem:norm_estimate}}]
  Let \(\psi= \psi_1\cdot \psi_2\) with pointwise multiplication,
  where \(\psi_1\) is a measurable function \(H^1\times_{\rg,H^0,F^0}
  L^0 \to [0,\infty)\), viewed as a function on~\(L^1\) through the
  map~\((F^1,\rg)\), and~\(\psi_2\) is an
  \(\alpha\circ\nu\)-measurable section of~\(\Banb\).  We choose
  \(\psi_1(h,x) \defeq \norm{\psi|_{L_h}}_{\Cst(L_h)}^{1/2}\) and
  \(\psi_2(l) \defeq \psi(l)/\psi_2(F^1(l),\rg(l))\).  The
  function~\(\psi_1\) is upper semicontinuous and hence Borel, and
  hence~\(\psi_2\) is a Borel section.

  We are going to decompose \(\pi(\psi) = T_{\psi_1}^* \circ
  \tilde{T}_{\psi_2}\) for two operators
  \[
  T_{\psi_1}, \tilde{T}_{\psi_2}\colon \LL^2(L^0,(\Hils_x),\alpha)
  \rightrightarrows
  \LL^2(H^1\times_{\rg,H^0,F^0} L^0,\pr_2^*(\Hils_x), \alpha\circ\mu).
  \]
  Here
  \[
  \alpha\circ\mu(f) = \int_{L^0} \int_{H^{F^0(x)}}
  f(h,x) \,\dd\mu^{F^0(x)}(h) \,\dd\alpha(x)
  \]
  and \(\pr_2^*(\Hils_x)\) means the \(\alpha\circ\mu\)-measurable
  field of Hilbert spaces over \(H^1\times_{\rg,H^0,F^0} L^0\) with
  fibre~\(\Hils_x\) at~\((h,x)\).

  The operator~\(T_{\psi_1}\) is defined by
  \[
  (T_{\psi_1} \xi)(h,x) \defeq \psi_1(h,x)\cdot \xi(x)
  \]
  for \(\xi\in \LL^2(L^0,(\Hils_x),\alpha)\) and \((h,x)\in
  H^1\times_{\rg,H^0,F^0} L^0\).  There is a canonical isomorphism
  between \(\LL^2(H^1\times_{\rg,H^0,F^0} L^0,\pr_2^*(\Hils_x),
  \alpha\circ\mu)\) and the tensor product \(\LL^2(H^1,\rg,\mu)
  \otimes_{L^\infty(H^0,F^0_*\alpha)} \LL^2(L^0,(\Hils_x), \alpha)\),
  where we view the measurable field of Hilbert spaces
  \(\LL^2(H^x,\mu^x)\) as a Hilbert module \(\LL^2(H^1,\rg,\mu)\)
  over~\(L^\infty(H^0,F^0_*\alpha)\), which acts on
  \(\LL^2(L^0,(\Hils_x), \alpha)\) by pointwise multiplication
  through~\(F^0\).  With this identification, \(T_{\psi_1}(\xi) =
  \psi_1\otimes \xi\).  This allows to compute \(T_{\psi_1}^*\)
  and~\(T_{\psi_1}^* T_{\psi_1}\).  First,
  \[
  (T_{\psi_1}^* \omega)(x) =
  \int_{H^{F^0(x)}} \conj{\psi_1(h,x)} \omega(h,x)
  \,\dd\mu^{F^0(x)}(h)
  \]
  for \(\omega\in \LL^2(H^1\times_{\rg,H^0,F^0} L^0,\pr_2^*(\Hils_x),
  \alpha\circ\mu)\) and \(x\in L^0\).  Secondly, \(T_{\psi_1}^*
  T_{\psi_1}\) multiplies a section in \(\LL^2(L^0,(\Hils_x),\alpha)\)
  pointwise with the function
  \[
  x\mapsto \int_{H^{F^0(x)}} \abs{\psi_1(h,x)}^2
  \,\dd\mu^{F^0(x)}(h).
  \]
  So
  \[
  \norm{T_{\psi_1}}^2 = \sup_{x\in L^0} \int_{H^{F^0(x)}}
  \abs{\psi_1(h,x)}^2 \,\dd\mu^{F^0(x)}(h).
  \]
  When we choose~\(\psi_1\) as above, we get \(\norm{T_{\psi_1}} =
  \norm{\psi}_H^{1/2}\).

  Let \(\Delta(l) \defeq
  \frac{\dd(\alpha\circ\tilde\nu)}{\dd(\alpha\circ\nu)}(l)\) for \(l\in
  L^1\).  We define
  \[
  \tilde{T}_{\psi_2} \xi(h,x) \defeq
  \int_{L_h^x} \Pi_l(\psi_2(l)) \xi(\s(l)) \Delta(l)
  \,\dd\dot\lambda^{(h,x)}(l)
  \]
  for \(\xi\in \LL^2(L^0,(\Hils_x),\alpha)\) and \((h,x)\in
  H^1\times_{\rg,H^0,F^0} L^0\).  The function~\(\Delta\) is defined
  \(\alpha\nu\)-almost everywhere on~\(L^1\).  Therefore, there is an
  \(\alpha\mu\)\nb-nullset in \(H^1\times_{\rg,H^0,F^0} L^0\) so
  that~\(\Delta(l)\) is defined for \(\dot\lambda^{(h,x)}\)\nb-almost
  all \(l\in L_h^x\) for~\((h,x)\) outside this nullset.  Hence the
  integrand above exists on sufficiently many points to be meaningful.

  We assume that~\(\psi_2\) has quasi-compact support until we have
  proved estimates that allow to extend to more general functions.
  Then the integral defining~\(\tilde{T}_{\psi_2}\xi (h,x)\) is over a
  quasi-compact subset and hence finite, and the resulting function on
  \(H^1\times_{\rg,H^0,F^0} L^0\) has quasi-compact support and hence
  belongs to the Hilbert space \(\LL^2(H^1\times_{\rg,H^0,F^0}
  L^0,\pr_2^*(\Hils_x), \alpha\circ\mu)\).  By definition,
  \begin{align*}
    T_{\psi_1}^* \tilde{T}_{\psi_2}\xi(x)
    &= \int_{H^{F^0(x)}} \int_{L_h^x} \conj{\psi_1(h,x)} \Pi_l(\psi_2(l))
    \xi(\s(l)) \Delta(l) \,\dd\dot\lambda^{(h,x)}(l)
    \,\dd\mu^{F^0(x)}(h)
    \\&= \pi(\psi_1\cdot \psi_2)\xi(x).
  \end{align*}
  Hence
  \begin{equation}
    \label{eq:norm_estimate_reduction}
    \norm{\pi(\psi)} \le \norm{T_{\psi_1}}
    \norm{\tilde{T}_{\psi_2}}.
  \end{equation}
  It remains to estimate \(\norm{\tilde{T}_{\psi_2}}\).  For this
  purpose, we compute \(\tilde{T}_{\psi_2}^*\) and
  \(\tilde{T}_{\psi_2}^* \tilde{T}_{\psi_2}\).

  Let \(\omega\in \LL^2(H^1\times_{\rg,H^0,F^0} L^0,\pr_2^*(\Hils_x),
  \alpha\circ\mu)\) and \(\xi\in \LL^2(L^0,(\Hils_x),\alpha)\) with
  quasi-compact support.  Then
  \begin{align*}
    \bigbraket{\omega}{\tilde{T}_{\psi_2} \xi}
    &= \iiint \bigbraket{\omega(h,x)}{\Pi_l(\psi_2(l)) \xi(\s(l)) \Delta(l)}
    \,\dd\dot\lambda^{(h,x)}(l) \,\dd\mu^{F^0(x)}(h) \,\dd\alpha(x)
    \\&= \int \bigbraket{\omega(F^1(l),\rg(l))}{\Pi_l(\psi_2(l))
      \xi(\s(l))} \Delta(l) \,\dd(\nu\alpha)(l)
    \\&= \int \bigbraket{\omega(F^1(l^{-1}),\s(l))}{\Pi_{l^{-1}}(\psi_2(l^{-1}))
    \xi(\rg(l))} \Delta(l) \,\dd(\nu\alpha)(l)
    \\&= \iint \bigbraket{\Pi_l(\psi_2(l^{-1})^*) \omega(F^1(l^{-1}),\s(l))}
    {\xi(\rg(l))} \Delta(l) \,\dd\nu^x(l) \,\dd\alpha(x).
  \end{align*}
  Here the first step uses the definition of~\(\tilde{T}_{\psi_2}\);
  the second step uses the construction of the Haar system on~\(L\)
  and the definition of the composite measure \(\nu\circ\alpha\)
  on~\(L^1\); the third step uses the substitution \(l\mapsto
  l^{-1}\), which multiplies with the Radon--Nikodym
  derivative~\(\Delta(l)^{-2}\), and uses that~\(\Delta\) is a
  character almost everywhere to simplify \(\Delta(l^{-1}) \Delta(l)^2
  = \Delta(l)\); the last step expands \(\dd(\nu\alpha)(l)=
  \dd\nu^x(l) \,\dd\alpha(x)\) and uses that~\(\Pi\) is a
  \Star{}representation.  Thus \(\braket{\omega}{\tilde{T}_{\psi_2}
    \xi} = \braket{\tilde{T}_{\psi_2}^*\omega}{\xi}\) with
  \[
  \tilde{T}_{\psi_2}^*\omega(x)
  = \int \Pi_l(\psi_2^*(l)) \omega(F^1(l^{-1}),\s(l))
  \Delta(l) \,\dd\nu^x(l).
  \]
  Let \(\Delta_G(g) \defeq
  \frac{\dd(\alpha\circ\tilde\lambda)}{\dd(\alpha\circ\lambda)}(g)\)
  for \(g\in G^1\).  Then
  \begin{align*}
    &\phantom{{}={}}\tilde{T}_{\psi_2}^* \tilde{T}_{\psi_2} \xi(x)
    \\&= \iint \Pi_{l_1}(\psi_2^*(l_1)) \Pi_{l_2}(\psi_2(l_2))
    \xi(\s(l_2)) \Delta(l_1) \Delta(l_2)
    \,\dd\dot\lambda^{(F^1(l_1^{-1}),\s(l_1))}(l_2) \,\dd\nu^x(l_1)
    \\&= \iint
    \Pi_{l_1 l_2}(\psi_2^*(l_1) \psi_2(l_2))
    \xi(\s(l_1 l_2)) \Delta(l_1)\Delta(l_2)
    \,\dd\dot\lambda^{(F^1(l_1^{-1}),\s(l_1))}(l_2) \,\dd\nu^x(l_1)
    \\&= \iint
    \Pi_g(\psi_2^*(l_1) \psi_2(l_1^{-1} g)) \xi(\s(g))
    \Delta(l_1) \Delta(l_1^{-1} g)
    \,\dd\lambda^x(g) \,\dd\nu^x(l_1).
  \end{align*}
  The first step uses the definitions.  The second step uses
  that~\(\Pi\) is multiplicative and \(\s(l_1 l_2) = \s(l_2)\).  The
  third step substitutes \(g \defeq l_1 l_2\) for~\(l_2\), using the
  left \(L\)\nb-invariance of~\(\dot\lambda\).  Since \(F^1(l_2) =
  F^1(l_1^{-1})\), this gives elements of~\(G\).  Since
  \(\Delta(l^{-1}) = \Delta(l)^{-1}\) \(\alpha\circ\nu\)-almost
  everywhere and \(\Delta_G(g^{-1}) = \Delta_G(g)^{-1}\)
  \(\alpha\circ\lambda\)-almost everywhere, we may rewrite
  \(\Delta(l_1) \Delta(l_1^{-1} g) = \Delta(l_1) \Delta(g^{-1}
  l_1)^{-1} = \Delta(l_1) \Delta(l_1)^{-1} \Delta_G(g^{-1})^{-1} =
  \Delta_G(g)\) by~\eqref{eq:modular_functions_equal}; this holds
  almost everywhere with respect to the relevant measure.  Thus
  \(\tilde{T}_{\psi_2}^* \tilde{T}_{\psi_2} = \pi_G(\varphi)\) with
  \begin{align*}
    \varphi(g) &\defeq \int
    \psi_2^*(l) \psi_2(l^{-1} g) \,\dd\nu^{\rg(g)}(l)
    \\&= \int_{H^{F^0(\rg(g))}} \int_{L_h^{\rg(g)}}
    \psi_2^*(l) \psi_2(l^{-1} g)
    \,\dd\dot\lambda^{(h,\rg(g))}(l) \,\dd\mu^{F^0(\rg(g))}(h).
  \end{align*}
  The \(\Cst\)\nb-norm of this element of \(\Sect(G,\Banb)\) is a
  supremum over the \(\Cst\)\nb-norms of its restrictions in
  \(\Sect(G_y,\Banb)\) for \(y\in H^0\).  For \(g\in G_y\),
  \[
  \varphi(g) = \int_{H^y}
  \bigbraket{\psi_2|_{L_{h^{-1}}}}{\psi_2|_{L_{h^{-1}}}}_{\Cst(L_h)}(g)
  \,\dd\mu^y(h).
  \]
  Hence
  \begin{multline*}
    \norm{\varphi}_{\Cst(G,\Banb)} =
    \sup_{y\in H^0} \norm{\varphi|_{G_y}}_{\Cst(G_y,\Banb)}
    = \sup_{y\in H^0} \int_{H^y}
    \norm{\psi_2|_{L_{h^{-1}}}}^2_{\Cst(L_{h^{-1}},\Banb)} \,\dd\mu^y(h)
    \\= \sup_{y\in H^0} \int_{H^y}
    \norm{\psi_2^*|_{L_h}}^2_{\Cst(L_h,\Banb)} \,\dd\mu^y(h).
  \end{multline*}
  The last step uses that the involution is an isometry between
  \(\Cst(L_{h^{\pm1}},\Banb)\).  With the choice of~\(\psi_2\) above,
  this gives
  \[
  \norm{\tilde{T}_{\psi_2}}
  = \norm{\tilde{T}_{\psi_2}^* \tilde{T}_{\psi_2}}^{1/2}
  = \norm{\varphi}_{\Cst(G,\Banb)}^{1/2}
  = \norm{\psi^*}_H^{1/2}.
  \]
  In particular, \(\tilde{T}_{\psi_2}\) is bounded and well defined on
  \(\LL^2(L^0,(\Hils_x),\alpha)\) although the relevant~\(\psi_2\)
  need not have quasi-compact support.  Putting the norm estimates for
  \(T_{\psi_1}\) and~\(\tilde{T}_{\psi_2}\)
  into~\eqref{eq:norm_estimate_reduction} gives the desired estimate
  for~\(\pi(\psi)\).
\end{proof}

Since \(\Sect(L,\Banb) \cong \Sect(H,\FellH)\) is dense
in~\(\Sect(H,\bar{\FellH})\), Lemma~\ref{lem:norm_estimate} implies
that the faithful representation~\(\pi\) of~\(\Cst(L,\Banb)\) extends
uniquely to a representation of~\(\Sect(H,\bar{\FellH})\) that is
bounded for the \(I\)\nb-norm.  Hence we may extend further to the
\(\Cst\)\nb-algebra~\(\Cst(H,\bar{\FellH})\).  The resulting
representation of~\(\Cst(H,\bar{\FellH})\) maps the dense subalgebra
\(\Sect(H,\FellH) = \Sect(L,\Banb)\) into \(\pi(\Cst(L,\Banb))\).
Hence it maps~\(\Cst(H,\bar{\FellH})\) to~\(\pi(\Cst(L,\Banb)) \cong
\Cst(L,\Banb)\) as well.  We already know that~\(\Cst(L,\Banb)\)
maps into~\(\Cst(H,\bar{\FellH})\) via the assignment \(\xi\mapsto \tilde\xi\).
The map backwards is its inverse and shows that
\(\Cst(L,\Banb) \cong \Cst(H,\bar{\FellH})\).  This finishes the proof
of Theorem~\ref{the:sections_full}.

\begin{remark}
  Theorem~\ref{the:sections_full} has a large overlap with
  \cite{Buss-Meyer:Actions_groupoids}*{Theorem~5.5}, which asserts a
  similar crossed product decomposition for actions of inverse
  semigroups on groupoids.  The result
  in~\cite{Buss-Meyer:Actions_groupoids} implies
  Theorem~\ref{the:sections_full} in case the groupoid~\(H\) is étale.
  Here we use Theorem~\ref{the:S-action_groupoid_fibration} to replace
  the fibration by an action of the inverse semigroup of bisections
  of~\(H\), without changing the section algebras of the Fell bundles.
  Any inverse semigroup action can be replaced by an action of a
  topological groupoid,
  compare~\cite{Buss-Exel-Meyer:InverseSemigroupActions}; but these
  groupoids may have a non-Hausdorff object space and thus lie outside
  the scope of Theorem~\ref{the:sections_full}, which only covers
  inverse semigroup actions that come from actions of groupoids with
  Hausdorff, locally compact object spaces.
\end{remark}

\begin{remark}
  \label{rem:fibration_crucial_Ramazan}
  Without the groupoid fibration condition, \(L_h\) need not be an
  equivalence between \(G_{\rg(h)}\) and~\(G_{\s(h)}\) and so the Fell
  bundle over~\(H\) constructed above need not be saturated.  The Fell
  bundle constructed in
  \cite{Deaconu-Kumjian-Ramazan:Fell_groupoid_morphism}*{Theorem 3.4}
  under weaker assumptions is not saturated.  Actually, the
  assumptions that are made
  in~\cite{Deaconu-Kumjian-Ramazan:Fell_groupoid_morphism}*{Theorem
    3.4} are not used in the proof: \emph{any} continuous functor
  \(L\to H\) between étale groupoids yields a (possibly non-saturated)
  Fell bundle over~\(H\) with fibres~\(L_h\) as above.
  Theorem~\ref{the:sections_full} should also carry over to this
  situation, but we did not check details carefully.
\end{remark}

\begin{remark}
  \label{rem:continuous_field}
  Some results about crossed products and Fell bundles are only proved
  if~\(\Banb\) restricts to a continuous field on the unit
  space~\(L^0\).  The Fell bundle~\(\FellH\) over~\(H\) has this
  property if~\(\Banb\) has it and \(F^0\colon L^0\to H^0\) is open.
\end{remark}

\begin{remark}
  \label{rem:act_Prim}
  An action of~\(H\) on a \(\Cst\)\nb-algebra, even in the form of a
  Fell bundle, induces a continuous action of~\(H\) on the primitive
  ideal space of the unit fibre
  (see~\cite{Ionescu-Williams:Remarks_ideal_structure}).  We briefly
  explain this construction.  We first assume~\(H\) to be Hausdorff.

  The \(\Cont_0(H^0)\)-\(\Cst\)-algebra structure on~\(\Cst(G,\Banb)\)
  with fibres~\(\Cst(G_y,\Banb)\) induces a continuous map~\(\psi\)
  from the primitive ideal space \(\Prim \Cst(G,\Banb)\) to~\(H^0\).
  This is the anchor map of the action.  When we pull back
  \(\Cst(G,\Banb)\) along \(\rg,\s\colon H^1\rightrightarrows H^0\) to
  \(\Cont_0(H^1)\)-\(\Cst\)\nb-algebras \(\rg^*(\Cst(G,\Banb))\) and
  \(\s^*(\Cst(G,\Banb))\), these have the primitive ideal spaces
  \(H^1\times_{\rg,H^0,\psi} \Prim \Cst(G,\Banb)\) and
  \(H^1\times_{\s,H^0,\psi} \Prim \Cst(G,\Banb)\), respectively.  The
  space of \(\Cont_0\)\nb-sections of the Fell bundle
  \((\Cst(L_h,\Banb))_{h\in H^1}\) is a \(\Cont_0(H^1)\)-linear
  imprimitivity bimodule between these two pull-backs of
  \(\Cst(G,\Banb)\).  By the Rieffel correspondence, this
  imprimitivity bimodule induces a homeomorphism over~\(H^1\) between
  the primitive ideal spaces
  \[
  H^1\times_{\s,H^0,\psi} \Prim \Cst(G,\Banb) \congto
  H^1\times_{\rg,H^0,\psi} \Prim \Cst(G,\Banb).
  \]
  This homeomorphism is of the form \((h,\mathfrak{p})\mapsto
  (h,h\cdot \mathfrak{p})\) for a continuous action of~\(H\) on
  \(\Prim \Cst(G,\Banb)\).

  If~\(H\) is not Hausdorff, then we should replace~\(H^1\) by a
  disjoint union \(\tilde{H}\defeq \bigsqcup_{U\in\mathcal{U}} U\) for
  a cover by Hausdorff, open subsets.  Then \(\Cont_0(\tilde{H})\)
  makes sense, and we get a map
  \[
  \tilde{H}\times_{\s,H^0,\psi} \Prim \Cst(G,\Banb) \congto
  \tilde{H}\times_{\rg,H^0,\psi} \Prim \Cst(G,\Banb)
  \]
  as above.  Furthermore, the maps coming from
  \(U_1,U_2\in\mathcal{U}\) coincide on the intersection \(U_1\cap
  U_2\), so that we get a well defined, continuous action
  of~\(H\) on \(\Prim \Cst(G,\Banb)\) even if~\(H\) is only locally
  Hausdorff.
\end{remark}

\section{Special cases and applications}
\label{sec:applications}

We now exhibit some special cases of Theorem~\ref{the:sections_full}.
In this section, we tacitly assume all groupoids to be second
countable, locally Hausdorff, locally compact, with Hausdorff unit
space and with a Haar system, and we tacitly assume all
\cstar{}algebras and Fell bundles to be separable.

\subsection{Action on the arrow space of a non-Hausdorff groupoid}
\label{sec:non-Hausdorff_act_arrow}

Let~\(H\) be a groupoid.  To make the problem non-trivial, assume that
the arrow space~\(H^1\) is non-Hausdorff.  We continue the
construction in~§\ref{sec:translation_arrows}.  Let \(H^1 =
\bigcup_{U\in\mathcal{U}} U\) be a cover by Hausdorff, open subsets
and let \(X\defeq \bigsqcup_{U\in\mathcal{U}} U\) with the canonical
map \(p\colon X\onto H^1\).  We have constructed a groupoid fibration
\((\s p)^*(H^0) \to H\) with fibre~\(p^*(H^1)\).
Theorem~\ref{the:sections_full} gives a Fell bundle over~\(H\) with
unit fibre~\(\Cst(p^*(H^1))\) and section \(\Cst\)\nb-algebra
\(\Cst((\s p)^*(H^0))\).  The \(\Cst\)\nb-algebra of the \v{C}ech
groupoid~\(p^*(H^1)\) is the standard way to turn the locally
Hausdorff space~\(H^1\) into a \(\Cst\)\nb-algebra.  It is a Fell
algebra with spectrum~\(H^1\) and trivial Dixmier--Douady invariant
(see~\cite{Huef-Kumjian-Sims:Dixmier-Douady}).  As
in~\cite{Buss-Meyer:Actions_groupoids}, the action of~\(H\) on \(\Prim
\Cst(G)\) is the translation action of~\(H\) on~\(H^1\); this follows
from the description of the action in Remark~\ref{rem:act_Prim}.  All
this justifies viewing the Fell bundle over~\(H\) constructed above as
a good \(\Cst\)\nb-algebraic description of the action of~\(H\)
on~\(H^1\).

Similarly, \(\Cst((\s p)^*(H^0))\) is a Fell algebra with
spectrum~\(H^0\) and trivial Dixmier--Douady invariant.  Since~\(H^0\)
is Hausdorff, it is even a continuous trace \(\Cst\)\nb-algebra, and
Morita--Rieffel equivalent to \(\Cst(H^0)=\Cont_0(H^0)\).  This
generalises the well known Morita--Rieffel equivalence \(\Cst(H\ltimes
H^1) \sim \Cont_0(H^0)\) for a Hausdorff groupoid.

More generally, let~\(X\) be a basic \(H\)\nb-space, that is, \(X\) is
a locally Hausdorff, locally compact space with a basic
\(H\)\nb-action (see \cites{Buss-Meyer:Actions_groupoids,
  Meyer-Zhu:Groupoids}).  Then we may cover~\(X\) by Hausdorff open
subsets,
producing a \v{C}ech groupoid~\(G\) that is equivalent to~\(X\).  The
action of~\(H\) on~\(X\) transfers to an ``action'' on~\(G\) in the
form of a groupoid fibration \(L\to H\) with fibre~\(G\).  Here \(L
\defeq p^*(H\ltimes X)\), where \(p\colon G^0\onto X\) is obtained
from the open covering of~\(X\).  Theorem~\ref{the:sections_full}
gives a Fell
bundle over~\(H\) with unit fibre~\(\Cst(G)\) and section
\(\Cst\)\nb-algebra \(\Cst(p^*(H\ltimes X))\).  This Fell bundle
describes the \(H\)\nb-action on~\(X\) in \(\Cst\)\nb-algebraic terms.

The \(\Cst\)\nb-algebras \(\Cst(G)\) and \(\Cst(p^*(H\ltimes X))\) are
Fell algebras with trivial Dixmier--Douady invariants and spectrum
\(X\) and~\(H\backslash X\), respectively.  For \(\Cst(p^*(H\ltimes
X))\), the proof uses that \(H\ltimes X\) is isomorphic to the \v{C}ech
groupoid of the open surjection \(X\onto H\backslash X\) because the
\(H\)\nb-action on~\(X\) is basic.

The following proposition says that the action of~\(H\) on its arrow
space~\(H^1\) cannot be modelled by a classical action by
automorphisms:

\begin{proposition}
  \label{pro:no_classical_action}
  Let~\(H\) be a non-Hausdorff, locally Hausdorff groupoid.  There is
  no classical action of~\(H\) on a groupoid~\(G\) such that~\(G\) has
  Hausdorff object space, is equivalent to the space~\(H^1\), and has
  an open surjective anchor map \(\rg_H\colon G^0 \onto H^0\), and
  such that the transformation groupoid is equivalent to~\(H^0\).
\end{proposition}

\begin{proof}
  Let \(L\to H\) be a groupoid fibration with fibre~\(G\), such
  that~\(G\) is equivalent to~\(H^1\) and~\(L\) is equivalent
  to~\(H^0\).  Since~\(G\) is equivalent to a space, it is a basic
  groupoid.  So \((\rg,\s)\colon G^1\to G^0\times G^0\) is injective,
  forcing~\(G^1\) to be Hausdorff because~\(G^0\) is Hausdorff.  The
  same argument shows that~\(L^1\) is Hausdorff.

  The arrow space for a classical action of~\(H\) on~\(G\) is
  \(H^1\times_{\s,H^0,\rg_H} G^1\).  We claim that this is
  non-Hausdorff if~\(H^1\) is non-Hausdorff and \(\rg_H\colon G^0\onto
  H^0\) is an open surjection.  Hence the fibration \(L\to H\)
  cannot come from a classical action.

  Since~\(H^1\) is non-Hausdorff, there is a sequence~\((h_n)\)
  in~\(H^1\) with two limits \(h\neq h'\).  Since \(\rg_H(g) =
  \rg_H(\rg(g))\) for \(g\in G^1\) and both \(\rg\colon G^1\onto G^0\)
  and \(\rg_H\colon G^0\onto H^0\) are open surjections, so is
  \(\rg_H\colon G^1\onto H^0\).  Hence we may lift the convergent
  sequence~\(\s(h_n)_{n\in\N}\) in~\(H^0\) to a convergent
  sequence~\((g_n)\) in~\(G^1\) with limit, say, \(g\) (see
  \cite{Williams:crossed-products}*{Proposition~1.15}).
  Then~\((h_n,g_n)\) is a sequence in \(H^1\times_{\s,H^0,\rg_H} G^1\)
  that converges both to \((h,g)\) and~\((h',g)\).
\end{proof}

If the \(H\)\nb-action on~\(G\) models the translation action
on~\(H^1\), then~\(G\) should be equivalent to~\(H^1\), the anchor map
\(G^0\to H^0\) should correspond to the anchor map \(\rg\colon
H^1\onto H^0\) of the translation action and hence be an open
surjection, and the transformation groupoid should be equivalent to
\(H\ltimes H^1 \sim H^0\).  Thus the assumptions in
Proposition~\ref{pro:no_classical_action} should hold for any model of
this action.

Proposition~\ref{pro:no_classical_action} is related to
\cite{Buss-Meyer:Actions_groupoids}*{Theorem~7.1}, which forbids the
existence of a classical action (by isomorphisms) of~\(H\) on a
\cstar{}algebra~\(A\) with \(\Prim(A)\cong H^1\) so that the induced
action on~\(\Prim(A)\) is the translation action on~\(H^1\).  The
results in~\cite{Buss-Meyer:Actions_groupoids} are written only for
étale groupoids.  The idea of the proof may, however, be extended to
general locally compact groupoids.

\subsection{Groupoid fibration from a bibundle}
\label{sec:fib_bibundle}

Consider the groupoid fibration associated to a
\(G,H\)-bibundle~\(X\) in Example~\ref{exa:bitransformation_groupoid}.
We assume~\(X\) to be second countable, Hausdorff and locally compact,
so that the groupoid \(G\ltimes X\rtimes H\) satisfies the standing
assumptions for this section.

The canonical functor \(G\ltimes X\rtimes H\to H\) is a fibration with
fibre \(G\ltimes X\), see Example~\ref{exa:bitransformation_groupoid}.
The groupoid \(G\ltimes X\) inherits a Haar system from~\(G\), and
\(G\ltimes X\rtimes H\) inherits a Haar system from \(G\ltimes X\)
and~\(H\) by Theorem~\ref{the:inherit_Haar_measures}.  This is equal
to the Haar system induced in the most obvious way from the Haar
systems on \(G\) and~\(H\).  Given a (separable, saturated) Fell
bundle~\(\Banb\) over \(G\ltimes X\rtimes H\),
Theorem~\ref{the:sections_full} provides a Fell bundle over~\(H\)
whose section \cstar{}algebra is canonically isomorphic to
\(\Cst(G\ltimes X\rtimes H,\Banb)\) and with unit fibre
\(\Cst(G\ltimes X,\Banb)\).

Brown, Goehle and
Williams~\cite{Brown-Goehle-Williams:Equivalence_iterated} have
recently proved this under the extra assumptions that~\(X\) be an
equivalence bibundle and~\(\Banb\) be the Fell bundle associated to an
ordinary action of \(G\ltimes X\rtimes H\) on some
\cstar{}algebra~\(A\) by automorphisms.  Our analysis works for any
 \(G,H\)\nb-space~\(X\).  If~\(\Banb\) comes from an
action by automorphisms, then the Fell bundle over~\(X\rtimes H\)
constructed above is also associated to an action of~\(H\) on
\(A\rtimes (G\ltimes X) = A\rtimes G\) by automorphisms; this
\(H\)\nb-action is defined as in
\cite{Brown-Goehle-Williams:Equivalence_iterated}*{Proposition~3.7}.
As explained in \cite{Brown-Goehle-Williams:Equivalence_iterated}, the
crossed product decomposition \(A\rtimes (G\ltimes X\rtimes H)\cong
(A\rtimes G)\rtimes H\) is useful to study Brauer semigroups and
symmetric imprimitivity theorems for groupoid actions as
in~\cite{Brown-Goehle:Brauer_semigroup}.

\subsection{Crossed products for groupoid extensions}
\label{sec:crossed_groupoid_extensions}

The statement of Theorem~\ref{the:sections_full} was already proved
for several classes of groupoid extensions \(G\into L\onto H\).
If~\(F^0\) is a homeomorphism as in~§\ref{sec:groupoid_extensions},
then we replace~\(L\) by an isomorphic groupoid so that~\(F^0\) is the
identity map.  We assume this in this subsection.

\subsubsection{Group extensions}
\label{sec:group_extensions}

First assume that we are dealing with a group extension, that is, the
object spaces are a single point.  This is the most classical case of
an iterated crossed product decomposition.  The assertion of
Theorem~\ref{the:sections_full} in this case follows from a similar
statement in \cite{Doran-Fell:Representations_2}*{VIII.6} for
\(\LL^1\)\nb-section algebras of Fell bundles instead of
\cstar{}algebras.  It also follows from
\cite{Buss-Meyer:Crossed_products}*{Theorem~5.2}, which proves a
similar decomposition for Fell bundles over crossed modules.

Even if the Fell bundle~\(\Banb\) over~\(L\) is the most trivial one
that leads to the group \(\Cst\)\nb-algebra~\(\Cst(L)\), the Fell
bundle over \(H=L/G\) is usually not associated to an action of~\(H\)
on~\(\Cst(G)\) by automorphisms.  We need either Fell bundles,
Busby--Smith twisted actions, or Green twisted actions to make sense
of this action.  In the language of Green twisted actions, the
assertion of Theorem~\ref{the:sections_full} for group extensions is
proved in \cites{Green:Local_twisted,Chabert-Echterhoff:Twisted}.

A split extension \(G\into L\onto H\) of locally compact groups
corresponds to a semidirect product decomposition \(L\cong G\rtimes
H\) for an action of~\(H\) on~\(G\) as in
Definition~\ref{def:classical_action}.  In this case, the induced
action of~\(H\) on~\(\Cst(L)\) is by automorphisms, and the assertion
of Theorem~\ref{the:sections_full} is well known and easy to prove
using the universal properties of the group \cstar{}algebras and
crossed products involved.

We also get an analogous result for the transformation (or semidirect
product) groupoid \(L=G\rtimes H\) associated to a classical action of
a groupoid~\(H\) on another groupoid~\(G\) as constructed in
§\ref{sec:classical_action}.  We may even allow arbitrary Fell bundles
over~\(L\) and get a decomposition of the form
\[
\Cst(L,\Banb)\cong \Cst(G,\Banb|_G)\rtimes H.
\]
The special case of this result where~\(H\) is a (locally compact)
group, \(G\) is a locally compact Hausdorff groupoid (with Haar
system) and~\(\Banb\) is the ``semidirect product'' Fell bundle
\(\Banb=\Banb|_G\rtimes H\) associated to an (``invariant'') action
of~\(H\) on~\(\Banb|_G\) appears in
\cite{Kaliszewski-Muhly-Quigg-Williams:Coactions_Fell}*{Theorem~7.1}.
This special case does not cover all Fell bundles over~\(L\), but it
already has many applications: it is used in
\cite{Kaliszewski-Muhly-Quigg-Williams:Coactions_Fell} to prove that
``dual coactions'' on section \cstar{}algebras of (saturated)
Fell bundles over locally compact groups are maximal; and it is also
used in
\cites{Kaliszewski-Muhly-Quigg-Williams:Fell_bundles_and_imprimitivity_theoremsI,
  Kaliszewski-Muhly-Quigg-Williams:Fell_bundles_and_imprimitivity_theoremsII,
  Kaliszewski-Muhly-Quigg-Williams:Fell_bundles_and_imprimitivity_theoremsIII}
to prove imprimitivity theorems for Fell bundles.

\subsubsection{The open isotropy group bundle}
\label{sec:open_isotropy}

Let~\(L\) be an étale groupoid.  Let \(G\defeq\Iso^\circ(L)\) be the
interior of the isotropy subgroupoid of~\(L\) as in
Example~\ref{exa:isotropy-ext}; that is, \(G^0=L^0\) and \(G^1\) is
the interior of the subset \(\{g\in L^1\mid \s(g)=\rg(g)\}\)
of~\(L^1\).  This subset is open and hence also locally compact, and
it gives an étale, normal subgroup bundle in~\(L\).  The quotient
\(H=L/G\) is an étale groupoid with \(H^0\cong
G^0=L^0\) which fits into an extension \(G\into L\onto H\)
(Remark~\ref{rem:H_etale} explains why~\(H\) is étale).

Theorem~\ref{the:sections_full} says, in particular, that \(\Cst(L)\)
is the \cstar{}algebra of a saturated Fell bundle over~\(H\) with
fibres~\(\Cst(L_h)\) and with unit fibre~\(\Cst(G)\).  More generally,
if \(\Banb\) is a Fell bundle over~\(L\), then \(\Cst(L,\Banb)\) is
the \cstar{}algebra of a Fell bundle over~\(H\) with fibres
\(\Cst(L_h,\Banb)\) and with unit fibre~\(\Cst(G,\Banb)\).

A special case of this result is proved in
\cite{Ionescu-Kumjian-Sims-Williams:Stabilization}*{Corollary~3.12},
under several extra assumptions.  Namely,
in~\cite{Ionescu-Kumjian-Sims-Williams:Stabilization} \(L\) is assumed
amenable and Hausdorff, \(G\) is assumed closed and of the special
form \(G^0\times \Gamma\) for some amenable group~\(\Gamma\), and
\(\Banb\) is assumed to be the Fell line bundle associated to a
continuous \(2\)\nb-cocycle on~\(L\) (see also~§\ref{sec:twists} below).
Theorem~\ref{the:sections_full} shows that all these assumptions are
unnecessary.

The crossed product decomposition theorem for the open isotropy group
bundle is particularly useful to describe the ideal structure of a
crossed product.  By Renault's pioneering result
\cite{Renault:Ideal_structure}*{Corollary~4.9}, the ideals of a
groupoid crossed product \(A\rtimes L\) for a Green twisted action of
a groupoid~\(L\) are
equivalent to \(L\)\nb-invariant ideals of the coefficient
\cstar{}algebra~\(A\) provided the induced action of~\(L\)
on~\(\Prim(A)\) is amenable and essentially free; see
\cite{Renault:Ideal_structure} for the precise definitions.  Renault's
theorem applies, in particular, if the groupoid that is acting is
Hausdorff, amenable, and essentially free.  Counterexamples to
Renault's theorem for non-Hausdorff~\(L\) are described
in~\cite{Exel:Non-Hausdorff}.

Renault assumes the coefficient algebra to be a continuous field of
\(\Cst\)\nb-algebras over~\(L^0\).  Presumably this was because the
larger class of \(\Cont_0(L^0)\)-\(\Cst\)-algebras that we use was not
yet so widely used at the time of~\cite{Renault:Ideal_structure}.
Similarly, Renault only studies Green twisted actions because Fell
bundles were not so widely used then.  His result carries over
to saturated Fell bundles for the following reason.

Every Fell bundle over a Hausdorff groupoid~\(L\) is Morita--Rieffel
equivalent to an ordinary action of~\(L\).  This is proved in
\cite{Ionescu-Kumjian-Sims-Williams:Stabilization}*{Corollary~3.9},
and it also follows from
\cite{Buss-Meyer-Zhu:Higher_twisted}*{Theorem~5.3} because Fell
bundles over~\(L\) are equivariantly equivalent to weak actions
of~\(L\) in the sense of~\cite{Buss-Meyer-Zhu:Higher_twisted}; these
are exactly the actions by \cstar{}equivalences that we use here and
in~\cite{Buss-Meyer:Crossed_products}.
Hence Renault's result carries over to (saturated) \emph{continuous}
Fell bundles over Hausdorff groupoids.  Continuity here means that the
underlying field of \cstar{}algebras over~\(L^0\) is continuous.  In
this case, the corresponding ordinary action is also continuous and we
may use Renault's original theorem in~\cite{Renault:Ideal_structure};
this is exactly what is done in
\cite{Ionescu-Kumjian-Sims-Williams:Stabilization}*{Corollary~3.9}.
The continuity assumption is probably not necessary here, but
Hausdorffness is crucial: for non-Hausdorff~\(L\), there are Fell
bundles that are not equivariantly equivalent to ordinary actions, see
\cite{Buss-Meyer:Crossed_products}*{§7}.

The quotient \(H\defeq L/\Iso(L)^\circ\) for an étale groupoid~\(L\)
is the largest quotient that is essentially free.  Thus any action
of~\(H\) is essentially free.  If \(H\) is Hausdorff and amenable,
then the ideal structure of section \(\Cst\)\nb-algebras of Fell
bundles over~\(H\) is completely understood in terms of the action
of~\(H\) on the ideal lattice of the unit fibre.  This is determined
by the action of~\(H\) on the primitive ideal space, which is
described in Remark~\ref{rem:act_Prim}.  The
quotient~\(L/\Iso(L)^\circ\) is Hausdorff if and only
if~\(\Iso(L)^\circ\) is closed in~\(L\) by
Proposition~\ref{pro:condition-Hausdorff}.

Theorem~\ref{the:sections_full} decomposes a crossed product for a
general étale groupoid~\(L\) into one by the étale, normal group
bundle~\(\Iso(L)^\circ\) and one by the essentially free quotient \(H
\defeq L/\Iso(L)^\circ\).  Renault's description of the ideal
structure applies to the crossed product by~\(H\) provided~\(H\) is
Hausdorff and amenable.  This idea is used
\cite{Ionescu-Kumjian-Sims-Williams:Stabilization}*{Theorem~4.5} to
describe the ideal structure of \cstar{}algebras of twisted
higher-rank graphs without singular vertices.  Such algebras are
section \cstar{}algebras of Fell bundles over étale groupoids~\(L\) as
above, and the assumptions imply that the isotropy group
bundle~\(\Iso(L)^\circ\) is of the form \(L^0\times \Gamma\) for some
abelian group~\(\Gamma\).

\subsubsection{Twists over groupoids}
\label{sec:twists}

A twist over a groupoid~\(H\) is another groupoid~\(L\) satisfying our
standing assumptions, which fits into a central extension of the form
\(X\times \Torus\to L\to H\), where \(X=L^0=H^0\), see
Example~\ref{exa:extensions}.  Theorem~\ref{the:inherit_Haar_measures}
gives a canonical Haar system on~\(L\), which allows to form the
groupoid \cstar{}algebra~\(\Cst(L)\).

A \emph{Fell line bundle} over~\(H\) is a Fell bundle~\(\Banb\) where
all fibres~\(\Banb_l\) have dimension~\(1\).  Twists over~\(H\)
correspond to Fell line bundles over~\(H\) by going back and forth
between principal \(\Torus\)\nb-bundles and (complex) Hermitian line
bundles (see~\cite{Deaconu-Kumjian-Ramazan:Fell_groupoid_morphism}).
If~\(L\) is a twist as above, then~\(\Torus\) acts on~\(L\) by
pointwise multiplication, and we get a Hermitian line bundle
\(\Banb\defeq \C\times_\Torus L\); this is the orbit space of
\(\C\times L\) by the diagonal \(\Torus\)\nb-action \(z\cdot
(\lambda,\sigma)=(\lambda\overline{z},z\cdot \sigma)\).  The groupoid
structure of~\(L\) gives~\(\Banb\) the structure of a Fell line bundle
over~\(G\).  Conversely, a Fell line bundle~\(\Banb\) gives a twist
\(L\defeq \{u\in \Banb\mid u^* u=1\}\) or, equivalently, the subspace
of unitary elements of~\(\Banb\).  These two constructions are inverse
to each other.

The twisted groupoid \cstar{}algebra \(\Cst(G,L)\) is, by
definition, the section \cstar{}algebra of the corresponding Fell line
bundle~\(\Banb\).  We are going to identify this with a certain direct
summand in~\(\Cst(L)\).

\begin{corollary}
  Let~\(L\) be a twist over~\(H\) as above and let~\(\Banb\) be the
  associated Fell line bundle.  Then \(\Cst(L)\) is isomorphic to
  the section \cstar{}algebra of a Fell bundle~\(\FellH\) over~\(H\)
  with unit fibre
  \[
  \Cst(H^0,\FellH)\cong \Cst(X\times \Torus)
  \cong \bigoplus_{n\in\Z} \Cont_0(X).
  \]
  The Fell bundle~\(\FellH\) is a direct sum of Fell line bundles,
  where the \(n\)th summand is the \(n\)\nb-fold tensor product
  \(\Banb^{\otimes n} = \Banb\otimes \Banb \otimes \dotsb \otimes
  \Banb\).  Hence
  \[
  \Cst(L)\cong \bigoplus_{n\in \Z} \Cst(\FellH_n)
  \cong \bigoplus_{n\in \Z} \Cst(\Banb^{\otimes n}).
  \]
  If~\(L\) is the extension associated to a Borel \(2\)\nb-cocycle
  \(\omega\colon H^1\times_{\s,\rg} H^1\to \Torus\), then
  \(\Cst(\FellH_n)\cong \Cst(H,\omega^n)\) and \(\Cst(L)\cong
  \bigoplus_{n\in \Z}\Cst(H,\omega^n)\).
\end{corollary}

This is proved in~\cite{Brown-Huef:Decomposing} for twists coming from
\emph{continuous} \(2\)\nb-cocycles over \emph{Hausdorff} groupoids.

\begin{proof}
  Since \(X\times \Torus\into L\) is central, the resulting
  \Star{}homomorphism \(\varphi\colon
  \Cst(\Torus)\to\Mult(\Cst(L))\) takes values in the centre.
  Since \(\Cst(\Torus)\cong \Cont_0(\Z)\), \(\Cst(L)\) is a
  \(\Cst\)\nb-algebra over the discrete space~\(\Z\).  That is,
  \(\Cst(L) \cong \bigoplus_{n\in\Z} \Cst(L_n)\), where
  \(\Cst(L_n)\subseteq \Cst(L)\) is the ideal generated by
  the central projection \(p_n=\varphi(z^n)\).  The issue is to
  identify these summands.

  This is fairly easy for the dense \Star{}subalgebra~\(\Sect(L)\):
  the \(n\)th spectral subspace~\(\Sect(L)_n\) consists of all
  quasi-continuous functions \(L\to\C\) that satisfy \(f(z\cdot
  \sigma) = z^n f(\sigma)\) for all \(z\in\Torus\), \(\sigma\in L\).
  For \(n=1\), this is the space of quasi-continuous sections
  of~\(\Banb\) (see, for instance, \cite{Renault:Cartan.Subalgebras}).
  For \(n\in\Z\), \(\Sect(L)_n\) is the space of quasi-continuous
  sections of~\(\Banb^{\otimes n}\).  The direct sum decomposition
  \(\Sect(L) \cong \bigoplus_{n\in\Z} \Sect(H,\Banb^{\otimes n})\)
  remains the same after taking \(\Cst\)\nb-completions by
  Theorem~\ref{the:sections_full}.  This is easier than the general
  situation considered in Theorem~\ref{the:sections_full} because the
  fibre has the simple form \(X\times\Torus\), but it is still
  non-trivial.

  A Borel \(2\)\nb-cocycle~\(\omega\) gives rise to a twist and thus
  to a Fell line bundle~\(\Banb\) with \(\Cst(H,\omega) =
  \Cst(H,\Banb)\).  The multiplication of cocycles corresponds to the
  tensor product of Fell line bundles.  Hence the claim for Fell line
  bundles contains the statement for Borel \(2\)\nb-cocycles as a
  special case.
\end{proof}

\subsection{Linking groupoids}
\label{sec:linking}

Now we return to the fibration involving the linking groupoid of an
equivalence of groupoids in Example~\ref{exa:linking-groupoid}.
Let \(G\) and~\(H\) be groupoids that satisfy our standing
assumptions, and let~\(X\) be an equivalence \(G,H\)-bibundle.
Let~\(L\) be the linking groupoid of~\(X\) as in
Example~\ref{exa:linking-groupoid}, and let \(F\colon L\to K\) be the
fibration to the finite groupoid with four arrows
\(\gamma,\gamma^{-1},1_{\s(\gamma)}, 1_{\rg(\gamma)}\).  The fibre
of~\(F\) is the disjoint union \(G\sqcup H\) of the groupoids \(G\)
and~\(H\).  The groupoid~\(K\) is discrete and hence étale, so it has
a canonical Haar system.  Theorem~\ref{the:inherit_Haar_measures}
gives a canonical Haar system on the linking groupoid~\(L\) which is
induced by the Haar systems on \(G\) and~\(H\) (see
also~\cite{Sims-Williams:Equivalence_reduced_groupoid}).

A (saturated) Fell bundle over~\(K\) consists of \cstar{}algebras
\(A\) and~\(B\), the fibres over \(\rg(\gamma)\) and~\(\s(\gamma)\),
and an equivalence \(A,B\)-bimodule~\(\Hilm\), namely, the fibre
over~\(\gamma\).  The section \(\Cst\)\nb-algebra of such a Fell
bundle over~\(K\) is simply the direct sum of the fibres with the
canonical \Star{}algebra structure: this is already complete in the
unique \(\Cst\)\nb-norm.  Hence it is the linking \cstar{}algebra
of~\(\Hilm\).  In particular, the Fell bundle
\(\bigl(\Cst(L_h)\bigr)_{h\in K}\) over~\(K\) constructed
in~§\ref{sec:crossed} has the fibres \(\Cst(G)\) and~\(\Cst(H)\) at
the unit arrows and \(\Cst(X)\) and~\(\Cst(X^*)\) at the non-identity
arrows in~\(K\).

Thus the assertion of Theorem~\ref{the:sections_full} in this case is
that the \(\Cst\)\nb-algebra of the linking groupoid of an equivalence
is the linking \(\Cst\)\nb-algebra of the induced equivalence of
groupoid \(\Cst\)\nb-algebras.  This result and its analogue for
reduced norms are known for Hausdorff groupoids, see
\cite{Sims-Williams:Equivalence_reduced_groupoid}*{Remark~2.4} and
\cite{Kumjian:Fell_bundles}*{Example~3.5(ii)}.

More generally, Theorem~\ref{the:sections_full} yields a similar
result with Fell bundle coefficients.  Let \(\A\) and~\(\Banb\) be
Fell bundles over \(G\) and~\(H\), respectively, and let~\(\Hilm\) be
an \emph{equivalence} over~\(X\) between them.  That is, \(\Hilm\) is
an upper-semicontinuous Banach bundle over~\(X\) with commuting
actions of~\(\A\) on the left and~\(\Banb\) on the right and ``inner
products'' satisfying the usual algebraic conditions
(see~\cite{Muhly-Williams:Equivalence.FellBundles}).  We may combine
this data into a Fell bundle~\(L(\Hilm)\) over~\(L\) whose
restrictions to \(G\), \(H\) and~\(X\) are the original bundles
\(\A\), \(\Banb\) and~\(\Hilm\), respectively; this is the
\emph{linking Fell bundle} as constructed
in~\cite{Sims-Williams:Equivalence_reduced}.
Theorem~\ref{the:sections_full} says that the linking \cstar{}algebra
of the equivalence \(\Cst(G,\A),\Cst(H,\Banb)\)-bimodule
\(\Cst(\Hilm)\) associated to the \(\A,\Banb\)-equivalence~\(\Hilm\)
is the section \cstar{}algebra of the linking Fell
bundle~\(L(\Hilm)\).  In brief, \(\Cst(L(\Hilm))\cong
L(\Cst(\Hilm))\).  For Hausdorff groupoids, this result is proved
in~\cite{Sims-Williams:Equivalence_reduced} together with the
analogous statement for reduced groupoid \(\Cst\)\nb-algebras.

\subsection{Strongly surjective cocycles}
\label{sec:cocycle_strongly_surjective}

Let~\(L\) be a groupoid and~\(H\) a discrete group.  A functor
\(F\colon L\to H\) is also called a \emph{cocycle}.  It is called
\emph{strongly surjective} if the map \((F,\rg)\colon L^1 \to H\times
L^0\) is surjective.  Equivalently, the map~\((F,\s)\) is surjective.
This map is automatically open by Remark~\ref{rem:openness_redundant}.
Hence a groupoid fibration to a discrete group~\(H\) is the same as
a strongly surjective cocycle \(L\to H\).  We interpret such a cocycle
as an \(H\)\nb-action on the fibre \(G\defeq F^{-1}(\{1\})\), which
may also be called the kernel of the cocycle.
Theorem~\ref{the:sections_full} describes \(\Cst(L)\) as the section
\(\Cst\)\nb-algebra of a Fell bundle over~\(H\) with unit
fibre~\(\Cst(G)\).

The case \(H=\Z\) is particularly simple.  The section
\(\Cst\)\nb-algebra of a \emph{saturated} Fell
bundle~\((\Banb_n)_{n\in\Z}\) over~\(\Z\) is the same as the
Cuntz--Pimsner algebra of the Hilbert bimodule~\(\Banb_1\) over the
\(\Cst\)\nb-algebra~\(\Banb_0\) because both have the same universal
property.  Thus Theorem~\ref{the:sections_full} specialises to the
main result of~\cite{Rennie-Robertson-Sims:Groupoid_CP} in this case.
The theorem in~\cite{Rennie-Robertson-Sims:Groupoid_CP} also covers
cocycles that are only ``unperforated,'' which is weaker than being
strongly surjective.  In this case, the construction of the Fell
bundle over~\(\Z\) still goes through, but it is no longer saturated.
The unperforation assumption ensures that the Fell bundle is
``generated'' by~\(\Banb_1\), which is good enough to identify its
section \(\Cst\)\nb-algebra with the Cuntz--Pimsner algebra
of~\(\Banb_1\).

If~\(H\) is a locally compact group such as~\(\R\), then the correct
generalisation of a strongly surjective cocycle \(L\to H\) is a
groupoid fibration, that is, a functor \(F^1\colon L\to H\) such that
\((F^1,\s)\) or, equivalently, \((F^1,\rg)\) is surjective and
\emph{open}.

The map~\(F^0\) to the one-point space~\(H^0\) is automatically an
open surjection.  Hence \(F^1\colon L^1\to H^1\) is an open surjection
for any groupoid fibration to a topological group by
Lemma~\ref{lem:F1_F0_open_surjective}.  But~\(F\) need not be a
groupoid fibration if~\(F^1\) is just an open surjection, see
Example~\ref{exa:surjective_not_enough}.


\begin{bibdiv}
  \begin{biblist}
\bib{Renault_AnantharamanDelaroche:Amenable_groupoids}{book}{
  author={Anantharaman-Delaroche, Claire},
  author={Renault, Jean},
  title={Amenable groupoids},
  series={Monographies de L'Enseignement Math\'ematique},
  volume={36},
  publisher={L'Enseignement Math\'ematique, Geneva},
  year={2000},
  pages={196},
  isbn={2-940264-01-5},
  review={\MRref {1799683}{2001m:22005}},
}

\bib{Brown-Goehle:Brauer_semigroup}{article}{
  author={Brown, Jonathan Henry},
  author={Goehle, Geoff},
  title={The Brauer semigroup of a groupoid and a symmetric imprimitivity theorem},
  journal={Trans. Amer. Math. Soc.},
  volume={366},
  date={2014},
  number={4},
  pages={1943--1972},
  issn={0002-9947},
  review={\MRref {3152718}{}},
  doi={10.1090/S0002-9947-2013-05953-3},
}

\bib{Brown-Goehle-Williams:Equivalence_iterated}{article}{
  author={Brown, Jonathan Henry},
  author={Goehle, Geoff},
  author={Williams, Dana P.},
  title={Groupoid equivalence and the associated iterated crossed product},
  journal={Houston J. Math.},
  volume={41},
  date={2015},
  number={1},
  pages={153--175},
  issn={0362-1588},
  review={\MRref {3347942}{}},
  eprint={http://www.math.uh.edu/~hjm/restricted/pdf41(1)/09brown.pdf},
}

\bib{Brown-Huef:Decomposing}{article}{
  author={Brown, Jonathan Henry},
  author={an Huef, Astrid},
  title={Decomposing the $C^*$\nobreakdash -algebras of groupoid extensions},
  journal={Proc. Amer. Math. Soc.},
  volume={142},
  date={2014},
  number={4},
  pages={1261--1274},
  issn={0002-9939},
  review={\MRref {3162248}{}},
  doi={10.1090/S0002-9939-2014-11492-4},
}

\bib{Brown:Extensions}{article}{
  author={Brown, Lawrence G.},
  title={Extensions of topological groups},
  journal={Pacific J. Math.},
  volume={39},
  date={1971},
  pages={71--78},
  issn={0030-8730},
  review={\MRref {0307264}{46\,\#6384}},
  eprint={http://projecteuclid.org/euclid.pjm/1102969770},
}

\bib{Brown:Fibrations_groupoids}{article}{
  author={Brown, Ronald},
  title={Fibrations of groupoids},
  journal={J. Algebra},
  volume={15},
  date={1970},
  pages={103--132},
  issn={0021-8693},
  review={\MRref {0271194}{42 \#6077}},
  doi={10.1016/0021-8693(70)90089-X},
}

\bib{Brown-Danesh-Hardy:Topological_groupoidsII}{article}{
  author={Brown, Ronald},
  author={Danesh-Naruie, G.},
  author={Hardy, J. P. L.},
  title={Topological groupoids. II. Covering morphisms and $G$\nobreakdash -spaces},
  journal={Math. Nachr.},
  volume={74},
  date={1976},
  pages={143--156},
  issn={0025-584X},
  review={\MRref {0442937}{56 \#1312}},
  doi={10.1002/mana.3210740110},
}

\bib{BussExel:Fell.Bundle.and.Twisted.Groupoids}{article}{
  author={Buss, Alcides},
  author={Exel, Ruy},
  title={Fell bundles over inverse semigroups and twisted \'etale groupoids},
  journal={J. Operator Theory},
  volume={67},
  date={2012},
  number={1},
  pages={153--205},
  issn={0379-4024},
  review={\MRref {2881538}{}},
  eprint={http://www.theta.ro/jot/archive/2012-067-001/2012-067-001-007.html},
}

\bib{Buss-Exel-Meyer:InverseSemigroupActions}{article}{
  author={Buss, Alcides},
  author={Exel, Ruy},
  author={Meyer, Ralf},
  title={Inverse semigroup actions as groupoid actions},
  journal={Semigroup Forum},
  volume={85},
  date={2012},
  number={2},
  pages={227--243},
  issn={0037-1912},
  doi={10.1007/s00233-012-9418-y},
  review={\MRref {2969047}{}},
}

\bib{Buss-Exel-Meyer:Reduced}{article}{
  author={Buss, Alcides},
  author={Exel, Ruy},
  author={Meyer, Ralf},
  title={Reduced \(C^*\)\nobreakdash -algebras of Fell bundles over inverse semigroups},
  note={\arxiv {1512.05570}},
  date={2015},
}

\bib{Buss-Meyer:Crossed_products}{article}{
  author={Buss, Alcides},
  author={Meyer, Ralf},
  title={Crossed products for actions of crossed modules on \(\textup C^*\)\nobreakdash -algebras},
  status={accepted},
  journal={J. Noncommut. Geom.},
  issn={1661-6952},
  note={\arxiv {1304.6540}},
  date={2016},
}

\bib{Buss-Meyer:Actions_groupoids}{article}{
  author={Buss, Alcides},
  author={Meyer, Ralf},
  title={Inverse semigroup actions on groupoids},
  status={accepted},
  note={\arxiv {1410.2051}},
  date={2016},
  issn={0035-7596},
  journal={Rocky Mountain J. Math.},
}

\bib{Buss-Meyer-Zhu:Higher_twisted}{article}{
  author={Buss, Alcides},
  author={Meyer, Ralf},
  author={Zhu, {Ch}enchang},
  title={A higher category approach to twisted actions on \(\textup C^*\)\nobreakdash -algebras},
  journal={Proc. Edinb. Math. Soc. (2)},
  date={2013},
  volume={56},
  number={2},
  pages={387--426},
  issn={0013-0915},
  doi={10.1017/S0013091512000259},
  review={\MRref {3056650}{}},
}

\bib{Chabert-Echterhoff:Twisted}{article}{
  author={Chabert, J\'er\^ome},
  author={Echterhoff, Siegfried},
  title={Twisted equivariant $KK$\nobreakdash -theory and the Baum--Connes conjecture for group extensions},
  journal={$K$\nobreakdash -Theory},
  volume={23},
  date={2001},
  number={2},
  pages={157--200},
  issn={0920-3036},
  review={\MRref {1857079}{2002m:19003}},
  doi={10.1023/A:1017916521415},
}

\bib{Deaconu-Kumjian-Ramazan:Fell_groupoid_morphism}{article}{
  author={Deaconu, Valentin},
  author={Kumjian, Alex},
  author={Ramazan, Birant},
  title={Fell bundles associated to groupoid morphisms},
  journal={Math. Scand.},
  volume={102},
  date={2008},
  number={2},
  pages={305--319},
  issn={0025-5521},
  review={\MRref {2437192}{2010f:46086}},
  eprint={http://www.mscand.dk/article/view/15064},
}

\bib{Doran-Fell:Representations_2}{book}{
  author={Doran, Robert S.},
  author={Fell, James M. G.},
  title={Representations of $^*$\nobreakdash -algebras, locally compact groups, and Banach $^*$\nobreakdash -algebraic bundles. Vol. 2},
  series={Pure and Applied Mathematics},
  volume={126},
  publisher={Academic Press Inc.},
  place={Boston, MA},
  date={1988},
  pages={i--viii and 747--1486},
  isbn={0-12-252722-4},
  review={\MRref {936629}{90c:46002}},
}

\bib{Exel:Inverse_combinatorial}{article}{
  author={Exel, Ruy},
  title={Inverse semigroups and combinatorial $C^*$\nobreakdash -algebras},
  journal={Bull. Braz. Math. Soc. (N.S.)},
  volume={39},
  date={2008},
  number={2},
  pages={191--313},
  issn={1678-7544},
  review={\MRref {2419901}{2009b:46115}},
  doi={10.1007/s00574-008-0080-7},
}

\bib{Exel:Non-Hausdorff}{article}{
  author={Exel, Ruy},
  title={Non-Hausdorff \'etale groupoids},
  journal={Proc. Amer. Math. Soc.},
  volume={139},
  date={2011},
  number={3},
  pages={897--907},
  issn={0002-9939},
  review={\MRref {2745642}{2012b:46148}},
  doi={10.1090/S0002-9939-2010-10477-X},
}

\bib{Green:Local_twisted}{article}{
  author={Green, Philip},
  title={The local structure of twisted covariance algebras},
  journal={Acta Math.},
  volume={140},
  date={1978},
  number={3-4},
  pages={191--250},
  issn={0001-5962},
  review={\MRref {0493349}{58\,\#12376}},
  doi={10.1007/BF02392308},
}

\bib{Higgins-Machenzie:Fibrations}{article}{
  author={Higgins, Philip J.},
  author={Mackenzie, Kirill C. H.},
  title={Fibrations and quotients of differentiable groupoids},
  journal={J. London Math. Soc. (2)},
  volume={42},
  year={1990},
  number={1},
  pages={101--110},
  issn={0024-6107},
  review={\MRref {1078178}{91k:20062}},
  doi={10.1112/jlms/s2-42.1.101},
}

\bib{Huef-Kumjian-Sims:Dixmier-Douady}{article}{
  author={an Huef, Astrid},
  author={Kumjian, Alex},
  author={Sims, Aidan},
  title={A Dixmier--Douady theorem for Fell algebras},
  journal={J. Funct. Anal.},
  volume={260},
  date={2011},
  number={5},
  pages={1543--1581},
  issn={0022-1236},
  review={\MRref {2749438}{2012i:46066}},
  doi={10.1016/j.jfa.2010.11.011},
}

\bib{Ionescu-Kumjian-Sims-Williams:Stabilization}{article}{
  author={Ionescu, Marius},
  author={Kumjian, Alex},
  author={Sims, Aidan},
  author={Williams, Dana P.},
  title={A stabilization theorem for Fell bundles over groupoids},
  status={preprint},
  date={2015},
  note={\arxiv {1512.06046}},
}

\bib{Ionescu-Williams:Remarks_ideal_structure}{article}{
  author={Ionescu, Marius},
  author={Williams, Dana P.},
  title={Remarks on the ideal structure of Fell bundle $C^*$\nobreakdash -algebras},
  journal={Houston J. Math.},
  volume={38},
  date={2012},
  number={4},
  pages={1241--1260},
  issn={0362-1588},
  review={\MRref {3019033}{}},
  eprint={http://www.math.uh.edu/~hjm/restricted/pdf38(4)/13ionescu.pdf},
}

\bib{Kaliszewski-Muhly-Quigg-Williams:Coactions_Fell}{article}{
  author={Kaliszewski, Steven P.},
  author={Muhly, Paul S.},
  author={Quigg, John},
  author={Williams, Dana P.},
  title={Coactions and Fell bundles},
  journal={New York J. Math.},
  volume={16},
  date={2010},
  pages={315--359},
  issn={1076-9803},
  review={\MRref {2740580}{2012d:46165}},
  eprint={http://nyjm.albany.edu/j/2010/16-13.html},
}

\bib{Kaliszewski-Muhly-Quigg-Williams:Fell_bundles_and_imprimitivity_theoremsI}{article}{
  author={Kaliszewski, Steven P.},
  author={Muhly, Paul S.},
  author={Quigg, John},
  author={Williams, Dana P.},
  title={Fell bundles and imprimitivity theorems},
  journal={M\"unster J. Math.},
  volume={6},
  date={2013},
  pages={53--83},
  issn={1867-5778},
  eprint={http://nbn-resolving.de/urn:nbn:de:hbz:6-25319580780},
  review={\MRref {3148208}{}},
}

\bib{Kaliszewski-Muhly-Quigg-Williams:Fell_bundles_and_imprimitivity_theoremsIII}{article}{
  author={Kaliszewski, Steven P.},
  author={Muhly, Paul S.},
  author={Quigg, John},
  author={Williams, Dana P.},
  title={Fell bundles and imprimitivity theorems: Mansfield's and Fell's theorems},
  journal={J. Aust. Math. Soc.},
  volume={95},
  date={2013},
  number={1},
  pages={68--75},
  issn={1446-7887},
  review={\MRref {3123744}{}},
  doi={10.1017/S1446788713000153},
}

\bib{Kaliszewski-Muhly-Quigg-Williams:Fell_bundles_and_imprimitivity_theoremsII}{article}{
  author={Kaliszewski, Steven P.},
  author={Muhly, Paul S.},
  author={Quigg, John},
  author={Williams, Dana P.},
  title={Fell bundles and imprimitivity theorems: towards a universal generalized fixed point algebra},
  journal={Indiana Univ. Math. J.},
  volume={62},
  date={2013},
  number={6},
  pages={1691--1716},
  issn={0022-2518},
  review={\MRref {3205528}{}},
  doi={10.1512/iumj.2013.62.5107},
}

\bib{Kumjian:Fell_bundles}{article}{
  author={Kumjian, Alex},
  title={Fell bundles over groupoids},
  journal={Proc. Amer. Math. Soc.},
  volume={126},
  date={1998},
  number={4},
  pages={1115--1125},
  issn={0002-9939},
  review={\MRref {1443836}{98i:46055}},
  doi={10.1090/S0002-9939-98-04240-3},
}

\bib{Tu-Xu-Laurent-Gengoux:Twisted_K}{article}{
  author={Laurent-Gengoux, Camille},
  author={Tu, Jean-Louis},
  author={Xu, Ping},
  title={Twisted $K$\nobreakdash -theory of differentiable stacks},
  journal={Ann. Sci. \'Ecole Norm. Sup. (4)},
  volume={37},
  date={2004},
  number={6},
  pages={841--910},
  issn={0012-9593},
  review={\MRref {2119241}{2005k:58037}},
  doi={10.1016/j.ansens.2004.10.002},
}

\bib{LiDu:Thesis}{thesis}{
  author={Li, Du},
  title={Higher groupoid actions, bibundles, and differentiation},
  institution={Georg-August-Universit\"at G\"ottingen},
  type={phdthesis},
  date={2014},
  eprint={http://hdl.handle.net/11858/00-1735-0000-0022-5F4F-A},
}

\bib{Mackenzie:General_Lie_groupoid_algebroid}{book}{
  author={Mackenzie, Kirill C. H.},
  title={General theory of Lie groupoids and Lie algebroids},
  series={London Mathematical Society Lecture Note Series},
  volume={213},
  publisher={Cambridge University Press, Cambridge},
  date={2005},
  pages={xxxviii+501},
  isbn={978-0-521-49928-3},
  isbn={0-521-49928-3},
  review={\MRref {2157566}{2006k:58035}},
  doi={10.1017/CBO9781107325883},
}

\bib{Meyer-Zhu:Groupoids}{article}{
  author={Meyer, Ralf},
  author={Zhu, {Ch}enchang},
  title={Groupoids in categories with pretopology},
  journal={Theory Appl. Categ.},
  volume={30},
  date={2015},
  pages={1906--1998},
  issn={1201-561X},
  eprint={http://www.tac.mta.ca/tac/volumes/30/55/30-55abs.html},
}

\bib{Muhly-Williams:Continuous-traceII}{article}{
  author={Muhly, Paul S.},
  author={Williams, Dana P.},
  title={Continuous trace groupoid $C^*$\nobreakdash -algebras. II},
  journal={Math. Scand.},
  volume={70},
  date={1992},
  number={1},
  pages={127--145},
  issn={0025-5521},
  review={\MRref {1174207}{93i:46117}},
  eprint={http://www.mscand.dk/article/view/12390},
}

\bib{Muhly-Williams:Equivalence.FellBundles}{article}{
  author={Muhly, Paul S.},
  author={Williams, Dana P.},
  title={Equivalence and disintegration theorems for Fell bundles and their \(C^*\)\nobreakdash -algebras},
  journal={Dissertationes Math. (Rozprawy Mat.)},
  volume={456},
  date={2008},
  pages={1--57},
  issn={0012-3862},
  review={\MRref {2446021}{2010b:46146}},
  doi={10.4064/dm456-0-1},
}

\bib{Muhly-Williams:Renaults_equivalence}{book}{
  author={Muhly, Paul S.},
  author={Williams, Dana P.},
  title={Renault's equivalence theorem for groupoid crossed products},
  series={NYJM Monographs},
  volume={3},
  publisher={State University of New York University at Albany},
  place={Albany, NY},
  date={2008},
  pages={87},
  review={\MRref {2547343}{2010h:46112}},
  eprint={http://nyjm.albany.edu/m/2008/3.htm},
}

\bib{Renault:Groupoid_Cstar}{book}{
  author={Renault, Jean},
  title={A groupoid approach to $\textup C^*$\nobreakdash -algebras},
  series={Lecture Notes in Mathematics},
  volume={793},
  publisher={Springer},
  place={Berlin},
  date={1980},
  pages={ii+160},
  isbn={3-540-09977-8},
  review={\MRref {584266}{82h:46075}},
  doi={10.1007/BFb0091072},
}

\bib{Renault:Representations}{article}{
  author={Renault, Jean},
  title={Repr\'esentation des produits crois\'es d'alg\`ebres de groupo\"\i des},
  journal={J. Operator Theory},
  volume={18},
  date={1987},
  number={1},
  pages={67--97},
  issn={0379-4024},
  review={\MRref {912813}{89g:46108}},
  eprint={http://www.theta.ro/jot/archive/1987-018-001/1987-018-001-005.html},
}

\bib{Renault:Ideal_structure}{article}{
  author={Renault, Jean},
  title={The ideal structure of groupoid crossed product $C^*$\nobreakdash -algebras},
  note={With an appendix by Georges Skandalis},
  journal={J. Operator Theory},
  volume={25},
  date={1991},
  number={1},
  pages={3--36},
  issn={0379-4024},
  review={\MRref {1191252}{94g:46074}},
  eprint={http://www.theta.ro/jot/archive/1991-025-001/1991-025-001-001.html},
}

\bib{Renault:Transverse_dynamical}{article}{
  author={Renault, Jean},
  title={Transverse properties of dynamical systems},
  conference={ title={Representation theory, dynamical systems, and asymptotic combinatorics}, },
  book={ series={Amer. Math. Soc. Transl. Ser. 2}, volume={217}, publisher={Amer. Math. Soc.}, place={Providence, RI}, },
  date={2006},
  pages={185--199},
  review={\MRref {2276108}{2007k:46130}},
}

\bib{Renault:Cartan.Subalgebras}{article}{
  author={Renault, Jean},
  title={Cartan subalgebras in $C^*$\nobreakdash -algebras},
  journal={Irish Math. Soc. Bull.},
  number={61},
  date={2008},
  pages={29--63},
  issn={0791-5578},
  review={\MRref {2460017}{2009k:46135}},
  eprint={http://www.maths.tcd.ie/pub/ims/bull61/S6101.pdf},
}

\bib{Renault:Topological_amenability_is_a_Borel_property}{article}{
  author={Renault, Jean},
  title={Topological amenability is a Borel property},
  journal={Math. Scand.},
  volume={117},
  date={2015},
  number={1},
  pages={5--30},
  issn={0025-5521},
  review={\MR {3403785}},
  eprint={http://www.mscand.dk/article/view/22235},
}

\bib{Rennie-Robertson-Sims:Groupoid_CP}{article}{
  author={Rennie, Adam},
  author={Robertson, David},
  author={Sims, Aidan},
  title={Groupoid algebras as Cuntz--Pimsner algebras},
  journal={Math. Scand.},
  status={accepted},
  note={\arxiv {1402.7126}},
  date={2014},
}

\bib{Resende:Etale_groupoids}{article}{
  author={Resende, Pedro},
  title={\'Etale groupoids and their quantales},
  journal={Adv. Math.},
  volume={208},
  date={2007},
  number={1},
  pages={147--209},
  issn={0001-8708},
  review={\MRref {2304314}{}},
  doi={10.1016/j.aim.2006.02.004},
}

\bib{Sims-Williams:Equivalence_reduced_groupoid}{article}{
  author={Sims, Aidan},
  author={Williams, Dana P.},
  title={Renault's equivalence theorem for reduced groupoid $C^*$\nobreakdash -algebras},
  journal={J. Operator Theory},
  volume={68},
  date={2012},
  number={1},
  pages={223--239},
  issn={0379-4024},
  review={\MRref {2966043}{}},
  eprint={http://www.theta.ro/jot/archive/2012-068-001/2012-068-001-012.html},
}

\bib{Sims-Williams:Equivalence_reduced}{article}{
  author={Sims, Aidan},
  author={Williams, Dana P.},
  title={An equivalence theorem for reduced Fell bundle $C^*$\nobreakdash -algebras},
  journal={New York J. Math.},
  volume={19},
  date={2013},
  pages={159--178},
  issn={1076-9803},
  review={\MRref {3084702}{}},
  eprint={http://nyjm.albany.edu/j/2013/19-11.html},
}

\bib{Sims-Williams:primitive_ideals}{article}{
  author={Sims, Aidan},
  author={Williams, Dana P.},
  title={The primitive ideals of some étale groupoid C*-algebras},
  date={2015},
  status={preprint},
  note={\arxiv {1501.02302}},
}

\bib{Williams:crossed-products}{book}{
  author={Williams, Dana P.},
  title={Crossed products of $C^*$\nobreakdash -algebras},
  series={Mathematical Surveys and Monographs},
  volume={134},
  publisher={Amer. Math. Soc.},
  place={Providence, RI},
  date={2007},
  pages={xvi+528},
  isbn={978-0-8218-4242-3; 0-8218-4242-0},
  review={\MRref {2288954}{2007m:46003}},
  doi={10.1090/surv/134},
}
  \end{biblist}
\end{bibdiv}
\end{document}